\documentclass{svjour3}       
\journalname{}
\usepackage{biblatex}
\smartqed  
\usepackage{graphicx}
\usepackage{lipsum}
\usepackage{amsfonts,amsmath,amssymb}
\usepackage{graphicx}
\usepackage{epstopdf}
\usepackage{algorithm}
\usepackage{algorithmic}

\usepackage{commath,amsopn}
\usepackage{multirow}
\usepackage{hyperref}
\hypersetup{colorlinks=true,linkcolor=blue,citecolor=blue}
\usepackage{stmaryrd}
\usepackage{xifthen}
\usepackage{booktabs}
\usepackage{graphicx}
\usepackage{subfigure}
\usepackage{mathtools}
\usepackage{epstopdf}
\usepackage{arydshln}
\usepackage{tablefootnote}
\usepackage{extarrows}
\usepackage{colortbl}
\usepackage{adjustbox}
\usepackage{bbding} 

\newcommand{\vecx}{\mathbf{x}}
\newcommand{\vecr}{\mathbf{r}}
\newcommand{\vecu}{\mathbf{u}}
\newcommand{\vecv}{\mathbf{v}}
\newcommand{\vecy}{\mathbf{y}}
\newcommand{\vecz}{\mathbf{z}}

\newcommand{\mata}{\mathbf{A}}
\newcommand{\matb}{\mathbf{B}}
\newcommand{\matc}{\mathbf{C}}
\newcommand{\matd}{\mathbf{D}}

\newcommand{\matf}{\mathbf{F}}
\newcommand{\matx}{\mathbf{X}}
\newcommand{\matu}{\mathbf{U}}
\newcommand{\matv}{\mathbf{V}}

\newcommand{\matA}{\mathbf{A}}
\newcommand{\matB}{\mathbf{B}}

\newcommand{\matE}{\mathbf{E}}
\newcommand{\matF}{\mathbf{F}}
\newcommand{\matG}{\mathbf{G}}
\newcommand{\matI}{\mathbf{I}}

\newcommand{\matJ}{\mathbf{J}}
\newcommand{\matK}{\mathbf{K}}

\newcommand{\matP}{\mathbf{P}}
\newcommand{\matR}{\mathbf{R}}
\newcommand{\matM}{\mathbf{M}}

\newcommand{\matQ}{\mathbf{Q}}

\newcommand{\matX}{\mathbf{X}}
\newcommand{\matU}{\mathbf{U}}

\newcommand{\tensA}{\mathcal{A}}
\newcommand{\tensB}{\mathcal{B}}
\newcommand{\tensC}{\mathcal{C}}
\newcommand{\tensE}{\mathcal{E}}

\newcommand{\tensG}{\mathcal{G}}

\newcommand{\tensM}{\mathcal{M}}

\newcommand{\tensP}{\mathcal{P}}
\newcommand{\tensS}{\mathcal{S}}
\newcommand{\tensT}{\mathcal{T}}
\newcommand{\tensU}{\mathcal{U}}
\newcommand{\tensV}{\mathcal{V}}

\newcommand{\tensY}{\mathcal{Y}}
\newcommand{\tensX}{\mathcal{X}}

\newcommand{\rank}{\mathrm{rank}}

\newcommand{\ranktc}{\mathrm{rank}_{\mathrm{tc}}}
\newcommand{\ranktt}{\mathrm{rank}_{\mathrm{TT}}}

\newcommand{\St}{\mathrm{St}}
\newcommand{\Gr}{\mathrm{Gr}}
\newcommand{\Span}{\mathrm{span}}
\newcommand{\grad}{\mathrm{grad}}

\newcommand{\subjectto}{\mathrm{s.\,t.}}
\newcommand{\T}{\mathsf{T}}
\newcommand{\GL}{\mathrm{GL}}

\newcommand{\barxi}{{\bar{\xi}}}
\newcommand{\bareta}{{\bar{\eta}}}

\newcommand{\frob}{\mathrm{F}}

\newcommand{\leftunfolding}{\mathrm{L}}
\newcommand{\rightunfolding}{\mathrm{R}}

\newcommand{\PGD}{$\mathrm{PGD}$}
\newcommand{\PtwoGD}{$\mathrm{P^2GD}$}
\newcommand{\PtwoGDR}{$\mathrm{P^2GDR}$}
\newcommand{\RFD}{$\mathrm{RFD}$}
\newcommand{\RFDR}{$\mathrm{RFDR}$}

\usepackage{tikz}
\usepackage{tikz-3dplot}
\usetikzlibrary{patterns}
\usetikzlibrary{shapes,calc,shapes,arrows}

\makeatletter
\@addtoreset{equation}{section}
\makeatother

\DeclareMathOperator{\ten}{ten}
\DeclareMathOperator{\tangent}{T}

\DeclareMathOperator{\vertical}{V}
\DeclareMathOperator{\horizontal}{H}
\DeclareMathOperator{\retr}{R}
\DeclareMathOperator{\Hess}{Hess}
\DeclareMathOperator{\sym}{sym}
\DeclareMathOperator{\Sym}{Sym}
\DeclareMathOperator{\proj}{P}

\DeclareMathOperator*{\argmin}{arg\,min}

\begin{document}

\title{Desingularization of bounded-rank tensor sets\thanks{This work was supported by the National Key R\&D Program of China (grant 2023YFA1009300). BG was supported by the Young Elite Scientist Sponsorship Program by CAST. YY was supported by the National Natural Science Foundation of China (grant No. 12288201).}
}

\titlerunning{Desingularization of tensor varieties}        

\author{Bin Gao \and Renfeng Peng \and Ya-xiang Yuan}


\institute{Bin Gao \and Ya-xiang Yuan \at
    State Key Laboratory of Scientific and Engineering Computing, Academy of Mathematics and Systems Science, Chinese Academy of Sciences, Beijing, China \\
              \email{\{gaobin,yyx\}@lsec.cc.ac.cn; }
           \and
           Renfeng Peng \at
           State Key Laboratory of Scientific and Engineering Computing, Academy of Mathematics and Systems Science, Chinese Academy of Sciences, and University of Chinese Academy of Sciences, Beijing, China\\
           \email{pengrenfeng@lsec.cc.ac.cn} 
}

\date{Received: date / Accepted: date}

\maketitle

\begin{abstract} 
    Low-rank tensors appear to be prosperous in many applications. However, the sets of bounded-rank tensors are non-smooth and non-convex algebraic varieties, rendering the low-rank optimization problems to be challenging. To this end, we delve into the geometry of bounded-rank tensor sets, including Tucker and tensor train formats. We propose a desingularization approach for bounded-rank tensor sets by introducing slack variables, resulting in a low-dimensional smooth manifold embedded in a higher-dimensional space while preserving the structure of low-rank tensor formats. Subsequently, optimization on tensor varieties can be reformulated to optimization on smooth manifolds, where the methods and convergence are well explored. We reveal the relationship between the landscape of optimization on varieties and that of optimization on manifolds. Numerical experiments on tensor completion illustrate that the proposed methods are in favor of others under different rank parameters.
\end{abstract}

\section{Introduction}
Low-rank tensor decompositions are powerful for representing multi-dimensional data, enabling the extraction of essential information from data tensors while significantly reducing storage requirements. The canonical polyadic decomposition~\cite{hitchcock1928multiple}, Tucker decomposition~\cite{tucker1964extension}, tensor train decomposition~\cite{oseledets2011tensor} and tensor ring decomposition~\cite{zhao2016tensor} are among the most typical decomposition formats; see~\cite{kolda2009tensor} for an overview. Matrix and tensor optimization leveraging low-rank decompositions appears to be effective in various applications, including image processing~\cite{vasilescu2003multilinear}, matrix and tensor completion~\cite{vandereycken2013low,kasai2016low,gao2023low,gao2024riemannian}, semidefinite programming~\cite{burer2003nonlinear,tang2023solving}, high-dimensional partial differential equations~\cite{bachmayr2016tensor,bachmayr2023low}, and dynamical tensor approximation~\cite{koch2010dynamical,bachmayr2023dynamical}.

In this paper, we are concerned with minimizing smooth functions on bounded-rank tensor sets. For Tucker decomposition, we consider the following optimization problem where the search space consists of tensors with bounded Tucker rank, i.e.,
\begin{equation}
    \begin{aligned}
        \min_{\tensX}\ &\ \ \ \ \ \ f(\tensX) \\
        \subjectto\ &\quad \tensX\in\mathbb{R}^{n_1\times n_2\times\cdots\times n_d}_{\leq\vecr}:=\{\tensX\in\mathbb{R}^{n_1\times n_2\times\cdots\times n_d}:\ranktc(\tensX)\leq\vecr\},
    \end{aligned}
    \label{eq: problem (P)}
    \tag{P}
\end{equation}
where $f:\mathbb{R}^{n_1\times n_2\times\cdots\times n_d}\to\mathbb{R}$ is a smooth function, $\vecr=(r_1,r_2,\dots,r_d)$ is an array of $d$ positive integers, and $\ranktc(\tensX)$ denotes the Tucker rank of $\tensX$. The set $\mathbb{R}^{n_1\times n_2\times\cdots\times n_d}_{\leq\vecr}$ is referred to as the Tucker tensor variety~\cite{gao2023low}. For tensor train (TT) decomposition, the optimization problem can be formulated on TT varieties~\cite{kutschan2018tangent} in a same fashion.

\paragraph{Related work and motivation}
We start with an overview of the existing research in the field of low-rank matrix optimization on $\mathbb{R}_{\leq r}^{m\times n}:=\{\matx\in\mathbb{R}^{m\times n}:\rank(\matx)\leq r\}$, which inspires the low-rank tensor optimization methods. In general, there are mainly three different geometric approaches for solving the low-rank matrix optimization problems.

The first approach is to minimize $f$ on the smooth manifold $\mathbb{R}_{r}^{m\times n}:=\{\matx\in\mathbb{R}^{m\times n}:\rank(\matx)=r\}$, ignoring the set of rank-deficient matrices $\mathbb{R}_{\leq r}^{m\times n}\setminus\mathbb{R}_{r}^{m\times n}$. One can adopt the Riemannian optimization methods to minimize $f$ on $\mathbb{R}_{r}^{m\times n}$; see, e.g., \cite{shalit2012online,vandereycken2013low}. Since $\mathbb{R}_{r}^{m\times n}$ is not closed, the classical convergence results established in Riemannian optimization (e.g.,~\cite{boumal2019global}) do not hold if a method converges to points on the boundary $\mathbb{R}_{\leq r}^{m\times n}\setminus\mathbb{R}_{r}^{m\times n}$.

Instead of working on fixed-rank manifolds, one can directly minimize $f$ on $\mathbb{R}_{\leq r}^{m\times n}$. For instance, Jain et al.~\cite{jain2014iterative} proposed the iterative hard thresholding method---also known as the projected gradient descent method (\PGD)---for low-rank matrix regression. Schneider and Uschmajew~\cite{schneider2015convergence} proposed the projected steepest descent method (\PtwoGD) and a retraction-free variant (\RFD) by exploiting the geometry of the matrix variety $\mathbb{R}_{\leq r}^{m\times n}$. Olikier and Absil~\cite{olikier2022apocalypse,olikier2023apocalypse} proposed provable first-order algorithms (\PtwoGDR\ and \RFDR). Furthermore, Olikier et~al.~\cite{olikier2023first} developed a framework for first-order optimization on general stratified sets of matrices, based on \PtwoGDR\ and \RFDR. In addition, Riemannian rank-adaptive methods (e.g.,~\cite{zhou2016riemannian,gao2022riemannian}) on~$\mathbb{R}_{\leq r}^{m\times n}$ are able to adjust the rank of iterates.

The third approach is based on parametrization, which aims to circumvent the non-smoothness of varieties. For instance, one can parametrize the feasible set $\mathbb{R}_{\leq r}^{m\times n}$ by a manifold $\tensM=\mathbb{R}^{m\times r}\times\mathbb{R}^{n\times r}$ along with a smooth mapping $\varphi:\tensM\to\mathbb{R}_{\leq r}^{m\times n}$ satisfying $\varphi(\tensM)=\mathbb{R}_{\leq r}^{m\times n}$. Consequently, minimizing $f$ on $\mathbb{R}_{\leq r}^{m\times n}$ can be reformulated as minimizing $f\circ\varphi$ on the manifold~$\tensM$, which can be solved by Riemannian optimization methods (see, e.g.,~\cite{absil2009optimization,boumal2023intromanifolds} for an overview). Theoretically, if $x\in\tensM$ is a second-order stationary point of $f\circ\varphi$, $\varphi(x)\in\mathbb{R}_{\leq r}^{m\times n}$ is a first-order stationary point of $f$; see~\cite{ha2020equivalence}. In the same spirit, Levin et~al.~\cite{levin2023finding} proposed a Riemannian trust-region method for minimizing $f\circ\varphi$ on the feasible set $\mathbb{R}^{m\times r}\times\mathbb{R}^{n\times r}$, denoted by ``RTR-LR''. More recently, a parametrization of matrix variety $\mathbb{R}^{m\times n}_{\leq r}$ called \emph{desingularization} was considered in~\cite{naldi2015exact,khrulkov2018desingularization,rebjock2024optimization}. This parametrization involves a slack variable~$\matP$ in the Grassmann manifold $\Gr(n-r,n)$ that records the kernel of $\matx$, leading to a smooth manifold 
\[\tensM(\mathbb{R}^{m\times n},r)=\{(\matx,\matP):\matx\matP=0,\matx\in\mathbb{R}^{m\times n},\matP\in\Gr(n-r,n)\}\]
and mapping $\varphi:(\matx,\matP)\mapsto\matx$. Rebjock and Boumal~\cite{rebjock2024optimization} adopted Riemannian optimization methods to this manifold and provided both global and local convergence guarantees.

These approaches have been extensively studied in low-rank matrix optimization. However, due to the intricate geometric structure of low-rank tensors, optimization on tensor sets is much more complicated than the matrix case. Nevertheless, one can still explore low-rank tensor optimization following a similar path.

Low-rank optimization on the set of Tucker tensors can be formulated as an optimization problem on a smoot h manifold of tensors with fixed Tucker rank~\cite{uschmajew2013geometry}, i.e., $\mathbb{R}^{n_1\times n_2\times\cdots\times n_d}_{\vecr}:=\{\tensX\in\mathbb{R}^{n_1\times n_2\times\cdots\times n_d}:\ranktc(\tensX)=\vecr\}$. Geometric methods can be adopted to solve the fixed-rank problem; see~\cite{uschmajew2020geometric} for an overview. For instance, Kressner et al.~\cite{kressner2014low} proposed a Riemannian conjugate gradient method (GeomCG). However, the feasible set $\mathbb{R}^{n_1\times n_2\times\cdots\times n_d}_{\vecr}$ is still not closed. Therefore, one may consider minimizing $f$ on the Tucker tensor variety $\mathbb{R}^{n_1\times n_2\times\cdots\times n_d}_{\leq\vecr}$ (the closure of $\mathbb{R}^{n_1\times n_2\times\cdots\times n_d}_{\vecr}$), which is the optimization problem~\eqref{eq: problem (P)}. Gao~et al.~\cite{gao2023low} developed the tangent cone of the Tucker tensor varieties and proposed the gradient-related approximate projection method (GRAP). One difficulty of optimization on tensor varieties is that the inherent non-smoothness of the varieties hampers the convergence analysis. Specifically, the local convergence results established in optimization on tensor varieties do not hold if an accumulation point is rank-deficient, i.e., a point in $\mathbb{R}^{n_1\times n_2\times\cdots\times n_d}_{\leq\vecr}\setminus\mathbb{R}^{n_1\times n_2\times\cdots\times n_d}_{\vecr}$; see~\cite{schneider2015convergence,gao2023low}. 

Due to the aforementioned difficulty, we resort to the third approach, which aims to solve~\eqref{eq: problem (P)} via a parametrization of Tucker tensor varieties: a smooth manifold $\tensM$ and a smooth mapping $\varphi:\tensM\to\mathbb{R}^{n_1\times n_2\times\cdots\times n_d}$ such that $\varphi(\tensM)=\mathbb{R}^{n_1\times n_2\times\cdots\times n_d}_{\leq\vecr}$. Subsequently, minimizing $f$ on $\mathbb{R}^{n_1\times n_2\times\cdots\times n_d}_{\leq\vecr}$ can be recast to minimizing $g:=f\circ\varphi$ on the smooth manifold~$\tensM$ as follows,
\begin{equation}
    \min_{x\in\tensM} g(x)=f(\varphi(x)).\label{eq: problem (Q)}\tag{Q}
\end{equation}
The smooth manifold structure enables Riemannian optimization methods for~\eqref{eq: problem (Q)}. For instance, a tensor in $\mathbb{R}^{n_1\times n_2\times\cdots\times n_d}_{\leq\vecr}$ can be parametrized by
\[\tensM^\mathrm{Tucker}=\mathbb{R}^{r_1\times r_2\times\cdots\times r_d}\times\mathbb{R}^{n_1\times r_1}\times\mathbb{R}^{n_2\times r_2}\times\cdots\times\mathbb{R}^{n_d\times r_d}\]
via Tucker decomposition; see section~\ref{sec: Preliminaries} for details. Kasai and Mishra~\cite{kasai2016low} considered a parametrization for the fixed-rank manifold $\mathbb{R}^{n_1\times n_2\times\cdots\times n_d}_{\vecr}$ by employing a quotient manifold 
\[\tensM^\mathrm{quotient}=\mathbb{R}^{r_1\times\cdots\times r_d}\times\St(r_1,n_1)\times\cdots\times\St(r_d,n_d)/(\mathcal{O}(r_1)\times\cdots\times\mathcal{O}(r_d)).\]
More recently, the desingularization of bounded-rank matrices provides a smooth parametrization of non-smooth varieties that facilitates the Riemannian optimization methods, motivating us to explore the desingularization tailored for bounded-rank tensors. However, on the one hand, the parametrizations $\tensM^\mathrm{Tucker}$ and $\tensM^\mathrm{quotient}$ only capture information in the parameter space, while overlooking intrinsic structure in the original tensor space. It is challenging to seek a parametrization that the search space both enjoys a manifold structure and preserves the structure of tensor formats. On the other hand, while tensor varieties are closely related to matrix varieties, the desingularization of tensor varieties can not be simply generalized from the desingularization of matrix varieties due to the intricacy of Tucker and TT tensor varieties. Additionally, the essential difference between Tucker and TT tensor varieties underscores the need for distinct desingularization approaches for these varieties.

\paragraph{Contribution}
We propose new parametrizations for tensor varieties. Specifically, since the Tucker tensor varieties can be characterized by the intersection of matrix varieties along each mode, we introduce slack variables $\matP_1,\matP_2,\dots,\matP_d$ and propose to desingularize $\mathbb{R}^{n_1\times n_2\times\cdots\times n_d}_{\leq\vecr}$ along each mode, i.e.,  
\[\tensM(\mathbb{R}^{n_1\times n_2\times\cdots\times n_d},\vecr)=\left\{
\begin{array}{rl}
    (\tensX,\matP_1,\matP_2,\dots,\matP_d):&\tensX\in\mathbb{R}^{n_1\times n_2\times\cdots\times n_d},\\
    &\matP_k\in\Gr(n_k-r_k,n_k),\\
    &\tensX\times_k\matP_k=0\ \text{for}\ k\in[d]
\end{array}
\right\}.\]
We prove that the set $\tensM(\mathbb{R}^{n_1\times n_2\times\cdots\times n_d},\vecr)$ is a smooth manifold and develop the Riemannian geometry of~$\tensM(\mathbb{R}^{n_1\times n_2\times\cdots\times n_d},\vecr)$, paving the way for first- and second-order Riemannian methods. By incorporating the new parametrization, we reformulate the original problem~\eqref{eq: problem (P)} as an optimization problem on a smooth manifold~\eqref{eq: problem (Q)}; see Fig.~\ref{fig: communicative}. 

\begin{figure}[htbp]
    \centering
    \begin{tikzpicture}
        \node (M) at (0,1.6) {$\tensM$};
        \node (Tc) at (0,0) {$\mathbb{R}^{n_1\times n_2\times\cdots\times n_d}_{\leq\vecr}$};
        \node (R) at (2.4,0) {$\mathbb{R}$};
        \node (E) at ($(-4.2,0)+(M)$) {$\mathbb{R}^{n_1\times\cdots\times n_d}\times\Sym(n_1)\times\cdots\times\Sym(n_d)$};

        \draw[->,thick] (M) -- (R);
        \draw[->,thick] (Tc) -- (R);
        \draw[->,thick] (M) -- (Tc);
        \draw[left hook->,thick] (M) -- (E);    

        \node[left] at ($0.5*(M)+0.5*(Tc)$) {$\varphi$};
        \node[below] at ($0.5*(R)+0.5*(Tc)+(0.4,0)$) {$f$};
        \node[right] at ($0.5*(M)+0.5*(R)+(0.1,0.1)$) {$g=f\circ\varphi$};
    \end{tikzpicture}
    \caption{Diagram of optimization via desingularization. $\tensM=\tensM(\mathbb{R}^{n_1\times n_2\times\cdots\times n_d},\vecr)$, $\Sym(n)=\{\mata\in\mathbb{R}^{n\times n}:\mata^\top=\mata\}$}
    \label{fig: communicative}
\end{figure}
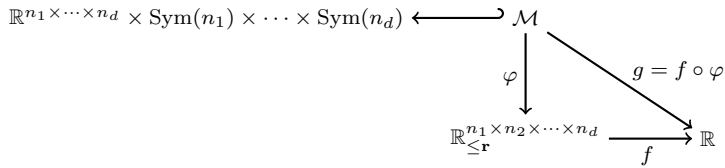

We propose Riemannian gradient descent (RGD-desing), Riemannian conjugate gradient (RCG-desing), and Riemannian trust-region (RTR-desing) methods to solve~\eqref{eq: problem (Q)}. A preliminary numerical experiment in Fig.~\ref{fig: 0} illustrates that the proposed desingularization brings a significant improvement among different parametrizations when the rank parameter is over-estimated, even without employing any rank-adaptive strategy, and reveal the underlying low-rank structure of the data tensor.

\begin{figure}[htbp]
    \centering
    \includegraphics[width=\textwidth]{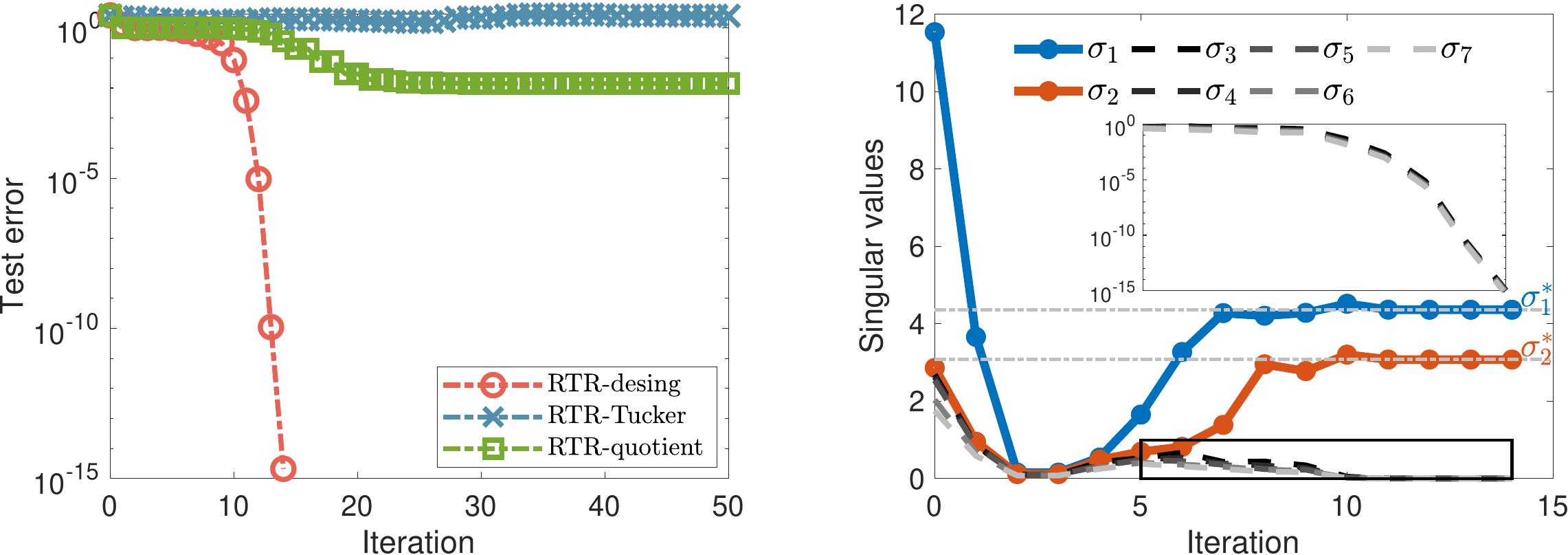}
    \caption{Errors of Riemannian trust-region methods for low-rank Tucker tensor completion under three parametrizations---Tucker parametrization $\tensM^\mathrm{Tucker}$, quotient parametrization $\tensM^\mathrm{quotient}$, and desingularization---with over-estimated rank parameter. Left: test error. Right: history of singular values of mode-1 unfolding matrices in RTR-desing}
    \label{fig: 0}
\end{figure}

Additionally, we investigate the relationship of stationary points between the original problem~\eqref{eq: problem (P)} and the parametrized problem~\eqref{eq: problem (Q)}. We provide a counterexample to demonstrate that even though $x\in\tensM(\mathbb{R}^{n_1\times n_2\times\cdots\times n_d},\vecr)$ is a second-order stationary point of~\eqref{eq: problem (Q)}, $\varphi(x)$ is not guaranteed to be first-order stationary of~\eqref{eq: problem (P)}. Moreover, we establish the connection of the sets of stationary points of~\eqref{eq: problem (Q)} between Tucker parametrization $\tensM^\mathrm{Tucker}$ and desingularization. In summary, the geometric methods in low-rank matrix and tensor optimization are listed in Table~\ref{tab: existing works}.  
\begin{table}[htbp]
    \centering
    \caption{Geometric methods in low-rank optimization. ``param.'' denotes ``parametrization''. We say a method satisfying the property ``Convergence'' if the norm of the projected anti-gradient of $f$ converges to zero. A method satisfies the ``Stationary'' property if any accumulation point is first-order stationary of the bounded-rank problem}
    \label{tab: existing works}
    \begin{tabular}{lllrcc}
        \toprule
        \multicolumn{6}{c}{Low-rank matrix optimization}\\
        \midrule
        Method & & Type & Order & Convergence  & Stationary\\ 
        \midrule
        LRGeomCG~\cite{vandereycken2013low} & & manifold & 1st & \Checkmark & \XSolidBrush\\
        \PGD~\cite{jain2014iterative} & & varieties & 1st & \Checkmark & \Checkmark\\
        \PtwoGD, \RFD~\cite{schneider2015convergence} & & varieties & 1st & \Checkmark & \XSolidBrush\\
        \PtwoGDR~\cite{olikier2022apocalypse} & & varieties & 1st & \Checkmark & \Checkmark\\
        \RFDR~\cite{olikier2023apocalypse} & & varieties & 1st & \Checkmark & \Checkmark\\
        RTR-LR~\cite{levin2023finding} & & param. & 2nd & \Checkmark & \Checkmark\\
        RTR-desing~\cite{rebjock2024optimization} & & param. & 2nd & \Checkmark & \Checkmark\\
        \midrule
        \multicolumn{6}{c}{Low-rank tensor optimization}\\
        \midrule
        Method & Format & Type & Order & Convergence & Stationary\\
        \midrule
        GeomCG~\cite{kressner2014low} & Tucker & manifold & 1st & \Checkmark & \XSolidBrush\\
        RCG~\cite{steinlechner2016riemannian} & TT & manifold & 1st & \Checkmark & \XSolidBrush\\
        RTR~\cite{psenka2020secondorderttrank} & TT & manifold & 2nd & \Checkmark & -\\
        GRAP, rfGRAP~\cite{gao2023low} & Tucker & varieties & 1st & \Checkmark & \XSolidBrush\\
        RCG-quotient~\cite{kasai2016low} & Tucker & param. & 1st & - & \XSolidBrush\\
        RGD-desing (this work) & Tucker & param. & 1st & - & \XSolidBrush\\
        RCG-desing (this work) & Tucker & param. & 1st & - & \XSolidBrush\\
        RTR-desing (this work) & Tucker & param. & 2nd & - & \XSolidBrush\\
        \bottomrule
    \end{tabular}
\end{table}

We compare the proposed RGD-desing and RCG-desing with the existing methods in tensor completion on synthetic and real-world datasets. The numerical results suggest that the proposed methods are favorably comparable to the rank-adaptive method, and perform better than the other methods under different rank parameters. 

For tensor train decomposition, we observe that the geometry of TT varieties is different from that of Tucker tensor varieties. Therefore, we propose a different parametrization by introducing a modified set of slack variables and providing a smooth manifold
\begin{equation*}
    \tensM^\mathrm{TT}=\left\{\begin{array}{rl}
        (\tensX,\matP_1,\matP_2,\dots,\matP_{d-1}):&\tensX=\llbracket\tensU_1,\tensU_2,\dots,\tensU_d\rrbracket,\\
        &\matP_k=\matI_{n_{k+1}\dots n_d}-\matP_{\geq k+1},\ k\in[d-1],\vspace{1mm}\\
        &(\tensU_1,\tensU_2,\dots,\tensU_d)\in \tensS^\mathrm{TT}
    \end{array}\right\}
\end{equation*}
to parametrize TT varieties, where $\tensS^\mathrm{TT}$ is the parameter space of TT tensors with rank constraints. 

\paragraph{Organization}
We introduce the low-rank matrix and tensor varieties in section~\ref{sec: Preliminaries}. In section~\ref{sec: desing}, we propose a new parametrization called desingularization of Tucker tensor varieties, and delve into the Riemannian geometry of the parametrization, recasting the original problem~\eqref{eq: problem (P)} to a parametrized problem. Optimization methods for the parametrized problem and its landscape are developed in section~\ref{sec: optim desing}. We provide a new approach to desingularize tensor train varieties in section~\ref{sec: TT desing}. Section~\ref{sec: experiments} reports the numerical performance of proposed methods in tensor completion. Finally, we draw the conclusion in section~\ref{sec: conclusion}.

\section{Low-rank matrices and tensors}\label{sec: Preliminaries}
In this section, we introduce the preliminaries of matrix manifold, varieties, and desingularization of matrix varieties. Then, we present tensor operations, the definition of Tucker decomposition, and the fixed-rank Tucker manifold and Tucker tensor varieties.

\subsection{Low-rank matrix manifold, varieties, and desingularization}\label{subsec: low-rank matrix}
Let $m,n,r$ be positive integers satisfying $r\leq\min\{m,n\}$. Given a matrix $\matx\in\mathbb{R}^{m\times n}$, the image of $\matx$ and its orthogonal complement are defined by $\Span(\matx):=\{\matx\vecy:\vecy\in\mathbb{R}^n\}\subseteq\mathbb{R}^m$ and $\Span(\matx)^\perp:=\{\vecz\in\mathbb{R}^m:\langle\vecx,\vecz\rangle=0\ \text{for all}\ \vecx\in\Span(\matx)\}$ respectively. The set $\St(r,n):=\{\matx\in\mathbb{R}^{n\times r}:\matx^\T\matx=\matI_r\}$ is the \emph{Stiefel manifold}. The orthogonal group is denoted by $\mathcal{O}(n):=\{\matQ\in\mathbb{R}^{n\times n}:\matQ^\T\matQ=\matQ\matQ^\T=\matI_n\}$. The set of $n$-by-$n$ symmetric matrices is defined by $\Sym(n)=\{\mata\in\mathbb{R}^{n\times n}:\mata^\top=\mata\}$. The symmetric part of a square matrix~$\mata$ is defined by $\sym(\mata):=(\mata+\mata^\top)/2$. The set of $n$-by-$n$ invertible matrices is denoted by $\GL(n)=\{\mata\in\mathbb{R}^{n\times n}:\rank(\mata)=n\}$. Recall that $\mathbb{R}_{r}^{m\times n}:=\{\matx\in\mathbb{R}^{m\times n}:\rank(\matx)=r\}$ is a smooth manifold (see, e.g.,~\cite{helmke1995critical,bruns2006determinantal}), and its closure $\mathbb{R}_{\leq r}^{m\times n}:=\{\matx\in\mathbb{R}^{m\times n}:\rank(\matx)\leq r\}$ is a determinantial variety.

\paragraph{Desingularization of matrix varieties}
In contrast with the fixed-rank manifold, geometry of the determinantial variety $\mathbb{R}_{\leq r}^{m\times n}$ is much more complicated due to the singular points~\cite{lakshmibai2015grassmannian}. Therefore, Khrulkov and Oseledets~\cite{khrulkov2018desingularization} proposed a modified version of Room--Kempf desingularization~\cite{room1938geometry,kempf1973geometry} for matrix varieties. Recently, Rebjock and Boumal~\cite{rebjock2024optimization} further revealed this manifold structure of the desingularization
\begin{equation}
    \label{eq: matrix desing}
    \tensM(\mathbb{R}^{m\times n},r)=\{(\matx,\matP):\matx\matP=0,\matx\in\mathbb{R}^{m\times n},\matP\in\Gr(n-r,n)\},
\end{equation}
where $\Gr(n-r,n):=\{\matP\in\Sym(n),\rank(\matP)=n-r,\matP=\matP^2\}$ is the Grassmann manifold. Note that a point $(\matx,\matP)\in\tensM(\mathbb{R}^{m\times n},r)$ can be represented by the parameters $\matu\in\St(r,m)$, a diagonal matrix $\Sigma\in\mathbb{R}^{r\times r}$, and $\matv\in\St(r,n)$ via $\matx=\matu\Sigma\matv^\top$ and $\matP=\matI_n-\matv\matv^\top$. The tangent space of $\tensM(\mathbb{R}^{m\times n},r)$ at $(\matx,\matP)$ can be characterized by 
    \begin{equation}
        \label{eq: tangent space for matrix desing}
        \tangent_{(\matx,\matP)}\!\tensM(\mathbb{R}^{m\times n},r)=\left\{\begin{array}{l}
            (\matK\matv^\top+\matu\Sigma\matv_p^\top,-\matv_p\matv^\top-\matv\matv_p^\top):\\
            \matK\in\mathbb{R}^{m\times r},
            \matv_p\in\mathbb{R}^{n\times r},
            \matv^\top\matv_p=0
        \end{array} \right\}.
    \end{equation}

\subsection{Tucker decomposition: definition and geometry}
We introduce tensor operations. Denote the index set $\{1,2,\dots,n\}$ by~$[n]$. The inner product between two tensors $\tensX,\tensY\in\mathbb{R}^{n_1\times n_2\times\cdots\times n_d}$ is defined by $\langle\tensX,\tensY\rangle := \sum_{i_1=1}^{n_1} \cdots \sum_{i_d=1}^{n_d} \tensX({i_1,\dots,i_d})\tensY({i_1,\dots,i_d})$. The Frobenius norm of a tensor $\tensX$ is defined by $\|\tensX\|_\mathrm{F}:=\sqrt{\langle\tensX,\tensX\rangle}$. The mode-$k$ unfolding of a tensor $\tensX \in \mathbb{R}^{n_1 \times\cdots\times n_d}$ is denoted by a matrix $\matx_{(k)}\in\mathbb{R}^{n_k\times n_{-k}} $ for $k=1,\dots,d$, where $n_{-k}:=\prod_{i\neq k}n_i$. The $ (i_1,i_2,\dots,i_d)$-th entry of $\tensX$ corresponds to the $(i_k,j)$-th entry of $\matx_{(k)}$, where
$ j = 1 + \sum_{\ell \neq k, \ell = 1}^d(i_\ell-1)J_\ell$ with $J_\ell = \prod_{m = 1,m \neq k}^{\ell-1} n_m$. The tensorization operator maps a matrix $\matx_k\in\mathbb{R}^{n_k\times n_{-k}}$ to a tensor $\ten_{(k)}(\matx_k)\in\mathbb{R}^{n_1\times\cdots\times n_d}$ defined by $\ten_{(k)}(\matx_k)(i_1,\dots,i_d)=\matx_k(i_k,1 + \sum_{\ell \neq k, \ell = 1}^d(i_\ell-1)J_\ell)$ for $(i_1,\dots,i_d)\in[n_1]\times\cdots\times[n_d]$. Note that $\ten_{(k)}(\matx_{(k)})=\tensX$ holds for fixed $n_1,\dots,n_d$. Therefore, the tensorization operator is invertible. 
The $k$-mode product of a tensor $\tensX$ and a matrix $\mata\in\mathbb{R}^{n_k\times M}$ is denoted by $\tensX\times_k\mata\in\mathbb{R}^{n_1\times\cdots\times M\times\cdots\times n_d}$, where the $ (i_1,\dots,i_{k-1},j,i_{k+1},\dots,i_d)$-th entry of $\tensX\times_k\mata$ is $\sum_{i_k=1}^{n_k}x_{i_1\dots i_d}a_{ji_k}$. It holds that $(\tensX\times_k\mata)_{(k)}=\mata\matx_{(k)}$. Given $\vecu_1\in\mathbb{R}^{n_1}\setminus\{0\},\dots,\vecu_d\in\mathbb{R}^{n_d}\setminus\{0\}$, a rank-$1$ tensor of size $n_1\times n_2\times\cdots\times n_d$ is defined by the outer product $\tensV=\vecu_1\circ\vecu_2\circ\cdots\circ\vecu_d$, or $v_{i_1,\dots,i_d}=u_{1,i_1}\cdots u_{d,i_d}$ equivalently. The Kronecker product of two matrices $\mata\in\mathbb{R}^{m_1\times n_1}$ and $\matb\in\mathbb{R}^{m_2\times n_2}$ is an $(m_1m_2)$-by-$(n_1n_2)$ matrix defined by $\mata\otimes\matb:=(a_{ij}\matb)_{ij}$. Given two vectors $\vecx,\vecy\in\mathbb{R}^d$, we denote $\vecx\leq\vecy$ ($\vecx<\vecy$) if $x_i\leq y_i$ ($x_i<y_i$) for all $i\in[d]$.

\begin{definition}[Tucker decomposition]
    Given a tensor $\tensX \in \mathbb{R}^{n_1 \times n_2\times\cdots\times n_d}$, 
    the Tucker decomposition is \[\tensX =\tensG\times_1\matu_1\times_2\matu_2\cdots\times_d\matu_d=\tensG\times_{k=1}^d\matu_k,\]
    where $\tensG\in\mathbb{R}^{r_1 \times r_2\times\cdots\times r_d}$ is a core tensor, $\matu_k\in\St(r_k,n_k)$ are factor matrices with orthogonal columns. The Tucker rank of a tensor $\tensX$ is defined by $$\ranktc(\tensX):=(\rank(\matx_{(1)}),\rank(\matx_{(2)}),\dots,\rank(\matx_{(d)})).$$ 
    We refer to $\tensX=\tensG\times_{k=1}^d\matu_k$ as a \emph{thin} Tucker decomposition if $r_k=\rank(\matx_{(k)})$ for all $k\in[d]$. Figure~\ref{fig: 3D Tucker} depicts the Tucker decomposition of a third-order tensor. 
\end{definition}
\begin{figure}[htbp]
    \centering
    \includegraphics[scale=0.833]{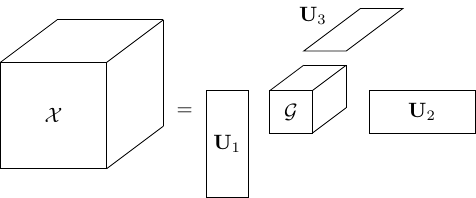}
    \caption{Tucker decomposition of a third-order tensor}
    \label{fig: 3D Tucker}
\end{figure}

The mode-$k$ unfolding of a tensor $\tensX =\tensG\times_1\matu_1\times_2\matu_2\cdots\times_d\matu_d$ satisfies 
\[\matx_{(k)}=\matu_k\matG_{(k)}\left(\matu_d\otimes\cdots\otimes\matu_{k+1}\otimes\matu_{k-1}\otimes\cdots\otimes\matu_{1}\right)^\T=\matu_k\matG_{(k)}((\matu_j)^{\otimes j\neq k})^\T,\]
where $(\matu_j)^{\otimes j\neq k}=\matu_d\otimes\cdots\otimes\matu_{k+1}\otimes\matu_{k-1}\otimes\cdots\otimes\matu_{1}$ for $k\in[d]$. Given a tensor $\tensA\in\mathbb{R}^{n_1\times n_2\times \cdots\times n_d}$, we denote $\tensA\in\bigotimes_{k=1}^d\Span(\matu_k)$ if $\tensA=\tensC\times_{k=1}^d\matu_k$ for some $\tensC\in\mathbb{R}^{r_1\times r_2\times\cdots\times r_d}$.

Given a tuple $\vecr=(r_1,r_2,\dots,r_d)$, the set of tensors of fixed Tucker rank is denoted by $\mathbb{R}^{n_1\times n_2\times \cdots\times n_d}_\vecr:=\{\tensX\in\mathbb{R}^{n_1\times n_2\times\cdots\times n_d}:\ranktc(\tensX)=\vecr\}$, which is a smooth manifold~\cite{koch2010dynamical} of dimension $r_1r_2\cdots r_d+\sum_{k=1}^{d}r_k(n_k-r_k)$. The set of tensors of bounded Tucker rank $\mathbb{R}^{n_1\times n_2\times \cdots\times n_d}_{\leq\vecr}:=\{\tensX\in\mathbb{R}^{n_1\times n_2\times \cdots\times n_d}:\ranktc(\tensX)\leq\vecr\}$ is referred to as the Tucker tensor variety. 
Given a tensor $\tensX\in\mathbb{R}^{n_1\times n_2\times\cdots\times n_d}$ with $\ranktc(\tensX)=\underline{\vecr}\leq\vecr$ and a thin Tucker decomposition $\tensX=\underline{\tensG}\times_{k=1}^d\underline{\matu}_k$, Gao et al.~\cite[Theorem 1]{gao2023low} provided an explicit parametrization 
\begin{equation}
    \label{eq: Tucker tangent cone}
    \begin{aligned}
        \tensV=\tensC\times_{k=1}^d\begin{bmatrix}
            \underline{\matu}_k & \matu_{k,1}
        \end{bmatrix}+\sum_{k=1}^d\underline{\tensG}\times_k(\matu_{k,2}\matR_{k,2})\times_{j\neq k}\underline{\matu}_j
    \end{aligned}
\end{equation}
of an element $\tensV$ in the \emph{tangent cone} $\tangent_\tensX\!\mathbb{R}^{n_1\times n_2\times \cdots\times n_d}_{\leq\vecr}$ at $\tensX$, where $\tensC\in\mathbb{R}^{r_1\times r_2\times\cdots\times r_d}$, $\matR_{k,2}\in\mathbb{R}^{(n_k-r_k)\times \underline{r}_k}$, $\matu_{k,1}\in\St(r_k-\underline{r}_k,n_k)$ and $\matu_{k,2}\in\St(n_k-r_k,n_k)$ are arbitrary that satisfy $[\underline{\matu}_k\ \matu_{k,1}\ \matu_{k,2}]\in\mathcal{O}(n_k)$ for $k\in[d]$.

\section{Desingularization of Tucker tensor varieties}\label{sec: desing}
In this section, we propose a new parametrization of the Tucker tensor varieties by applying desingularization of matrix varieties to each mode, resulting in a smooth manifold $\tensM$. Subsequently, we reveal the Riemannian geometry of $\tensM$ and provide implementation details of the geometric tools to avoid large matrix and tensor formulations. We also illustrate the connection of the proposed desingularization to known geometric results.
\subsection{A new parametrization}\label{subsec: Tucker desing}
In the light of the fact 
\[\mathbb{R}^{n_1\times n_2\times\cdots\times n_d}_{\leq\vecr}=\bigcap_{k=1}^d\ten_{(k)}(\mathbb{R}^{n_k\times n_{-k}}_{\leq r_k})\]
and desingularization of matrix varieties~\eqref{eq: matrix desing}, we propose to desingularize the Tucker tensor varieties along each mode, i.e.,
\begin{equation}
    \label{eq: desingularization of Tucker}
    \tensM(\mathbb{R}^{n_1\times n_2\times\cdots\times n_d},\vecr)=\left\{
    \begin{array}{rl}
        (\tensX,\matP_1,\matP_2,\dots,\matP_d):&\tensX\in\mathbb{R}^{n_1\times n_2\times\cdots\times n_d},\\
        &\matP_k\in\Gr(n_k-r_k,n_k),\\
        &\tensX\times_k\matP_k=0\ \text{for}\ k\in[d]
    \end{array}
    \right\}.
\end{equation}
For the sake of brevity, we denote $\tensM(\mathbb{R}^{n_1\times n_2\times\cdots\times n_d},\vecr)$ by $\tensM$. The following proposition confirms that~\eqref{eq: desingularization of Tucker} is a  parametrization of $\mathbb{R}^{n_1\times n_2\times\cdots\times n_d}_{\leq\vecr}$.

\begin{proposition}
    The Tucker tensor variety $\mathbb{R}^{n_1\times\cdots\times n_d}_{\leq\vecr}$ can be parametrized by~\eqref{eq: desingularization of Tucker} via the mapping 
    \begin{equation*}
        \begin{aligned}
            \varphi:\tensM\to\mathbb{R}^{n_1\times n_2\times\cdots\times n_d}_{\leq\vecr},\qquad
            \varphi(\tensX,\matP_1,\matP_2,\dots,\matP_d)=\tensX,
        \end{aligned}
    \end{equation*}
    i.e., $\varphi(\tensM)=\mathbb{R}^{n_1\times n_2\times\cdots\times n_d}_{\leq\vecr}$.
\end{proposition}
\begin{proof}
    On the one hand, for $\tensX\in\varphi(\tensM)$, there exists $\matP_k\in\Gr(n_k-r_k,n_k)$, such that $\matP_k\matx_{(k)}=(\tensX\times_k\matP_k)_{(k)}=0$. Hence, we have $\rank(\matx_{(k)})\leq n_k-(n_k-r_k)=r_k$, $\ranktc(\tensX)\leq\vecr$ and $\varphi(\tensM)\subseteq\mathbb{R}^{n_1\times n_2\times\cdots\times n_d}_{\leq\vecr}$. On the other hand, any $\tensX\in\mathbb{R}^{n_1\times n_2\times\cdots\times n_d}_{\leq\vecr}$ admits a Tucker decomposition $\tensX=\tensG\times_{k=1}^d\matu_k$ with $\tensG\in\mathbb{R}^{r_1\times r_2\times\cdots\times r_d}$ and $\matu_k\in\St(r_k,n_k)$. Let $\matP_k=\matI_{n_k}-\matu_k^{}\matu_k^\top$, we observe that $\matP_k\matX_{(k)}=\matP_k\matu_k\matG_{(k)}((\matu_j)^{\otimes j\neq k})^\top=0$. Therefore, we obtain a parametrization $(\tensX,\matP_1,\matP_2,\dots,\matP_d)\in\tensM$ and thus $\mathbb{R}^{n_1\times n_2\times\cdots\times n_d}_{\leq\vecr}\subseteq\varphi(\tensM)$. Consequently, it holds that $\varphi(\tensM)=\mathbb{R}^{n_1\times n_2\times\cdots\times n_d}_{\leq\vecr}$. \hfill\squareforqed
\end{proof}

It is worth noting that an element $(\tensX,\matP_1,\matP_2,\dots,\matP_d)$ in $\tensM$ can be represented by parameters $\tensG\in\mathbb{R}^{r_1\times r_2\times\cdots\times r_d}$ and $\matu_k\in\St(r_k,n_k)$ via 
\begin{equation}
    \label{eq: parametrization}\tensX=\tensG\times_{k=1}^d\matu_k\qquad\text{and}\qquad\matP_k=\matI_{n_k}-\matu_k^{}\matu_k^\top,
\end{equation}
which facilitates efficient computations in practice. In contrast with the desingularization of matrix varieties in subsection~\ref{subsec: low-rank matrix}, this parametrization is unique in the sense of $\Span(\matu_k)$. Additionally, in contrast with the parameters for the manifold of fixed-rank Tucker tensors $\mathbb{R}^{n_1\times n_2\times\cdots\times n_d}_{\vecr}$, there is no rank constraint for~$\tensG$.

\subsection{Manifold structure}
We prove that $\tensM$ is an embedded submanifold of
\[\tensE=\mathbb{R}^{n_1\times n_2\times\cdots\times n_d}\times\Sym(n_1)\times\Sym(n_2)\times\cdots\times\Sym(n_d).\] 
Consider two auxiliary sets
\begin{equation*}
    \begin{aligned}
        \tensS&=\{(\tensX,\matv_1,\matv_2,\dots,\matv_d)\in\bar{\tensS}:\tensX\times_k\matv_k^\top=0\ \text{for}\ k\in[d]\},\\
        \bar{\tensS}&=\mathbb{R}^{n_1\times n_2\times\cdots\times n_d}\times\St(n_1-r_1,n_1)\times\St(n_2-r_2,n_2)\times\cdots\times\St(n_d-r_d,n_d).
    \end{aligned}
\end{equation*}
First, we reveal the manifold structure of $\tensS$ by a reformulation. Next, we prove that $\tensM$ and a quotient manifold of $\tensS$ are diffeomorphic.

\begin{lemma}
    Consider the set 
    \[\tilde{\tensS}:=\{(\tensX,\matv_1,\matv_2,\dots,\matv_d)\in\bar{\tensS}:\tensX\times_{j=1}^{k-1}(\matv_j^\perp)^\top\times_k\matv_k^\top=0\ \text{for}\ k\in[d]\},\] 
    where $\matv_k^\perp\in\St(r_k,n_k)$ satisfies $\Span(\matv_k^\perp)=\Span(\matv_k)^\perp$. It holds that $\tensS=\tilde{\tensS}$.
\end{lemma}
\begin{proof}
    On the one hand, for $(\tensX,\matv_1,\matv_2,\dots,\matv_d)\in \tensS$, it is straightforward to verify that $(\tensX,\matv_1,\matv_2,\dots,\matv_d)\in\tilde{\tensS}$ and thus $\tensS\subseteq\tilde{\tensS}$. On the other hand, for $(\tensX,\matv_1,\matv_2,\dots,\matv_d)\in \tilde{\tensS}$, we prove $\tensX\times_k\matv_k^\top=0$ for all $k\in[d]$ by induction. In fact, $\tensX\times_k\matv_k^\top=0$ holds for $k=1$. Subsequently, by using $\matv_j^{}\matv_j^\top+\matv_j^\perp(\matv_j^\perp)^\top=\matI_{n_j}$ and $\tensX\times_j\matv_j^\top=0$ for all $j\in[k-1]$, we obtain that
    \begin{equation*}
        \begin{aligned}
            \tensX\times_k\matv_k^\top&=\tensX\times_{j=1}^{k-1}(\matv_j^{}\matv_j^\top+\matv_j^\perp(\matv_j^\perp)^\top)\times_k\matv_k^\top\\
            &=(\tensX\times_{j=1}^{k-1}(\matv_j^\perp)^\top\times_k\matv_k^\top)\times_{j=1}^{k-1}\matv_j^\perp\\
            &=0,
        \end{aligned}
    \end{equation*}
    i.e., $\tensX\times_j\matv_j^\top=0$ holds for all $j\in[k]$. Therefore, it holds that $\tensS=\tilde{\tensS}$. \hfill\squareforqed
\end{proof}

The set $\tilde{\tensS}$ does not depend on a specific $\matv_k^\perp$. For the sake of the following analysis, $\matv_k^\perp\in\St(r_k,n_k)$ is select by the last $r_k$ columns of $\matQ$, where $\matQ\in\mathcal{O}(n_k)$ is the Q-factor of the QR decomposition of $[\matv\ \matI_{n_k}]$. Since QR decomposition is differentiable, $\matv_k^\perp$ can be viewed as a smooth function of $\matv_k$~\cite[Proposition 3.4.6]{absil2009optimization}. The set $\tilde{\tensS}$ provides an equivalent characterization of $\tensS$, i.e., $\tilde{\tensS} = \tensS \subseteq \bar{\tensS}$, which facilitates the identification of manifold structure. 
\begin{proposition}
    The set $\tilde{\tensS}$ is an embedded submanifold of $\bar{\tensS}$ with dimension $\dim\tilde{\tensS}=r_1r_2\cdots r_d+\sum_{k=1}^{d}\dim\St(n_k-r_k,n_k)$.
\end{proposition}
\begin{proof}
    Consider the differentiable function 
    \[h(\tensX,\matv_1,\matv_2,\dots,\matv_d)=(\tensX\times_1\matv_1^\top,\tensX\times_1(\matv_1^\perp)^\top\times_2\matv_2^\top,\dots,\tensX\times_{j=1}^{d-1}(\matv_j^\perp)^\top\times_d\matv_d^\top)\] 
    defined on $\bar{\tensS}$, which is a local defining function, i.e., $h^{-1}(0)=\tilde{\tensS}$. We observe that the Jacobian matrix of $h(\tensX,\matv_1,\matv_2,\dots,\matv_d)$ is in the form of~$[\matJ_0\ \matJ_1\ \matJ_2\ \dots\ \matJ_d]$, where 
    \[\matJ_0=\begin{bmatrix}
        \matI_{n_d}\otimes\matI_{n_{d-1}}\otimes\cdots\otimes\matI_{n_3}\otimes\matI_{n_2}\otimes\matv_1^\top\\
        \matI_{n_d}\otimes\matI_{n_{d-1}}\otimes\cdots\otimes\matI_{n_3}\otimes\matv_2^\top\otimes(\matv_1^\perp)^\top\\
        \vdots\\
        \matv_d^\top\otimes(\matv_{d-1}^\perp)^\top\otimes\cdots\otimes(\matv_1^\perp)^\top
    \end{bmatrix}\in\mathbb{R}^{(n_1n_2\cdots n_d-r_1r_2\cdots r_d)\times(n_1n_2\cdots n_d)}\]
    and $\matJ_k$ with appropriate sizes. Since $\matJ_0^{}\matJ_0^\top=\matI_{n_1n_2\cdots n_d-r_1r_2\cdots r_d}$, the differential $\mathrm{D}h(\tensX,\matv_1,\matv_2,\dots,\matv_d)$ is of full rank. Therefore, $\tilde{\tensS}$ is a submanifold of $\bar{\tensS}$ with dimension 
    \[\dim\bar{\tensS}-(n_1n_2\cdots n_d-r_1r_2\cdots r_d)=r_1r_2\cdots r_d+\sum_{k=1}^{d}\dim\St(n_k-r_k,n_k).\]\hfill\squareforqed
\end{proof}

Subsequently, we can prove that $\tensM$ is a smooth manifold in the light of the quotient manifold. 
\begin{theorem}
    $\tensM$ is a smooth manifold with dimension $r_1r_2\cdots r_d+\sum_{k=1}^{d}r_k(n_k-r_k)$.
\end{theorem}
\begin{proof}
    First, we introduce a quotient manifold of $\tilde{\tensS}$. Given $(\tensX,\matv_1,\matv_2,\dots,\matv_d)\in\tilde{\tensS}$, we consider the  group action
    \[\theta((\tensX,\matv_1,\matv_2,\dots,\matv_d),\matQ_1,\matQ_2,\dots,\matQ_d)=(\tensX,\matv_1\matQ_1,\matv_2\matQ_2,\dots,\matv_d\matQ_d)\]
    of $(\times_{k=1}^d\mathcal{O}(n_k-r_k))$, which introduces an equivalent class $\sim$. Since the group action $\theta$ is smooth, \emph{free} (from orthogonality of $\matv_k\in\St(n_k-r_k,n_k)$), and \emph{proper} (from compactness of $(\times_{k=1}^d\mathcal{O}(n_k-r_k))$). Therefore, the set $\tilde{\tensS}/\!\sim$ is a quotient manifold of $\tilde{\tensS}$ with dimension
    \[r_1r_2\cdots r_d+\sum_{k=1}^{d}(\dim\St(n_k-r_k,n_k)-\dim\mathcal{O}(n_k-r_k))=r_1r_2\cdots r_d+\sum_{k=1}^{d}r_k(n_k-r_k);\]
    see \cite[Theorem 21.10]{lee2012smooth}. Note that $\tilde{\tensS}/\!\sim=\mathbb{R}^{n_1\times\cdots\times n_d}\times(\St(n_1-r_1,n_1)/\mathcal{O}(n_1-r_1))\times\cdots\times(\St(n_d-r_d,n_d)/\mathcal{O}(n_d-r_d))$.

    In order to show that $\tensM$ is a smooth manifold, we consider the mapping $\psi:\tilde{\tensS}/\!\sim\to\tensE$ defined by 
    \[\psi([\tensX,\matv_1,\matv_2,\dots,\matv_d])=(\tensX,\matv_1^{}\matv_1^\top,\matv_2^{}\matv_2^\top,\dots,\matv_d^{}\matv_d^\top).\]
    It holds that $\tensM=\psi(\tilde{\tensS}/\!\sim)$ and $\psi$ is a homeomorphism onto $\tensM$. Since the differential $\mathrm{D}\psi([\tensX,\matv_1,\matv_2,\dots,\matv_d])|_{\tangent_{[\tensX,\matv_1,\matv_2,\dots,\matv_d]}\!(\tilde{\tensS}/\!\sim)}$ is injective from the horizontal space of $\St(n_k-r_k,n_k)/\mathcal{O}(n_k-r_k)$ in~\cite[Example 9.26]{boumal2023intromanifolds}, the mapping $\psi$ is an immersion and thus $\tensM$ is an embedded submanifold of~$\tensE$~\cite[Proposition 5.2]{lee2012smooth}.  \hfill\squareforqed
\end{proof}

It is worth noting that while $\tensM$ is embedded in a higher-dimensional space~$\tensE$, the dimension of $\tensM$ equals to the dimension of the fixed-rank Tucker manifold $\mathbb{R}^{n_1\times n_2\times\cdots\times n_d}_{\vecr}$, which highlights the effectiveness of proposed desingularization. The proposed desingularization is able to provide a smooth manifold representation, which is beneficial to handle rank-deficient points ``gracefully'', enabling explicit projections onto tangent spaces without the complexities associated with non-smooth tensor varieties; see the following sections.

\subsection{Geometry of the desingularization: first order}
We develop the Riemannian geometry of $\tensM$ that enables first-order methods, including tangent space, Riemannian metric, projection onto the tangent space, and retraction.

Recall that $x=(\tensX,\matP_1,\matP_2,\dots,\matP_d)\in\tensM\subseteq\mathbb{R}^{n_1\times n_2\times\cdots\times n_d}\times\mathbb{R}^{n_1\times n_1}\times\mathbb{R}^{n_2\times n_2}\times\cdots\times\mathbb{R}^{n_d\times n_d}$ can be represented by $(\tensG,\matu_1,\matu_2,\dots,\matu_d)\subseteq\mathbb{R}^{r_1\times r_2\times\cdots\times r_d}\times\mathbb{R}^{n_1\times r_1}\times\mathbb{R}^{n_2\times r_2}\times\cdots\times\mathbb{R}^{n_d\times r_d}$ via~\eqref{eq: parametrization}. Therefore, the following geometric tools are also able to be represented by $(\tensG,\matu_1,\matu_2,\dots,\matu_d)$, which avoids large matrix and tensor formulations. In other words, we never assemble $x$ explicitly and only store the representation $(\tensG,\matu_1,\matu_2,\dots,\matu_d)$.

\begin{proposition}
    The tangent space of $\tensM$ at $x=(\tensX,\matP_1,\matP_2,\dots,\matP_d)$ is given by 
    \begin{equation}
        \label{eq: desingularization tangent space}
        \tangent_x\!\tensM=\left\{\begin{array}{ll}
            (\dot{\tensX},\dot{\matP}_1,\dots,\dot{\matP}_d):&\dot{\tensX}=\dot{\tensG}\times_{k=1}^d\matu_k+\sum_{k=1}^d\tensG\times_k\dot{\matu}_k\times_{j\neq k}\matu_j,\\
            &\dot{\matP}_k=-\dot{\matu}_k^{}\matu_k^\top-\matu_k^{}\dot{\matu}_k^\top,\\
            &\dot{\tensG}\in\mathbb{R}^{r_1\times r_2\times\cdots\times r_d},\dot{\matu}_k\in\mathbb{R}^{n_k\times r_k},\dot{\matu}_k^\top\matu_k^{}=0
        \end{array} \right\},
    \end{equation}
    where $(\tensG,\matu_1,\matu_2,\dots,\matu_d)$ is a representation of $x\in\tensM$.
\end{proposition}
\begin{proof}
    Denote the right hand side by $T$. First, we show that $T\subseteq\tangent_x\!\tensM$. For all $(\dot{\tensX},\dot{\matP}_1,\dots,\dot{\matP}_d)\in T$ with $\dot{\matP}_k=-\dot{\matu}_k^{}\matu_k^\top-\matu_k^{}\dot{\matu}_k^\top$, since $\dot{\matu}_k^\top\matu_k^{}=0$, it follows from~\cite[\S 7.3]{boumal2023intromanifolds} that there exists a smooth curve $\matu_k(t):\mathbb{R}\to\St(r_k,n_k)$ such that $\matu_k(0)=\matu_k$ and $\matu_k^\prime(0)=\dot{\matu}_k$. Subsequently, we consider the smooth curve $\gamma:\mathbb{R}\to\tensM$ defined by
    \[\gamma(t)=((\tensG+t\dot{\tensG})\times_{k=1}^d\matu_k(t),\gamma_1(t),\gamma_2(t),\dots,\gamma_d(t))\in\tensM,\]
    where $\gamma_k(t)=\matI_{n_k}-\matu_k(t)\matu_k(t)^\top$. It is straightforward to show that $\gamma(0)=x$ and $\gamma^\prime(0)=(\dot{\tensX},\dot{\matP}_1,\dots,\dot{\matP}_d)$. Therefore, $(\dot{\tensX},\dot{\matP}_1,\dots,\dot{\matP}_d)\in\tangent_x\!\tensM$. Since $\dim T=r_1r_2\cdots r_d+\sum_{k=1}^d (n_kr_k-r_k^2)=\dim\tensM=\dim\tangent_x\!\tensM$, we obtain that $T=\tangent_x\!\tensM$. \hfill\squareforqed
\end{proof}

In practice, a tangent vector in $\tangent_x\!\tensM$ can be represented by parameters $\dot{\tensG}\in\mathbb{R}^{r_1\times r_2\times\cdots\times r_d}$ and $\dot{\matu}_k\in\mathbb{R}^{n_k\times r_k}$ satisfying $\dot{\matu}_k^\top\matu_k^{}=0$, which is able to save computational cost. To facilitate the projections, we endow the ambient space $\tensE$ with an inner product $\langle\cdot,\cdot\rangle_x$ defined by 
\[\langle\barxi,\bareta\rangle_x=\langle\tensA,\tensB\rangle+\sum_{k=1}^{d}\langle\matA_k,\matB_k\rangle,\]
for $x=(\tensX,\matP_1,\matP_2,\dots,\matP_d)\in\tensM$, tangent vectors $\barxi=(\tensA,\matA_1,\dots,\mata_d)\in\tangent_x\!\tensE\simeq\tensE$ and $\bareta=(\tensB,\matB_1,\dots,\matb_d)\in\tangent_x\!\tensE\simeq\tensE$, enabling $(\tensE,\langle\cdot,\cdot\rangle)$ to be a Riemannian manifold. Then, $\tensM$ is a Riemannian submanifold of $\tensE$ by inheriting the inner product $\langle\cdot,\cdot\rangle$ from $\tensE$. Given two tangent vectors $\xi,\eta\in\tangent_x\!\tensM$, the inner product of $\xi,\eta$ can be computed by their parameters in the following proposition.
\begin{proposition}
    Given tangent vectors $\xi,\eta\in\tangent_x\!\tensM$ represented by $(\hat{\tensG},\hat{\matu}_1,\dots,\hat{\matu}_d)$ and $(\breve{\tensG},\breve{\matu}_1,\dots,\breve{\matu}_d)$, the inner product of $\xi,\eta$ can be efficiently computed by
    \begin{equation}
        \label{eq: inner product desing}
        \langle\xi,\eta\rangle_x=\langle\hat{\tensG},\breve{\tensG}\rangle+\sum_{k=1}^{d}\langle\hat{\matu}_k,\breve{\matu}_k(2\matI_{r_k}+\matG_{(k)}^{}\matG_{(k)}^\top)\rangle.
    \end{equation}
\end{proposition}
\begin{proof}
    By using $\matu_k\in\St(r_k,n_k)$, $\matu_k^\top\hat{\matu}_k^{}=0$ and $\matu_k^\top\breve{\matu}_k^{}=0$, we obtain that 
    \begin{equation*}
        \begin{aligned}
            \langle\xi,\eta\rangle_x&=\langle\hat{\tensG}\times_{k=1}^d\matu_k+\sum_{k=1}^d\tensG\times_k\hat{\matu}_k\times_{j\neq k}\matu_j,\breve{\tensG}\times_{k=1}^d\matu_k+\sum_{k=1}^d\tensG\times_k\breve{\matu}_k\times_{j\neq k}\matu_j\rangle\\
            &~~~+\sum_{k=1}^{d}\langle\hat{\matu}_k^{}\matu_k^\top+\matu_k^{}\hat{\matu}_k^\top,\breve{\matu}_k^{}\matu_k^\top+\matu_k^{}\breve{\matu}_k^\top\rangle\\
            &=\langle\hat{\tensG},\breve{\tensG}\rangle+\sum_{k=1}^d\langle\tensG\times_k\hat{\matu}_k,\tensG\times_k\breve{\matu}_k\rangle+2\sum_{k=1}^{d}\langle\hat{\matu}_k^{},\breve{\matu}_k^{}\rangle\\
            \omit &=\langle\hat{\tensG},\breve{\tensG}\rangle+\sum_{k=1}^{d}\langle\hat{\matu}_k,\breve{\matu}_k(2\matI_{r_k}+\matG_{(k)}^{}\matG_{(k)}^\top)\rangle. \hspace{4.7cm}\squareforqed
        \end{aligned}
    \end{equation*} 
\end{proof}

Subsequently, we compute the projection of a vector onto the tangent space.
\begin{proposition}\label{prop: proj onto tangent}
    Given $x=(\tensX,\matP_1,\dots,\matP_d)\in\tensM$ represented by $(\tensG,\matu_1,\dots,\matu_d)$, and a vector $\barxi=(\tensA,\mata_1,\dots,\mata_d)\in\tensE$, the orthogonal projection $\proj_x\!\barxi$ onto $\tangent_x\!\tensM$ is a tangent vector represented by
    \begin{equation}\label{eq: projection onto tangent}
        \begin{aligned}
            \tilde{\tensG}&=\tensA\times_{k=1}^d\matu_k^\top,\\ 
            \tilde{\matU}_k&=\matP_k(\mata_{(k)}^{}\!\matv_k^{}\matG_{(k)}^\top-2\sym(\mata_k)\matu_k)(2\matI_{r_k}+\matG_{(k)}^{}\matG_{(k)}^\top)^{-1},
        \end{aligned}
    \end{equation}
    where $\matv_k=(\matu_j)^{\otimes j\neq k}$ for $k\in[d]$.
\end{proposition}
\begin{proof}
    Since $\proj_x\!\barxi\in\tangent_x\tensM$ is a tangent vector, it suffices to compute the parameters $(\tilde{\tensG},\tilde{\matu}_1,\dots,\tilde{\matu}_d)$. For any $\eta\in\tangent_x\!\tensM$ represented by $(\dot{\tensG},\dot{\matu}_1,\dots,\dot{\matu}_d)$, it holds that $\langle\eta,\barxi-\proj_x\!\barxi\rangle_x=0$, i.e.,
    \begin{equation*}
        \langle\dot{\tensG},\tensA\times_{k=1}^d\matu_k^\top-\tilde{\tensG}\rangle+\sum_{k=1}^d\langle\dot{\matu}_k,\mata_{(k)}\!\matv_k\matG_{(k)}^\top-2\sym(\mata_k)\matu_k-\tilde{\matu}_k(2\matI_{r_k}+\matG_{(k)}^{}\matG_{(k)}^\top)\rangle=0
    \end{equation*}
    holds for all parameters $(\dot{\tensG},\dot{\matu}_1,\dots,\dot{\matu}_d)$. Therefore, we obtain that 
    \[\tilde{\tensG}=\tensA\times_{k=1}^d\matu_k^\top\ \text{and}\ \tilde{\matU}_k=\matP_k(\mata_{(k)}\!\matv_k\matG_{(k)}^\top-2\sym(\mata_k)\matu_k)(2\matI_{r_k}+\matG_{(k)}^{}\matG_{(k)}^\top)^{-1}.\] 
    Since $\matP_k\matu_k=0$, we obtain that $\tilde{\matU}_k^\top\matu_k^{}=0$. Consequently, $(\tilde{\tensG},\tilde{\matu}_1,\dots,\tilde{\matu}_d)$ is a representation of $\proj_x\!\barxi$. \hfill\squareforqed
\end{proof}

By using Proposition~\ref{prop: proj onto tangent}, we obtain the \emph{Riemannian gradient} of $g=f\circ\varphi$. Specifically, the Riemannian gradient $\grad g(x)$ is the unique tangent vector satisfying $\langle\grad g(x),\eta\rangle_x=\mathrm{D} g(x)[\eta]$ for all $\eta\in\tangent_x\!\tensM$, where $\mathrm{D}g(x)[\eta]$ refers to the differential of $g$ at $x$ along $\eta$. The Riemannian gradient can be computed by $\grad g(x)=\proj_{\tangent_x\!\tensM}(\nabla g(x))$~\cite[Proposition 3.61]{boumal2023intromanifolds}, where $\nabla g(x)$ is the Euclidean gradient of $g$ at $x$.
\begin{corollary}
    Let $\barxi=\nabla g(x)=(\nabla f(\tensX),0,0,\dots,0)\in\tensE$, the Riemannian gradient of $g$ at $x=(\tensX,\matP_1,\matP_2,\dots,\matP_d)$ can be computed by 
    \begin{equation}
        \label{eq: Riemannian gradient desing}
        \grad g(x)=\Big(\nabla f(\tensX)\times_{k=1}^d\proj_{\matu_k}+\sum_{k=1}^{d}\tensG\times_{k}\tilde{\matu}_k\times_{j\neq k}\matu_j,\tilde{\matP}_1,\tilde{\matP}_2,\dots,\tilde{\matP}_d\Big),
    \end{equation}
    where $\proj_{\matu_k}=\matu_k^{}\matu_k^\top$, $\tilde{\matP}_k=-2\sym(\tilde{\matu}_k^{}\matu_k^\top)$ and $\tilde{\matu}_k=\matP_k(\nabla f(\tensX))_{(k)}\!\matv_k\matG_{(k)}^\top(2\matI_{r_k}+\matG_{(k)}^{}\matG_{(k)}^\top)^{-1}$.
\end{corollary}

A retraction $\retr:\tangent\!\tensM\to\tensM$ is a smooth mapping such that the curve $c(t)=\retr_x(t\dot{x})$ satisfies $c(0)=x$ and $c^\prime(0)=\dot{x}$ for all $(x,\dot{x})\in\tangent\!\tensM$, which is a basic operation of Riemannian methods to map a point to $\tensM$. We provide a retraction on~$\tensM$ via parameters.
\begin{proposition}
    Given $x=(\tensX,\matP_1,\matP_2,\dots,\matP_d)\in\tensM$, $\dot{x}=(\dot{\tensX},\dot{\matP}_1,\dot{\matP}_2,\dots,\dot{\matP}_d)\in\tangent_x\!\tensM$ represented by $({\tensG},{\matu}_1,{\matu}_2,\dots,{\matu}_d)$ and $(\dot{\tensG},\dot{\matu}_1,\dot{\matu}_2,\dots,\dot{\matu}_d)$ respectively, the mapping
    \begin{equation}
        \label{eq: desingularization retraction}
        \retr_x(\dot{x})=((\tensX+\dot{\tensX})\times_{k=1}^d(\bar{\matu}_k^{}\bar{\matu}_k^\top),\matI_{n_1}-\bar{\matu}_1^{}\bar{\matu}_1^\top,\matI_{n_2}-\bar{\matu}_2^{}\bar{\matu}_2^\top,\dots,\matI_{n_d}-\bar{\matu}_d^{}\bar{\matu}_d^\top)
    \end{equation} 
    defines a retraction on $\tensM$, where $\bar{\matu}_k=(\matu_k+\dot{\matu}_k)((\matu_k+\dot{\matu}_k)^\top(\matu_k+\dot{\matu}_k))^{-\frac12}$.
\end{proposition}
\begin{proof}
    We consider the smooth curve $\gamma(t)=(\gamma_0(t),\gamma_1(t),\gamma_2(t),\dots,\gamma_d(t))$ on $\tensM$, where $\gamma_0(t)=(\tensX+t\dot{\tensX})\times_{k=1}^d(\bar{\matu}_k(t)\bar{\matu}_k(t)^\top)$, $\gamma_k(t)=\matI_{n_k}-\bar{\matu}_k(t)\bar{\matu}_k(t)^\top$ with $\bar{\matu}_k(t)=(\matu_k+t\dot{\matu}_k)\matM_k(t)$ and $\matM_k(t)=((\matu_k+t\dot{\matu}_k)^\top(\matu_k+t\dot{\matu}_k))^{-\frac12}$. We aim to show that $\gamma(0)=x$ and $\gamma^\prime(0)=\dot{x}$, or equivalently $\gamma_0(0)=\tensX$, $\gamma_0^\prime(0)=\dot{\tensX}$, $\gamma_k(0)=\matP_k$ and $\gamma_k^\prime(0)=\dot{\matP}_k$ for $k\in[d]$.

    To this end, we firstly compute the derivative of $\matM_k(t)$ at $0$. We differentiate the equality $\matM_k(t)^2=((\matu_k+t\dot{\matu}_k)^\top(\matu_k+t\dot{\matu}_k))^{-1}$ and yield \[\matM_k(t)\matM_k^\prime(t)+\matM_k^\prime(t)\matM_k(t)=-2\matM_k(t)^2\sym(\dot{\matu}_k^\top(\matu_k+t\dot{\matu}_k))\matM_k(t)^2.\]
    It follows from $\matM_k(0)=\matI_{r_k}$ and $\dot{\matu}_k^\top\matu_k^{}=0$ that $\matM_k^\prime(0)=0$. Hence, we yield that $\bar{\matu}_k(0)=\matu_k$ and $\bar{\matu}_k^\prime(0)=\dot{\matu}_k$. Subsequently, by using the parametrization of the tangent space~\eqref{eq: desingularization tangent space}, we obtain that
    \begin{equation*}
        \begin{aligned}
            \gamma_0^\prime(0)&=\dot{\tensX}\times_{k=1}^d(\matu_k^{}\matu_k^\top)+\sum_{k=1}^{d}\tensX\times_k(\dot{\matu}_k^{}\matu_k^\top+\matu_k^{}\dot{\matu}_k^\top)\times_{j\neq k}(\matu_j^{}\matu_j^\top)\\
            &=\dot{\tensG}\times_{k=1}^d\matu_k+\sum_{k=1}^{d}\tensG\times_{k}\dot{\matu}_k\times_{j\neq k}\matu_j=\dot{\tensX},\\
            \gamma_k^\prime(0)&=-(\dot{\matu}_k^{}\matu_k^\top+\matu_k^{}\dot{\matu}_k^\top)=\dot{\matP}_k\quad\text{for}\quad k\in[d],
        \end{aligned}
    \end{equation*}
    i.e., $\gamma^\prime(0)=\dot{x}$. Therefore, $\retr_x(\dot{x})$ defines a retraction on $\tensM$. \hfill\squareforqed
\end{proof}

\subsection{Geometry of the desingularization: second order}
In order to facilitate second-order methods, we compute the Riemannian Hessian $\Hess\!g(x)$. It follows from~\cite[\S 5.11]{boumal2023intromanifolds} that
\[\Hess\!g(x)[\dot{x}]=\proj_x\Big(\tensP(\nabla f(\tensX),0,0,\dots,0)+(\nabla^2 f(\tensX)[\dot{\tensX}],0,0,\dots,0)\Big)\]
for $(x,\dot{x})\in\tangent\!\tensM$, where $\tensP:\tensE\to\tensE$ denotes the differential of the orthogonal projection defined by $\tensP(\barxi):=\mathrm{D}(x\mapsto\proj_x\!\barxi)(x)[\dot{x}]$ for $\barxi=(\tensA,\mata_1,\mata_2,\dots,\mata_d)\in\tensE$, and $\nabla^2 f(\tensX)$ is the Euclidean Hessian of $f$ at $\tensX$.

\begin{proposition}
Given $x=(\tensX,\matP_1,\matP_2,\dots,\matP_d)\in\tensM$, $\dot{x}=(\dot{\tensX},\dot{\matP}_1,\dot{\matP}_2,\dots,\dot{\matP}_d)\in\tangent_x\!\tensM$ represented by $({\tensG},{\matu}_1,{\matu}_2,\dots,{\matu}_d)$ and $(\dot{\tensG},\dot{\matu}_1,\dot{\matu}_2,\dots,\dot{\matu}_d)$, the Riemannian Hessian $\Hess\!g(x)[\dot{x}]\in\tangent_x\!\tensM$ can be represented by $(\hat{\tensG}+\breve{\tensG},\hat{\matu}_1+\breve{\matu}_1,\hat{\matu}_2+\breve{\matu}_2,\dots,\hat{\matu}_d+\breve{\matu}_d)$, where
\begin{equation}\label{eq: Riemannian Hessian}
    \begin{aligned}
        \hat{\tensG}&=-\sum_{k=1}^{d}\big(\nabla f(\tensX)\times_k\dot{\matu}_k^\top\times_{j\neq k}\matu_j^\top+\tensG\times_k(\dot{\matu}_k^\top(\nabla f(\tensX))_{(k)}\!\matv_k\matG_{(k)}^\top\tilde{\matF}_k^{-1})\big),\\
        \hat{\matu}_k&=\matP_k(\nabla f(\tensX))_{(k)}\!\left(\dot{\matv}_k\matG_{(k)}^\top+\matv_k\dot{\matG}_{(k)}^\top-\matv_k\matG_{(k)}^\top\tilde{\matF}_k^{-1}\matG_{(k)}\dot{\matG}_{(k)}^\top\right)\tilde{\matF}_k^{-1},\\
        \breve{\tensG}&=\nabla^2 f(\tensX)[\dot{\tensX}]\times_{k=1}^d\matu_k^\top,\\
        \breve{\matu}_k&=\matP_k(\nabla^2 f(\tensX)[\dot{\tensX}])_{(k)}\!\matv_k\matG_{(k)}^\top\tilde{\matF}_k^{-1}
    \end{aligned}
\end{equation}
with $\dot{\matv}_k\in\mathbb{R}^{n_{-k}\times r_{-k}}$ such that $\dot{\matx}_{(k)}=\dot{\matu}_k\matG_{(k)}\!\matv_k^\top+\matu_k\dot{\matG}_{(k)}\!\matv_k^\top+\matu_k\matG_{(k)}\!\dot{\matv}_k^\top$ and $\tilde{\matF}_k=2\matI_{r_k}+\matG_{(k)}^{}\matG_{(k)}^\top$ for $k\in[d]$.
\end{proposition}
\begin{proof}
    Since $\Hess\!g(x)[\dot{x}]$ is a tangent vector in $\tangent_x\!\tensM$, it suffices to compute the parameters $(\hat{\tensG},\hat{\matu}_1,\hat{\matu}_2,\dots,\hat{\matu}_d)$ and $(\breve{\tensG},\breve{\matu}_1,\breve{\matu}_2,\dots,\breve{\matu}_d)$ for two tangent vectors $\proj_x(\tensP(\nabla f(\tensX),0,0,\dots,0))$ and $\proj_x(\nabla^2 f(\tensX)[\dot{\tensX}],0,0,\dots,0)$ respectively. The Riemannian Hessian $\Hess\!f(x)[\dot{x}]$ is computed by following steps.

    \paragraph{Step 1: rewrite the projection operator}
    We first express the projection $\proj_x\!\barxi$ by parameters $(\tensX,\matP_1,\matP_2,\dots,\matP_d)$ for $\barxi=(\tensA,\mata_1,\dots,\mata_d)$. Denote $\matF_k=2\matI_{n_k}+\matX_{(k)}^{}\matX_{(k)}^\top$ and $\tilde{\matF}_k=2\matI_{r_k}+\matG_{(k)}^{}\matG_{(k)}^\top$ for $k\in[d]$. We observe that $\matF_k^{-1}=\matu_k^{}\tilde{\matF}_k\matu_k^\top+\frac12\matP_k$ and $\tilde{\matF}_k^{-1}=\matu_k^\top\matF_k^{-1}\matu_k$. Subsequently, we can reformulate the parameters $\tilde{\matu}_k$ in Proposition~\ref{prop: proj onto tangent}:
        \begin{align}
            \tilde{\matu}_k&=\matP_k(\mata_{(k)}^{}\!\matv_k^{}\matG_{(k)}^\top-(\mata_k+\mata_k^\top)\matu_k)(2\matI_{r_k}+\matG_{(k)}^{}\matG_{(k)}^\top)^{-1}\nonumber\\
            &=\matP_k(\mata_{(k)}^{}\!\matv_k^{}\matG_{(k)}^\top-(\mata_k+\mata_k^\top)\matu_k)\matu_k^\top\matF_k^{-1}\matu_k\nonumber\\
            &=\matP_k\matE_k\matF_k^{-1}\matu_k,\label{eq: rewrite tildeUk}
        \end{align}
    where $\matE_k=\mata_{(k)}^{}\matx_{(k)}^\top-(\mata_k+\mata_k^\top)(\matI_{n_k}-\matP_k)$. 
    Therefore, we obtain that 
    \begin{equation*}
        \begin{aligned}
            \proj_x\!\barxi=\Big(&\tensA\times_{k=1}^d(\matI_{n_k}-\matP_k)+\sum_{k=1}^{d}\tensX\times_k(\matP_k\matE_k\matF_k^{-1}),\\
            &-2\sym(\matP_1\matE_1\matF_1^{-1}(\matI_{n_1}-\matP_1)),\\
            &-2\sym(\matP_2\matE_2\matF_2^{-1}(\matI_{n_2}-\matP_2)),\\
            &\qquad\qquad\qquad\vdots\\
            &-2\sym(\matP_d\matE_d\matF_d^{-1}(\matI_{n_d}-\matP_d))\Big),
        \end{aligned}
    \end{equation*} 
    which only depends on the parameters $(\tensX,\matP_1,\matP_2,\dots,\matP_d)$.
    
    \paragraph{Step 2: compute the differential of the orthogonal projection}
    The differential of the orthogonal projection $\tensP(\barxi)=\mathrm{D}(x\mapsto\proj_x\!\barxi)(x)[\dot{x}]$ is
    \begin{equation*}
        \begin{aligned}
            \tensP(\barxi)=\Big(&\sum_{k=1}^{d}(-\tensA\times_k\dot{\matP}_k\times_{j\neq k}(\matI_{n_j}-\matP_j)+\dot{\tensX}\times_k\matc_k+\tensX\times_k\dot{\matc}_k),\\
            &-2\sym(\dot{\matc}_1(\matI_{n_1}-\matP_1)-\matc_1\dot{\matP}_1),\\
            &-2\sym(\dot{\matc}_2(\matI_{n_2}-\matP_2)-\matc_2\dot{\matP}_2),\\
            &\qquad\qquad\qquad\vdots\\
            &-2\sym(\dot{\matc}_d(\matI_{n_d}-\matP_d)-\matc_d\dot{\matP}_d)\Big).
        \end{aligned}
    \end{equation*}
    where $\matc_k=\matP_k\matE_k\matF_k^{-1}$, $\dot{\matc}_k=(\dot{\matP}_k\matE_k+\matP_k\dot{\matE}_k-\matP_k\matE_k\matF_k^{-1}\dot{\matF}_k)\matF_k^{-1}$ is the differential of $\matc_k$, $\dot{\matE}_k=\mata_{(k)}^{}\dot{\matx}_{(k)}^\top+(\mata_k+\mata_k^\top)\dot{\matP}_k$ and $\dot{\matF}_k=\dot{\matX}_{(k)}^{}\matX_{(k)}^\top+\matX_{(k)}^{}\dot{\matX}_{(k)}^\top=2\sym((\dot{\matu}_k\matG_{(k)}+\matu_k\dot{\matG}_{(k)})\matG_{(k)}^\top\matu_k^\top)$.

    \paragraph{Step 3: compute parameters}
    Denote $\tensP(\barxi)=(\tensB,\matb_1,\matb_2,\dots,\matb_d)$. In order to yield parametrizations of $\proj_x(\tensP(\nabla f(\tensX),0,0,\dots,0))$ and $\proj_x(\nabla^2 f(\tensX)[\dot{\tensX}],0,0,\dots,0)$, we set $\barxi=(\nabla f(\tensX),0,0,\dots,0)$. Then, we can simplify the matrices 
    \[\matE_k=\mata_{(k)}\matx_{(k)}^\top,\ \dot{\matE}_k=\mata_{(k)}^{}\dot{\matx}_{(k)}^\top,\ \text{and}\ \matc_k=\matP_k\mata_{(k)}\!\matv_k\matG_{(k)}^\top\tilde{\matF}_k^{-1}\matu_k^\top\] 
    with $\tensA=\nabla f(\tensX)$. We start with the parameters $(\hat{\tensG},\hat{\matu}_1,\hat{\matu}_2,\dots,\hat{\matu}_d)$ for the tangent vector $\proj_x(\tensP(\nabla f(\tensX),0,0,\dots,0))$, and we obtain that 
    \begin{equation*}
        \begin{aligned}
            \hat{\tensG}&=\tensB\times_{k=1}^d\matu_k^\top\\
            &=\sum_{k=1}^{d}\big(-\tensA\times_k\dot{\matu}_k^\top\times_{j\neq k}\matu_j^\top+(\dot{\tensX}\times_k(\matu_k^\top\matc_k)+\tensX\times_k(\matu_k^\top\dot{\matc}_k))\times_{j\neq k}\matu_j^\top\big)\\
            &=\sum_{k=1}^{d}\big(-\tensA\times_k\dot{\matu}_k^\top\times_{j\neq k}\matu_j^\top+\tensX\times_k(\matu_k^\top\dot{\matc}_k)\times_{j\neq k}\matu_j^\top\big)\\
            &=-\sum_{k=1}^{d}\big(\tensA\times_k\dot{\matu}_k^\top\times_{j\neq k}\matu_j^\top+\tensG\times_k(\dot{\matu}_k^\top\mata_{(k)}\!\matv_k\matG_{(k)}^\top\tilde{\matF}_k^{-1})\big),
        \end{aligned}
    \end{equation*}
    where we uses the facts that $\matu_k^\top\matc_k=0$ and $\matu_k^\top\matF_k^{-1}=\tilde{\matF}_k^{-1}\matu_k^\top$. For parameter~$\hat{\matu}_k$, it follows from~\eqref{eq: rewrite tildeUk} that 
    \begin{equation*}
        \begin{aligned}
            \hat{\matu}_k&=\matP_k\big(\matb_{(k)}^{}\matx_{(k)}^\top-(\matb_k+\matb_k^\top)(\matI_{n_k}-\matP_k)\big)\matF_k^{-1}\matu_k\\
            &=\matP_k\matb_{(k)}^{}\matx_{(k)}^\top\matu_k\tilde{\matF}_k^{-1}-2\matP_k\matb_k(\matI_{n_k}-\matP_k)\matu_k\tilde{\matF}_k^{-1}.
        \end{aligned}
    \end{equation*}
    We first compute the term 
    \begin{equation*}
        \begin{aligned}
            \matP_k\matb_k(\matI_{n_k}-\matP_k)\matu_k&=\matP_k\matb_k\matu_k=-\matP_k\dot{\matc}_k\matu_k+\matc_k\dot{\matP}_k\matu_k\\
            &=-\matP_k\dot{\matc}_k\matu_k-\matc_k\dot{\matu}_k\\
            &=-\matP_k\dot{\matc}_k\matu_k,
        \end{aligned}
    \end{equation*}
    where we use the facts that $\matP_k\dot{\matP}_k=\dot{\matP}_k(\matI_{n_k}-\matP_k)=-\dot{\matu}_k^{}\matu_k^\top$, $\matc_k^\top\matu_k^{}=0$, and $\matc_k\dot{\matu}_k=0$. Then, we have
    \begin{equation*}
        \begin{aligned}
            \matP_k\dot{\matc}_k\matu_k&=-\dot{\matu}_k{}\matu_k^\top\mata_{(k)}\!\matv_k\matG_{(k)}^\top\tilde{\matF}_k^{-1}+\matP_k\mata_{(k)}(\dot{\matv}_k\matG_{(k)}^\top+\matv_k\dot{\matG}_{(k)}^\top)\tilde{\matF}_k^{-1}\\
            &~~~-2\matP_k\mata_{(k)}\!\matv_k\matG_{(k)}^\top\tilde{\matF}_k^{-1}\sym(\dot{\matG}_{(k)}\matG_{(k)}^\top)\tilde{\matF}_k^{-1},
        \end{aligned}
    \end{equation*}
    where $\dot{\matv}_k\in\mathbb{R}^{n_{-k}\times r_{-k}}$ such that $\dot{\matx}_{(k)}=\dot{\matu}_k\matG_{(k)}\matv_k^\top+\matu_k\dot{\matG}_{(k)}\matv_k^\top+\matu_k\matG_{(k)}\dot{\matv}_k^\top$. 
    Subsequently, we compute the other term
    \begin{equation*}
        \begin{aligned}
            \matP_k\matb_{(k)}^{}\matx_{(k)}^\top\matu_k&=\matP_k\matb_{(k)}^{}\!\matv_k\matG_{(k)}^\top\\
            &=\big(\dot{\matu}_k^{}\matu_k^\top\mata_{(k)}\!\matv_k+\matc_k(\dot{\matu}_k\matG_{(k)}+\matu_k\dot{\matG}_{(k)})+\matP_k\dot{\matc}_k\matu_k\matG_{(k)}\big)\matG_{(k)}^\top\\
            &=\big(\dot{\matu}_k^{}\matu_k^\top\mata_{(k)}\!\matv_k+\matc_k\matu_k\dot{\matG}_{(k)}+\matP_k\dot{\matc}_k\matu_k\matG_{(k)}\big)\matG_{(k)}^\top\\ 
            &=\big(\dot{\matu}_k^{}\matu_k^\top\mata_{(k)}\!\matv_k+\matP_k\mata_{(k)}\!\matv_k\matG_{(k)}^\top\tilde{\matF}_k^{-1}\dot{\matG}_{(k)}+\matP_k\dot{\matc}_k\matu_k\matG_{(k)}\big)\matG_{(k)}^\top,
        \end{aligned}
    \end{equation*}
    where we use the facts that $\matP_k\dot{\matP}_k=\dot{\matP}_k(\matI_{n_k}-\matP_k)=-\dot{\matu}_k^{}\matu_k^\top$, $\matc_k\dot{\matu}_k=0$, and $\tensX\times_j\dot{\matc}_j\times_k\matP_k=0$ if $j\neq k$. Note that $\matb_{(k)}$ is the mode-$k$ unfolding of $\tensB$, which is different from $\matb_k$. Therefore, we obtain that 
    \begin{equation*}
        \begin{aligned}
            \hat{\matu}_k&=\matP_k\matb_{(k)}^{}\matx_{(k)}^\top\matu_k\tilde{\matF}_k^{-1}+2\matP_k\dot{\matc}_k\matu_k\tilde{\matF}_k^{-1}\\
            &=\big(\dot{\matu}_k^{}\matu_k^\top\mata_{(k)}\!\matv_k+\matP_k\mata_{(k)}\!\matv_k\matG_{(k)}^\top\tilde{\matF}_k^{-1}\dot{\matG}_{(k)}\big)\matG_{(k)}^\top\tilde{\matF}_k^{-1}+\matP_k\dot{\matc}_k\matu_k\\
            &=\matP_k\mata_{(k)}(\dot{\matv}_k\matG_{(k)}^\top+\matv_k\dot{\matG}_{(k)}^\top)\tilde{\matF}_k^{-1}-\matP_k\mata_{(k)}\!\matv_k\matG_{(k)}^\top\tilde{\matF}_k^{-1}\matG_{(k)}\dot{\matG}_{(k)}^\top\tilde{\matF}_k^{-1}.
        \end{aligned}
    \end{equation*}
    
    Next, we compute the parameters $(\breve{\tensG},\breve{\matu}_1,\breve{\matu}_2,\dots,\breve{\matu}_d)$ of the tangent vector $\proj_x(\nabla^2 f(\tensX)[\dot{\tensX}],0,0,\dots,0)$. It follows from Proposition~\ref{prop: proj onto tangent} that 
    \begin{equation*}
        \begin{aligned}
            \breve{\tensG}&=\nabla^2 f(\tensX)[\dot{\tensX}]\times_{k=1}^d\matu_k^\top,\\
            \breve{\matu}_k&=\matP_k(\nabla^2 f(x)[\dot{x}])_{(k)}\matv_k\matG_{(k)}^\top\tilde{\matF}_k^{-1}.
        \end{aligned}
    \end{equation*}
    Consequently, since $\hat{\matu}_k^\top\matu_k^{}=\breve{\matu}_k^\top\matu_k^{}=0$, the Riemannian Hessian $\Hess\!g(x)$ can be viewed as an operator mapping the tangent vector $\dot{x}\in\tangent_x\!\tensM$ represented by $(\dot{\tensG},\dot{\matu}_1,\dot{\matu}_2,\dots,\dot{\matu}_d)$ to a tangent vector represented by $(\hat{\tensG}+\breve{\tensG},\hat{\matu}_1+\breve{\matu}_1,\hat{\matu}_2+\breve{\matu}_2,\dots,\hat{\matu}_d+\breve{\matu}_d)$. \hfill\squareforqed    
\end{proof}

\subsection{Connection to existing geometries}\label{subsec: connection}
We illustrate the connection of $\tensM$ to the geometries of matrix variety $\mathbb{R}^{m\times n}_{\leq r}$, desingularization $\tensM(\mathbb{R}^{m\times n},r)$ of matrix variety, and Tucker tensor variety $\mathbb{R}^{n_1\times\cdots\times n_d}_{\leq\vecr}$. 

\paragraph{Connection to matrix varieties}
Given $\barxi=(\tensA,\mata_1,\mata_2,\dots,\mata_d)\in\tensE$, we introduce the mappings $\proj_1,\proj_2,\dots,\proj_d$ defined by
\[\proj_k(\barxi):=(\mata_{(k)}^\top,\mata_k)\]
for $k\in[d]$. By using
\begin{equation*}
    \begin{array}{rclllllcl}
        x & = & (\tensX, & \matP_1, & \matP_2, & \dots, & \matP_d) & \in & \tensM,\vspace{2mm}\\
        \varphi(x) & = & ~\tensX & & & & & \in & \mathbb{R}_{\leq\vecr}^{n_1\times n_2\times \cdots\times n_d},\vspace{2mm}\\
        \proj_1(x) & = & (\matx_{(1)}^\top, & \matP_1) & & & & \in & \tensM(\mathbb{R}^{n_{-1}\times n_1},r_1), \vspace{2mm}\\
        \proj_2(x) & = & (\matx_{(2)}^\top, & & \matP_2) & & & \in & \tensM(\mathbb{R}^{n_{-2}\times n_2},r_2), \\
        \multicolumn{7}{c}{\vdots}\vspace{2mm}\\
        \proj_d(x) & = & (\matx_{(d)}^\top, & & & & \matP_d) & \in & \tensM(\mathbb{R}^{n_{-d}\times n_d},r_d),
    \end{array}
\end{equation*}
we observe that the desingularization $\tensM$ can be interpreted by the intersection of tensorized desingularization of matrix varieties
\[\tensM=\{x\in\tensE:\proj_k(x)\in\tensM(\mathbb{R}^{n_{-k}\times n_k},r_k),k\in[d]\}.\]
Note that in the matrix case ($d=2$), $\tensM(\mathbb{R}^{n_1\times n_2},\vecr)$ in~\eqref{eq: desingularization of Tucker} is essentially different from $\tensM(\mathbb{R}^{n_1\times n_2},r)$ in~\eqref{eq: matrix desing} since $\vecr=(r_1,r_2)$ is an array. In other words, we also propose a new parametrization for matrix varieties, which differs from the one in~\cite{khrulkov2018desingularization}.

\paragraph{Connection to tensor varieties}
Given $x=(\tensX,\matP_1,\matP_2,\dots,\matP_d)\in\tensM$, we show the connection between the first element of tangent vectors in $\tangent_x\!\tensM$ and the tangent cone of $\mathbb{R}^{n_1\times n_2\times\cdots\times n_d}_{\leq\vecr}$ at~$\tensX$. Note that the mappings $\varphi,\proj_1,\dots,\proj_d$ are also well-defined on the tangent space~$\tangent_x\!\tensM$.
\begin{proposition}
    Given $x=(\tensX,\matP_1,\matP_2,\dots,\matP_d)\in\tensM$, it holds that
    \begin{equation}
        \label{eq: inclusion of tangent space and tangent cone}
        \varphi(\tangent_x\!\tensM)\subseteq\tangent_{\tensX}\!\mathbb{R}^{n_1\times n_2\times\cdots\times n_d}_{\leq\vecr}.
    \end{equation}
    Moreover, for a fixed $\tensX\in\mathbb{R}^{n_1\times n_2\times\cdots\times n_d}_{\leq\vecr}$, it holds that
    \[\bigcup_{x\in\tensM\cap\varphi^{-1}(\tensX)}\varphi(\tangent_x\!\tensM)=\tangent_{\tensX}\!\mathbb{R}^{n_1\times n_2\times\cdots\times n_d}_{\leq\vecr}.\]
\end{proposition}
\begin{proof}
    Denote $\underline{\vecr}=\ranktc(\tensX)\leq\vecr$ and recall the parametrization $(\tensG,\matu_1,\matu_2,\dots,\matu_d)$ for $x$.    Given $\tensV\in\varphi(\tangent_x\!\tensM)$, it can be represented by $\tensV=\dot{\tensG}\times_{k=1}^d\matu_k+\sum_{k=1}^d\tensG\times_k\dot{\matu}_k\times_{j\neq k}\matu_j$ with $\dot{\matu}_k^\top\matu_k^{}=0$. In order to reformulate $\tensV$ to the form of~\eqref{eq: Tucker tangent cone}, we consider the thin Tucker decomposition $\tensX=\underline{\tensG}\times_{k=1}^d\underline{\matu}_k$ of $\tensX$, where $\underline{\tensG}\in\mathbb{R}^{\underline{r}_1\times\underline{r}_2\times\cdots\times\underline{r}_d}$ and $\underline{\matu}_k\in\St(\underline{r}_k,n_k)$. Since $\tensX\in\otimes_{k=1}^d\Span(\matu_k)$, there exist $\matu_{k,1}\in\St(r_k-\underline{r}_k,n_k)$ and $\matQ_k\in\mathcal{O}(r_k)$, such that $\matu_k=[\underline{\matu}_k\ \matu_{k,1}]\matQ_k$ for $k\in[d]$. Therefore, it holds that 
    \begin{equation*}
        \begin{aligned}
            \tensV&=(\dot{\tensG}\times_{k=1}^d\matQ_k)\times_{k=1}^d[\underline{\matu}_k\ \matu_{k,1}]+\sum_{k=1}^d\tensG\times_k\dot{\matu}_k\times_{j\neq k}\matu_j\\
            &=(\dot{\tensG}\times_{k=1}^d\matQ_k)\times_{k=1}^d[\underline{\matu}_k\ \matu_{k,1}]+\sum_{k=1}^d\underline{\tensG}\times_k(\dot{\matu}_k^{}\matu_k^\top\underline{\matu}_k^{})\times_{j\neq k}\underline{\matu}_j,
        \end{aligned}
    \end{equation*}
    which is in the form of~\eqref{eq: Tucker tangent cone}. 
    
    Subsequently, for all $\tensV\in\tangent_{\tensX}\!\mathbb{R}^{n_1\times n_2\times\cdots\times n_d}_{\leq\vecr}$, we aim to show that there exists $\tilde{\matP}_k\in\Gr(n_k-r_k,n_k)$, such that $\tilde{x}:=(\tensX,\tilde{\matP}_1,\tilde{\matP}_2,\dots,\tilde{\matP}_d)\in\tensM$ and $\tensV\in\varphi(\tangent_{\tilde{x}}\!\tensM)$. Consider the thin Tucker decomposition $\tensX=\underline{\tensG}\times_{k=1}^d\underline{\matu}_k$. The vector $\tensV$ can be represented by 
    \begin{equation*}
        \begin{aligned}
            \tensV&=\tensC\times_{k=1}^d[\underline{\matu}_k\ \matu_{k,1}]+\sum_{k=1}^d\underline{\tensG}\times_k({\matu}_{k,2}\matR_{k,2})\times_{j\neq k}\underline{\matu}_j\\
            &=\tensC\times_{k=1}^d[\underline{\matu}_k\ \matu_{k,1}]+\sum_{k=1}^d\tilde{\tensG}\times_k[{\matu}_{k,2}\matR_{k,2}\ 0]\times_{j\neq k}[\underline{\matu}_j\ \matu_{j,1}]\\
            &=\tensC\times_{k=1}^d\tilde{\matu}_k+\sum_{k=1}^d\tilde{\tensG}\times_k\dot{\matu}_k\times_{j\neq k}\tilde{\matu}_j,
        \end{aligned}
    \end{equation*}
    where $\tilde{\tensG}\in\mathbb{R}^{r_1\times r_2\times\cdots\times r_d}$ satisfies $\tilde{\tensG}(i_1,i_2,\dots,i_d)=\underline{\tensG}(i_1,i_2,\dots,i_d)$ if $i_k\in[\underline{r}_k]$ and $\tilde{\tensG}(i_1,i_2,\dots,i_d)=0$ otherwise, $\dot{\matu}_k=[{\matu}_{k,2}\matR_{k,2}\ 0]$ and $\tilde{\matu}_k=[\underline{\matu}_k\ \matu_{k,1}]$. Let $\tilde{\matP}_k=\matI_{n_k}-\tilde{\matu}_k^{}\tilde{\matu}_k^\top$. Since $\tensX=\tilde{\tensG}\times_{k=1}^d\tilde{\matu}_k$ and $\dot{\matu}_k^\top\tilde{\matu}_k^{}=0$, we have $\tilde{x}\in\tensM$ and $\tensV\in\varphi(\tangent_{\tilde{x}}\!\tensM)$.  \hfill\squareforqed
\end{proof}

It is worth noting that on the one hand, the inclusion in~\eqref{eq: inclusion of tangent space and tangent cone} can be strict. For instance, let $\tensG=0$ and $\matu_k\in\St(r_k,n_k)$ be arbitrary. We observe that $\varphi(\tangent_x\!\tensM)=\otimes_{k=1}^d\Span(\matu_k)\subsetneq\tangent_{\tensX}\!\mathbb{R}^{n_1\times n_2\times\cdots\times n_d}_{\leq\vecr}$. On the other hand, the inclusion can be equality if $\tensX\in\mathbb{R}^{n_1\times n_2\times\cdots\times n_d}_{\vecr}$, i.e.,
\[\varphi(\tangent_x\!\tensM)=\tangent_{\tensX}\!\mathbb{R}^{n_1\times n_2\times\cdots\times n_d}_{\vecr}.\]
The rationale behind this is that $\varphi^{-1}(\tensX)$ is unique for all $\tensX\in\mathbb{R}^{n_1\times n_2\times\cdots\times n_d}_{\vecr}$.

We show that the tangent space $\tangent_x\!\tensM$ in~\eqref{eq: desingularization tangent space} can be constructed through the tensorized tangent spaces of desingularization of matrix varieties in~\eqref{eq: tangent space for matrix desing}. 
\begin{proposition}
    Given $x=(\tensX,\matP_1,\dots,\matP_d)\in\tensM$ and parameters $(\tensG,\matu_1,\dots,\matu_d)$, the tangent space $\tangent_x\!\tensM$ can be expressed by
    \[\tangent_x\!\tensM=\{\dot{x}\in\tensE:\proj_k(\dot{x})\in\tangent_{\proj_k(x)}\!\tensM(\mathbb{R}^{n_{-k}\times n_k},r_k),\ k\in[d]\}.\]
\end{proposition}
\begin{proof}
    The proof sketch is similar to~\cite[Theorem 1]{gao2023low}. On the one hand, given $\dot{x}\in\tangent_x\!\tensM$, we observe that $\proj_k(\dot{x})=(\matv_k^{}\matG_{(k)}^\top\dot{\matu}_k^\top+\tilde{\matR}_k^{}\matu_k^\top,-2\sym(\dot{\matu}_k\matu_k^\top))$ for some $\tilde{\matR}_k\in\mathbb{R}^{n_{-k}\times r_k}$. Therefore, $\proj_k(\dot{x})\in\tangent_{\proj_k(x)}\!\tensM(\mathbb{R}^{n_{-k}\times n_k},r_k)$ by~\eqref{eq: tangent space for matrix desing}. 

    On the other hand, given $\dot{x}\in\tensE$ satisfying $\proj_k(\dot{x})\in\tangent_{\proj_k(x)}\!\tensM(\mathbb{R}^{n_{-k}\times n_k},r_k)$ for all $k\in[d]$, we aim to show that $\dot{x}\in\tangent_x\!\tensM$. We decompose $\matx_{(k)}$ by
    \[\matx_{(k)}=\matu_k\matG_{(k)}\!\matv_k^\top=\matu_k\tilde{\matu}_k\Sigma_k\tilde{\matv}_k^\top\matv_k^\top,\] 
    where $\matG_{(k)}=\tilde{\matu}_k\Sigma_k\tilde{\matv}_k^\top$ is a SVD with $\tilde{\matu}_k\in\mathcal{O}(r_k)$, $\Sigma_k\in\mathbb{R}^{r_k\times r_k}$ and $\tilde{\matv}_k\in\St(r_k,r_{-k})$. Therefore, the tangent vector $\dot{\matX}_{(k)}^\top$ can be represented by 
        \begin{align}
            \dot{\matX}_{(k)}^\top&=\matR_k^{}(\matu_k\tilde{\matu}_k)^\top+(\matv_k\tilde{\matv}_k)\Sigma_k\dot{\matu}_k^\top\nonumber\\
            &=\matR_k^{}\tilde{\matu}_k^\top\matu_k^\top+\matv_k\tilde{\matv}_k\Sigma_k\tilde{\matu}_k^\top\tilde{\matu}_k^{}\dot{\matu}_k^\top\nonumber\\
            &=\matR_k^{}\tilde{\matu}_k^\top\matu_k^\top+\matv_k\matG_{(k)}^\top\tilde{\matu}_k^{}\dot{\matu}_k^\top\label{eq: representation in matrix}
        \end{align}
    with some $\matR_k\in\mathbb{R}^{n_{-k}\times r_k}$ and $\dot{\matu}_k\in\mathbb{R}^{n_k\times r_k}$ satisfying $\dot{\matu}_k^\top\matu_k^{}=0$. It follows from~\eqref{eq: representation in matrix} that $\dot{\matX}_{(k)}-\dot{\matu}_k^{}\tilde{\matu}_k^\top\matG_{(k)}^{}\!\matv_k^\top=\matu_k^{}\tilde{\matu}_k^{}\matR_k^\top$, i.e., 
    \[\Span(\dot{\matX}_{(k)}-\dot{\matu}_k^{}\tilde{\matu}_k^\top\matG_{(k)}^{}\matv_k^\top)\subseteq\Span(\matu_k).\] 
    Since $\ten_{(k)}(\dot{\matu}_k^{}\tilde{\matu}_k^\top\matG_{(k)}^{}\!\matv_k^\top)=\tensG\times_k(\dot{\matu}_k^{}\tilde{\matu}_k^\top)\times_{j\neq k}\matu_j$, we obtain that 
    \[\dot{\tensX}-\sum_{k=1}^d\tensG\times_k(\dot{\matu}_k^{}\tilde{\matu}_k^\top)\times_{j\neq k}\matu_j\in\bigotimes_{k=1}^d\Span(\matu_k)\]
    by mode-$k$ unfolding matrices, i.e., there exists $\dot{\tensG}\in\mathbb{R}^{r_1\times r_2\times\cdots\times r_d}$, such that
    \[\dot{\tensX}=\dot{\tensG}\times_{k=1}^d\matu_k+\sum_{k=1}^d\tensG\times_k(\dot{\matu}_k^{}\tilde{\matu}_k^\top)\times_{j\neq k}\matu_j.\]
    Consequently, $\dot{x}\in\tangent_x\!\tensM$ can be represented by $(\dot{\tensG},\dot{\matu}_1^{}\tilde{\matu}_1^\top,\dot{\matu}_2^{}\tilde{\matu}_2^\top,\dots,\dot{\matu}_d^{}\tilde{\matu}_d^\top)$. \hfill\squareforqed
\end{proof}

\section{Optimization via desingularization: methods and stationary points}\label{sec: optim desing}
By using the desingularization of Tucker tensor varieties in section~\ref{sec: desing}, we recast the problem~\eqref{eq: problem (P)} to the following parametrized problem on the smooth manifold $\tensM=\tensM(\mathbb{R}^{n_1\times n_2\times\cdots\times n_d},\vecr)$,
\begin{equation}
    \label{eq: problem (Q-desing)}
    \tag{Q-desing}
    \min_{x} g(x)=f(\varphi(x)),\quad\subjectto\ x\in\tensM.
\end{equation} 
We refer to~\eqref{eq: problem (Q-desing)} as a desingularized problem. Since $\tensM$ enjoys manifold structure, one can adopt the Riemannian methods to solve~\eqref{eq: problem (Q-desing)}; see~\cite{absil2009optimization,boumal2023intromanifolds} for an overview. When the problem reduces to the matrix case, i.e., $d=2$, \eqref{eq: problem (Q-desing)} is a new parametrization for optimization on matrix varieties since $\tensM(\mathbb{R}^{n_1\times n_2},\vecr)$ in~\eqref{eq: desingularization of Tucker} is essentially different from $\tensM(\mathbb{R}^{n_1\times n_2},r)$ in~\eqref{eq: matrix desing}.

In this section, we propose Riemannian methods for solving~\eqref{eq: problem (Q-desing)}, and provide efficient implementations for these methods. Additionally, we analyze the optimality conditions of~\eqref{eq: problem (Q-desing)}, and investigate the connection of stationary points of the low-rank Tucker tensor optimization problems~\eqref{eq: problem (P)} and~\eqref{eq: problem (Q-desing)}. We demonstrate that $\varphi(x)$ is not necessarily to be stationary of~\eqref{eq: problem (P)} even if $x$ is a second-order stationary point of~\eqref{eq: problem (Q-desing)}. We also consider a parametrization using Tucker decomposition, where the Tucker tensor varieties is parametrized by
\begin{equation}\label{eq: Tucker parametrization}
    \begin{aligned}
        &\tensM^\mathrm{Tucker}=\mathbb{R}^{r_1\times r_2\times\cdots\times r_d}\times\mathbb{R}^{n_1\times r_1}\times\mathbb{R}^{n_2\times r_2}\times\cdots\times\mathbb{R}^{n_d\times r_d},\\
        \tilde{\varphi}:~&\tensM^\mathrm{Tucker}\to\mathbb{R}^{n_1\times n_2\times\dots\times n_d}_{\leq\vecr},\qquad\tilde{\varphi}(\tilde{\tensG},\tilde{\matu}_1,\tilde{\matu}_2,\dots,\tilde{\matu}_d)=\tilde{\tensG}\times_{k=1}^d\tilde{\matu}_k.
    \end{aligned}
\end{equation}
We illustrate the connection of the sets of first- and second-order stationary points of~\eqref{eq: problem (Q)} under two parametrizations: desingularization and Tucker parametrization.

\subsection{Optimization methods via desingularization}
Riemannian gradient descent method, Riemannian conjugate gradient method and Riemannian trust-region method for~\eqref{eq: problem (Q-desing)} are listed in Algorithms~\ref{alg: RGD-desing}--\ref{alg: RTR-desing} respectively. In contrast with the manifold of fixed-rank Tucker tensors $\mathbb{R}^{n_1\times \cdots\times n_d}_{\vecr}$, the proposed manifold $\tensM$ is closed. Therefore, the standard convergence analysis of Riemannian methods is applicable, and interested readers are referred to~\cite{boumal2019global,sato2022riemannian} for details.

\begin{algorithm}[htbp]
    \caption{Riemannian gradient descent method for desingularization (RGD-desing)}
    \label{alg: RGD-desing}
    \begin{algorithmic}[1]
        \REQUIRE Initial guess $x^{(0)}\in\tensM$ with representation $(\tensG^{(0)},\matu_1^{(0)},\dots,\matu_d^{(0)})$, $t=0$.
        \WHILE{the stopping criteria are not satisfied}
        \STATE Compute parameters $(\dot{\tensG}^{(t)},\dot{\matu}_1^{(t)},\dots,\dot{\matu}_d^{(t)})$ of $\eta^{(t)}=-\grad f(x^{(t)})$ by~\eqref{eq: Riemannian gradient desing}.
        \STATE Choose stepsize $s^{(t)}$.
        \STATE Update the parameters $({\tensG}^{(t+1)},{\matu}_1^{(t+1)},\dots,{\matu}_d^{(t+1)})$ of $x^{(t+1)}=\retr_{x^{(t)}}(s^{(t)}\eta^{(t)})$ by~\eqref{eq: desingularization retraction}; $t=t+1$.
        \ENDWHILE
        \ENSURE The parameters $({\tensG}^{(t)},{\matu}_1^{(t)},\dots,{\matu}_d^{(t)})$ of $x^{(t)}$.
    \end{algorithmic}
\end{algorithm}

\begin{algorithm}[htbp]
    \caption{Riemannian conjugate gradient method for desingularization (RCG-desing)}
    \label{alg: RCG-desing}
    \begin{algorithmic}[1]
        \REQUIRE Initial guess $x^{(0)}\in\tensM$ with representation $(\tensG^{(0)},\matu_1^{(0)},\dots,\matu_d^{(0)})$, $t=0$, $\beta^{(0)}=0$.
        \WHILE{the stopping criteria are not satisfied}
        \STATE Compute parameters $(\dot{\tensG}^{(t)},\dot{\matu}_1^{(t)},\dots,\dot{\matu}_d^{(t)})$ of $\eta^{(t)}=-\grad f(x^{(t)})+\beta^{(t)}\tensT_{t\gets t-1}\eta^{(t-1)}$ by~\eqref{eq: projection onto tangent} and~\eqref{eq: Riemannian gradient desing}, where $\beta^{(t)}$ is a conjugate gradient parameter and $\tensT_{t\gets t-1}:\tangent_{x^{(t-1)}}\!\tensM\to\tangent_{x^{(t)}}\!\tensM$ is a vector transport. 
        \STATE Choose stepsize $s^{(t)}$.
        \STATE Update the parameters $({\tensG}^{(t+1)},{\matu}_1^{(t+1)},\dots,{\matu}_d^{(t+1)})$ of $x^{(t+1)}=\retr_{x^{(t)}}(s^{(t)}\eta^{(t)})$ by~\eqref{eq: desingularization retraction}; $t=t+1$.
        \ENDWHILE
        \ENSURE The parameters $({\tensG}^{(t)},{\matu}_1^{(t)},\dots,{\matu}_d^{(t)})$ of $x^{(t)}$.
    \end{algorithmic}
\end{algorithm}

\begin{algorithm}[htbp]
    \caption{Riemannian trust-region method for desingularization (RTR-desing)}
    \label{alg: RTR-desing}
    \begin{algorithmic}[1]
        \REQUIRE Initial guess $x^{(0)}\in\tensM$ with representation $(\tensG^{(0)},\matu_1^{(0)},\dots,\matu_d^{(0)})$, $t=0$, $\Delta^{(0)}$, $\rho^\prime\in(0,1/4)$, maximum radius $\bar{\Delta}>0$.
        \WHILE{the stopping criteria are not satisfied}
        \STATE Compute parameters $(\dot{\tensG}^{(t)},\dot{\matu}_1^{(t)},\dots,\dot{\matu}_d^{(t)})$ of $\eta^{(t)}$ by (approximately) solving
        \begin{equation}
            \label{eq: RTR subproblem}\min_{\eta\in\tangent_{x^{(t)}}\!\tensM}\ \ m^{(t)}(\eta)=\langle\eta,\grad g(x^{(t)})\rangle_{x^{(t)}}+\frac12\langle\eta,\Hess\!g(x^{(t)})[\eta]\rangle_{x^{(t)}},\ \ \subjectto\ \|\eta\|_{x^{(t)}}\leq\Delta^{(t)}.
        \end{equation} 
        \STATE Update the parameters $({\tensG}^+,{\matu}_1^+,\dots,{\matu}_d^+)$ of $x^+= \retr_{x^{(t)}}(s^{(t)} \eta^{(t)})$.
        \STATE Compute $\rho^{(t)}=(f(x^{(t)})-f(x^+))/(m^{(t)}(0)-m^{(t)}(\eta^{(t)}))$.
        \STATE Set $x^{(t+1)}=x^+$ if $\rho\geq\rho^\prime$; otherwise $x^{(t+1)}=x^{(t)}$.
        \STATE Adjust radius $\Delta^{(t+1)}=\begin{cases}
                \Delta^{(t)}/4,\ \text{if}\ \rho^{(t)}<1/4,\\
                \min\{2\Delta^{(t)},\bar{\Delta}\},\ \text{if}\ \rho^{(t)}>3/4\ \text{and}\ \|\eta\|_{x^{(t)}}=\Delta^{(t)},\\
                \Delta^{(t)},\ \text{otherwise};\\
            \end{cases}$ \ \ $t=t+1$.
        \ENDWHILE
        \ENSURE The parameters $({\tensG}^{(t)},{\matu}_1^{(t)},\dots,{\matu}_d^{(t)})$ of $x^{(t)}$.
    \end{algorithmic}
\end{algorithm}

In practice, Algorithms~\ref{alg: RGD-desing}--\ref{alg: RTR-desing} operate directly on the parameters of $x^{(t)}$ and $\eta^{(t)}$, allowing for more efficient computations and avoiding the need to handle large-scale matrices and tensors explicitly. Recall that $x^{(t)}=(\tensX^{(t)},\matP_1^{(t)},\dots,\matP_d^{(t)})\in\tensM$ and $\eta^{(t)}=(\dot{\tensX}^{(t)},\dot{\matP}_1^{(t)},\dots,\dot{\matP}_d^{(t)})$ can be represented by $(\tensG^{(t)},\matu_1^{(t)},\dots,\matu_d^{(t)})$ and $(\dot{\tensG}^{(t)},\dot{\matu}_1^{(t)},\dots,\dot{\matu}_d^{(t)})$ respectively. We set the vector transport $\tensT_{t\gets t-1}$ in Algorithm~\ref{alg: RCG-desing} as the projection operator $\proj_{\tangent_{x^{(t)}}\!\tensM}$. Efficient computations of Riemannian gradient and projections immediately follow from~\eqref{eq: projection onto tangent} and~\eqref{eq: Riemannian gradient desing}. For the computation of retraction $\retr_{x^{(t)}}(s^{(t)}\eta^{(t)})$ in~\eqref{eq: desingularization retraction}, we first compute $\matu_k^{(t+1)}$ by QR decomposition of $\matu_k^{(t)}+s^{(t)}\dot{\matu}_k^{(t)}$. Then, we observe that
\begin{equation*}
    \begin{aligned}
        \tensX^{(t)}+s^{(t)}\dot{\tensX}^{(t)}
        &=(\tensG^{(t)} + s^{(t)}\dot{\tensG}^{(t)})\times_{k=1}^d\matu_k^{(t)}+s^{(t)}\sum_{k=1}^d\tensG\times_k\dot{\matu}_k^{(t)}\times_{j\neq k}\matu_j^{(t)}\\
        &=\tilde{\tensG}\times_{k=1}^d[\matu_k^{(t)}\ s^{(t)}\dot{\matu}_k^{(t)}]
    \end{aligned}
\end{equation*}
with some $\tilde{\tensG}\in\mathbb{R}^{(2r_1)\times(2r_2)\times\cdots\times(2r_d)}$. Subsequently, we have
\begin{equation*}
    \begin{aligned}
        (\tensX^{(t)}+s^{(t)}\dot{\tensX}^{(t)})\times_{k=1}^d\proj_{\matu_k^{(t+1)}}
        &=\big(\tilde{\tensG}\times_{k=1}^d((\matu_k^{(t+1)})^\top[\matu_k^{(t)}\ s^{(t)}\dot{\matu}_k^{(t)}])\big)\times_{k=1}^d \matu_k^{(t+1)}\\
        &={\tensG}^{(t+1)}\times_{k=1}^d \matu_k^{(t+1)}
    \end{aligned}
\end{equation*}
and thus $x^{(t+1)}=\retr_{x^{(t)}}(s^{(t)}\eta^{(t)})$ is represented by $(\tensG^{(t+1)},{\matu}_1^{(t+1)},\dots,{\matu}_d^{(t+1)})$. 

For Riemannian trust-region method in Algorithm~\ref{alg: RTR-desing}, the subproblem~\eqref{eq: RTR subproblem} can be approximately solved by the truncated conjugate gradient method; see~\cite{absil2007trust} for details. The inner product in~\eqref{eq: RTR subproblem} can be efficiently computed by~\eqref{eq: inner product desing}.

\subsection{Optimality conditions}
We analyze the optimality conditions of~\eqref{eq: problem (Q-desing)}. A point $x\in\tensM$ is a stationary point of $g=f\circ\varphi$ if $\grad g(x)=0$. Furthermore, a stationary point $x\in\tensM$ of $g$ is called second-order stationary of $g$ if $\Hess\!g(x)$ is positive semi-definite.
\begin{proposition}[first-order optimality]\label{prop: 1st-order optimality}
    A point $x=(\tensX,\matP_1,\matP_2,\dots,\matP_d)\in\tensM$ represented by $(\tensG,\matu_1,\matu_2,\dots,\matu_d)$ is a stationary point of~$g$ if 
    \[\nabla f(\tensX)\times_{k=1}^d\proj_{\matu_k}=0\quad\text{and}\quad(\nabla f(\tensX))_{(k)}\matv_k\matG_{(k)}^\top=0\]
    for $k\in[d]$, where $\matv_k=(\matu_j)^{\otimes j\neq k}$.
\end{proposition}
\begin{proof}
    Recall the Riemannian gradient~\eqref{eq: Riemannian gradient desing}.
    Since $\tilde{\matP}_k=0$ for $k\in[d]$, we have $\tilde{\matu}_k=-\matu_k^{}\tilde{\matu}_k^\top\matu_k^{}=0$. Therefore, it holds that 
    \[\nabla f(\tensX)\times_{k=1}^d\proj_{\matu_k}=0\quad\text{and}\quad\matP_k(\nabla f(\tensX))_{(k)}\matv_k\matG_{(k)}^\top=0.\]
    By using the {mode-$k$} unfolding $\matu_k^\top(\nabla f(\tensX))_{(k)}\matv_k=0$, we further simplify the equality to $(\nabla f(\tensX))_{(k)}\matv_k\matG_{(k)}^\top=(\matP_k+\matu_k^{}\matu_k^\top)(\nabla f(\tensX))_{(k)}\matv_k\matG_{(k)}^\top=0$. \hfill\squareforqed
\end{proof}

\begin{proposition}[second-order optimality]\label{prop: 2nd-order optimality}
    A first-order stationary point $x=(\tensX,\matP_1,\matP_2,\dots,\matP_d)\in\tensM$ represented by $(\tensG,\matu_1,\dots,\matu_d)$ is a second-order stationary point of $g$ if
    \[\langle\nabla^2f(\tensX)[\dot{\tensX}],\dot{\tensX}\rangle+\sum_{k=1}^{d}\langle\nabla f(\tensX),\sum_{j\neq k}\tensG\times_k\dot{\matu}_k\times_j\dot{\matu}_j\times_{\ell\notin\{j,k\}}\matu_\ell\rangle\geq 0\]
    holds for all $\dot{x}\in\tangent_x\!\tensM$ represented by $(\dot{\tensG},\dot{\matu}_1,\dots,\dot{\matu}_d)$, where $\dot{\tensX}=\varphi(\dot{x})$.
\end{proposition}
\begin{proof}
    If $x$ is second-order stationary of $g$, it holds that $\langle\Hess\!g(x)[\dot{x}],\dot{x}\rangle_x\geq 0$ for all $\dot{x}\in\tangent_x\!\tensM$, $\nabla f(\tensX)\times_{k=1}^d\proj_{\matu_k}=0$ and $(\nabla f(\tensX))_{(k)}\matv_k\matG_{(k)}^\top=0$ from Proposition~\ref{prop: 1st-order optimality}, where $\matv_k=(\matu_j)^{\otimes j\neq k}$. Recall the parameters $(\hat{\tensG},\hat{\matu}_1,\hat{\matu}_2,\dots,\hat{\matu}_d)$ and $(\breve{\tensG},\breve{\matu}_1,\breve{\matu}_2,\dots,\breve{\matu}_d)$ in~\eqref{eq: Riemannian Hessian}.
    It follows from the first-order optimality conditions that the parameters can be simplified to
    \begin{equation*}
        \begin{aligned}
            \hat{\tensG}&=-\sum_{k=1}^{d}\nabla f(\tensX)\times_k\dot{\matu}_k^\top\times_{j\neq k}\matu_j^\top,&
            \hat{\matu}_k&=\matP_k(\nabla f(\tensX))_{(k)}\!(\dot{\matv}_k\matG_{(k)}^\top+\matv_k\dot{\matG}_{(k)}^\top)\tilde{\matF}_k^{-1},\\
            \breve{\tensG}&=\nabla^2 f(\tensX)[\dot{\tensX}]\times_{k=1}^d\matu_k^\top,&
            \breve{\matu}_k&=\matP_k(\nabla^2 f(\tensX)[\dot{\tensX}])_{(k)}\!\matv_k\matG_{(k)}^\top\tilde{\matF}_k^{-1},
        \end{aligned}
    \end{equation*}
    where $\dot{\matv}_k\in\mathbb{R}^{n_{-k}\times r_{-k}}$ such that $\dot{\matx}_{(k)}=\dot{\matu}_k\matG_{(k)}\!\matv_k^\top+\matu_k\dot{\matG}_{(k)}\!\matv_k^\top+\matu_k\matG_{(k)}\!\dot{\matv}_k^\top$, and $\tilde{\matf}_k=2\matI_{r_k}+\matG_{(k)}^{}\matG_{(k)}^\top$. Therefore, we obtain from~\eqref{eq: inner product desing} that 
    \begin{equation*}
        \begin{aligned}
            &~~~~\langle\Hess\!g(x)[\dot{x}],\dot{x}\rangle_x\\
            &=\langle\hat{\tensG}+\breve{\tensG},\dot{\tensG}\rangle+\sum_{k=1}^{d}\langle\hat{\matu}_k+\breve{\matu}_k,\dot{\matu}_k(2\matI_{r_k}+\matG_{(k)}^{}\matG_{(k)}^\top)\rangle\\
            &=\langle\nabla^2 f(\tensX)[\dot{\tensX}]\times_{k=1}^d\matu_k^\top-\sum_{k=1}^{d}\nabla f(\tensX)\times_k\dot{\matu}_k^\top\times_{j\neq k}\matu_j^\top,\dot{\tensG}\rangle\\
            &~~~+\sum_{k=1}^{d}\Big\langle(\nabla f(\tensX))_{(k)}(\dot{\matv}_k\matG_{(k)}^\top+\matv_k\dot{\matG}_{(k)}^\top)+(\nabla^2 f(\tensX)[\dot{\tensX}])_{(k)}\matv_k\matG_{(k)}^\top,\matP_k\dot{\matu}_k\Big\rangle\\
            &=\big\langle\nabla^2 f(\tensX)[\dot{\tensX}],\dot{\tensG}\times_{k=1}^d\matu_k+\sum_{k=1}^d\tensG\times_k\matu_k\times_{j\neq k}\dot{\matu}_j\big\rangle\\
            &~~~+\sum_{k=1}^{d}\big\langle(\nabla f(\tensX))_{(k)},\dot{\matu}_k\matG_{(k)}\!\dot{\matv}_k^\top\big\rangle\\
            &=\langle\nabla^2f(\tensX)[\dot{\tensX}],\dot{\tensX}\rangle+\sum_{k=1}^{d}\langle\nabla f(\tensX),\sum_{j\neq k}\tensG\times_k\dot{\matu}_k\times_j\dot{\matu}_j\times_{\ell\notin\{j,k\}}\matu_\ell\rangle,
        \end{aligned}
    \end{equation*}
    where we use $\matP_k\dot{\matu}_k=\dot{\matu}_k$. \hfill\squareforqed
\end{proof}

\subsection{Difference between the desingularized and original problems}\label{subsec: counterexample}
The desingularization~\eqref{eq: desingularization of Tucker} provides a smooth parametrization of the non-smooth Tucker tensor variety $\mathbb{R}^{n_1\times n_2\times\cdots\times n_d}_{\leq\vecr}$, recasting the low-rank Tucker tensor optimization problem~\eqref{eq: problem (P)} to a Riemannian optimization problem~\eqref{eq: problem (Q-desing)}. However, the landscape of two problems is different. While the optimal values are equal, a stationary point $x\in\tensM$ of~\eqref{eq: problem (Q-desing)} does not necessarily imply that $\varphi(x)$ is a first-order stationary point of~\eqref{eq: problem (P)}. In the matrix case $d=2$, Levin et al.~\cite{levin2024effect} showed that a second-order stationary point $x$ of~\eqref{eq: problem (Q)} indicate that $\varphi(x)$ is a first-order stationary point of~\eqref{eq: problem (P)} for a class of parametrizations. However, this result does not extend to the tensor case.

Specifically, given $x=(\tensX,\matP_1,\matP_2,\dots,\matP_d)\in\tensM$ with $\underline{\vecr}:=\ranktc(\tensX)<\vecr$ represented by $(\tensG,\matu_1,\matu_2,\dots,\matu_d)$, we assume that there exists $k_0\in[d]$ such that $r_{k_0}<n_{k_0}$. We aim to show that there exists a smooth function $f$, such that $x$ is second-order stationary of~\eqref{eq: problem (Q-desing)} but $\varphi(x)$ is not stationary of~\eqref{eq: problem (P)}. 

To this end, we first construct a function $f$ and then validate the second-order optimality conditions for the given $x$. Consider the thin Tucker decomposition $\tensG=\underline{\tensG}\times_{k=1}^d\underline{\matu}_k$ of $\tensG$ with $\underline{\tensG}\in\mathbb{R}^{\underline{r}_1\times\underline{r}_2\times\cdots\times\underline{r}_d}$ and $\underline{\matu}_k\in\St(\underline{r}_k,r_k)$. If $r_k<n_k$, there exists $\vecv_k\in\mathbb{R}^{n_k}\setminus\{0\}$ such that $\matu_{k}^\top\vecv_k^{}=0$. Otherwise, since $\underline{r}_k<r_k=n_k$, there exists $\vecv_k\in\mathbb{R}^{n_k}\setminus\{0\}$ such that $(\matu_{k}\underline{\matu}_k)^\top\vecv_k^{}=0$. Let $f(\tensX) = \langle\tensX,\vecv_1\circ\vecv_2\circ\cdots\circ\vecv_d\rangle$. It follows from $\nabla f(\tensX)\times_{k_0}\matu_{k_0}^\top=0$ that $\nabla f(\tensX)\times_{k=1}^d\proj_{\matu_k}=0$ and
\[(\nabla f(\tensX))_{(k)}\matv_k\matG_{(k)}^\top=(\nabla f(\tensX)\times_{j\neq k}(\matu_j\underline{\matu}_j)^\top)_{(k)}\underline{\matG}_{(k)}^\top\underline{\matu}_k^\top=0,\]
where $\matv_k=(\matu_j)^{\otimes j\neq k}$. In view of Proposition~\ref{prop: 1st-order optimality}, $x$ is a first-order stationary point of~$g=f\circ\varphi$. For the second-order optimality, we observe that
\begin{equation*}
    \begin{aligned}
        &~~~\langle\nabla^2f(\tensX)[\dot{\tensX}],\dot{\tensX}\rangle+\sum_{k=1}^{d}\langle\nabla f(\tensX),\sum_{j\neq k}\tensG\times_k\dot{\matu}_k\times_j\dot{\matu}_j\times_{\ell\notin\{j,k\}}\matu_\ell\rangle\\
        &=\sum_{k=1}^{d}\sum_{j\neq k}\langle\vecv_1\circ\vecv_2\circ\cdots\circ\vecv_d,\tensG\times_k\dot{\matu}_k\times_j\dot{\matu}_j\times_{\ell\notin\{j,k\}}\matu_\ell\rangle\\
        &=\sum_{k=1}^{d}\sum_{j\neq k}\langle\vecv_1\circ\vecv_2\circ\cdots\circ\vecv_d,\underline{\tensG}\times_k(\dot{\matu}_k\underline{\matu}_k)\times_j(\dot{\matu}_j\underline{\matu}_j)\times_{\ell\notin\{j,k\}}(\matu_\ell\underline{\matu}_\ell)\rangle\\
        &=0
    \end{aligned}
\end{equation*}
holds for all $\dot{x}\in\tangent_x\!\tensM$ represented by $(\dot{\tensG},\dot{\matu}_1,\dots,\dot{\matu}_d)$, where $\tensX=\varphi(x)$ and $\dot{\tensX}=\varphi(\dot{x})$. Therefore, $x$ is a second-order stationary point of~\eqref{eq: problem (Q-desing)}. 

However, since $\ranktc(\tensX)<\vecr$ and $\nabla f(\tensX)\neq 0$, $\tensX$ is not a stationary point of~\eqref{eq: problem (P)} from~\cite[Proposition 3]{gao2023low}. In summary, we give a counterexample that a second-order stationary point of~\eqref{eq: problem (Q-desing)} does not necessarily imply a stationary point of~\eqref{eq: problem (P)}. Nevertheless, the desingularization preserves the intrinsic structure of the Tucker tensor varieties (see subsection~\ref{subsec: connection}) and provides a smooth parametrization. In addition, it provides a new perspective for low-rank tensor optimization.

\subsection{Parametrization via Tucker decomposition}\label{subsec: params via Tucker}
Alternatively, we consider the parametrization via Tucker decomposition with the search space $\tensM^\mathrm{Tucker}$ and mapping $\tilde{\varphi}$ in~\eqref{eq: Tucker parametrization}. Then, problem~\eqref{eq: problem (P)} can also be reformulated as
\begin{equation}
    \label{eq: problem (Q-Tucker)}
    \tag{Q-Tucker}
    \min_{\tilde{x}} \tilde{g}(\tilde{x})=f(\tilde{\varphi}(\tilde{x}))\quad\subjectto\ \tilde{x}\in\tensM^\mathrm{Tucker}.
\end{equation}
It is worth noting that the points in $\tensM$ can also be represented by $(\tensG,\matu_1,\dots,\matu_d)$. However, the manifold $\tensM$ is essentially different from $\tensM^\mathrm{Tucker}$ since they are embedded in different ambient spaces. Denote the set of first- and second-order stationary points of~\eqref{eq: problem (Q-Tucker)} (or~\eqref{eq: problem (Q-desing)}) by $\tilde{S}_1$ and $\tilde{S}_2$ (or $S_1$ and $S_2$). We illustrate the connection between $\tilde{S}_1$ ($\tilde{S}_2$) and $S_1$ ($S_2$). 

The next lemma validates that the group actions on $\tensM^\mathrm{Tucker}$ preserves stationarity.
\begin{lemma}\label{lem: stationary along group action}
    If $(\tilde{\tensG},\tilde{\matu}_1,\dots,\tilde{\matu}_d)\in\tensM^\mathrm{Tucker}$ is first-order (second-order) stationary of~$\tilde{g}$, then $(\tilde{\tensG}\times_{k=1}^d\matR_k^{-1},\tilde{\matu}_1\matR_1,\dots,\tilde{\matu}_d\matR_d)\in\tensM^\mathrm{Tucker}$ is also first-order (second-order) stationary of $\tilde{g}$ for all invertible matrices $\matR_k\in\GL(r_k)$, i.e., the group actions maintain stationarity.
\end{lemma}
\begin{proof}
    See Appendix~\ref{app: proof of lemma 4.2}.\hfill\squareforqed
\end{proof}

\begin{theorem}
    If $(\tilde{\tensG},\tilde{\matu}_1,\dots,\tilde{\matu}_d)\in\tensM^\mathrm{Tucker}$ is first-order stationary of $\tilde{g}$, then $x\in\tensM$ represented by $(\tilde{\tensG}\times_{k=1}^d\matR_k^{-1},\tilde{\matu}_1\matR_1,\dots,\tilde{\matu}_d\matR_d)$ is first-order stationary of $g$, where $\matR_k\in\GL(r_k)$ satisfies that $\tilde{\matu}_k\matR_k\in\St(r_k,n_k)$. Furthermore, it holds that $\varphi(S_1)=\tilde{\varphi}(\tilde{S}_1)$.
\end{theorem}
\begin{proof}
    Denote $\tilde{x}=(\tilde{\tensG},\tilde{\matu}_1,\dots,\tilde{\matu}_d)$ and $\tilde{x}^\matR=(\tilde{\tensG}\times_{k=1}^d\matR_k^{-1},\tilde{\matu}_1\matR_1,\dots,\tilde{\matu}_d\matR_d)$. Recall the partial derivatives~\eqref{eq: partial G} and~\eqref{eq: partial uk}, and equalities~\eqref{eq: partial G under group} and~\eqref{eq: partial uk under group}. It is straightforward to verify that Proposition~\ref{prop: 1st-order optimality} holds at $x\in\tensM$ represented by $(\tilde{\tensG}\times_{k=1}^d\matR_k^{-1},\tilde{\matu}_1\matR_1,\dots,\tilde{\matu}_d\matR_d)$ if and only if $\nabla \tilde{g}(\tilde{x}^\matR)=0$. Since $\varphi(x)=\tilde{\varphi}(\tilde{x}^\matR)=\tilde{\varphi}(\tilde{x})$, it holds that $\varphi(S_1)=\tilde{\varphi}(\tilde{S}_1)$.
\end{proof}

\begin{theorem}
    If $(\tilde{\tensG},\tilde{\matu}_1,\dots,\tilde{\matu}_d)\in\tensM^\mathrm{Tucker}$ is second-order stationary of $\tilde{g}$, then $x\in\tensM$ represented by $(\tilde{\tensG}\times_{k=1}^d\matR_k^{-1},\tilde{\matu}_1\matR_1,\dots,\tilde{\matu}_d\matR_d)$ is second-order stationary of $g$, where $\matR_k\in\GL(r_k)$ satisfies that $\tilde{\matu}_k\matR_k\in\St(r_k,n_k)$. Furthermore, it holds that $\varphi(S_2)=\tilde{\varphi}(\tilde{S}_2)$.
\end{theorem}
\begin{proof}
    On the one hand, for a second-order stationary point $\tilde{x}$ of $\tilde{g}$, there exists $\matR_k\in\GL(r_k)$, such that $\tilde{\matu}_k\matR_k\in\St(r_k,n_k)$ and $\tilde{x}^\matR=(\tilde{\tensG}\times_{k=1}^d\matR_k^{-1},\tilde{\matu}_1\matR_1,\dots,\tilde{\matu}_d\matR_d)$ is also a second-order stationary point of $\tilde{g}$ from Lemma~\ref{lem: stationary along group action}. We observe from~\eqref{eq: equal inner product} that 
    \begin{align}
        &~~~~\langle\dot{x}^\matR,\nabla^2 \tilde{g}(\tilde{x}^\matR)[\dot{x}^\matR]\rangle=\langle\dot{x},\nabla^2 \tilde{g}(\tilde{x})[\dot{x}]\rangle\nonumber\\
        &=\langle\nabla^2f(\tensX)[\dot{\tensX}],\dot{\tensX}\rangle+\sum_{k=1}^{d}\langle\nabla f(\tensX),\sum_{j\neq k}\tilde{\tensG}\times_k\dot{\matu}_k\times_j\dot{\matu}_j\times_{\ell\notin\{j,k\}}\tilde{\matu}_\ell\rangle\geq 0\label{eq: 2nd order stationary of Tucker}
    \end{align}
    holds for all $\dot{x}\in\tangent_{\tilde{x}}\!\tensM^\mathrm{Tucker}$, which is in the same form as Proposition~\ref{prop: 2nd-order optimality} without the orthogonality conditions $\dot{\matu}_k^\top\tilde{\matu}_k^{}=0$ for $k\in[d]$. Therefore, the point $x\in\tensM$ represented by $(\tilde{\tensG}\times_{k=1}^d\matR_k^{-1},\tilde{\matu}_1\matR_1,\dots,\tilde{\matu}_d\matR_d)$ is a second-order stationary point of~\eqref{eq: problem (Q-desing)}. Since $\tilde{\varphi}(\tilde{x})=\tilde{\varphi}(\tilde{x}^\matR)=\varphi(x)$, it holds that  $\tilde{\varphi}(\tilde{S}_2)\subseteq\varphi(S_2)$.

    On the other hand, let $x\in\tensM$ represented by $(\tensG,\matu_1,\dots,\matu_d)$ be a second-order stationary point of~\eqref{eq: problem (Q-desing)}. We aim to show that $\tilde{x}=(\tensG,\matu_1,\dots,\matu_d)$ is a second-order stationary point of~\eqref{eq: problem (Q-Tucker)}, i.e., \eqref{eq: 2nd order stationary of Tucker} holds for all $\dot{x}\in\tangent_{\tilde{x}}\!\tensM^\mathrm{Tucker}$. Since $\dot{x}\in\tensM^\mathrm{Tucker}$, there exists $\matR_{k,1}\in\mathbb{R}^{r_k\times r_k}$ and $\matR_{k,2}\in\mathbb{R}^{(n_k-r_k)\times r_k}$ such that $\dot{\matu}_k=\matu_k\matR_{k,1}+\matu_k^\perp\matR_{k,2}$. Taking the decomposition into~\eqref{eq: 2nd order stationary of Tucker}, we obtain that
    \begin{equation*}
        \begin{aligned}
            &~~~~\langle\nabla^2\tilde{g}(\tilde{x})[\dot{x}],\dot{x}\rangle\\
            &=\langle\nabla^2f(\tensX)[\dot{\tensX}],\dot{\tensX}\rangle+\sum_{k=1}^{d}\langle\nabla f(\tensX),\sum_{j\neq k}\tensG\times_k\dot{\matu}_k\times_j\dot{\matu}_j\times_{\ell\notin\{j,k\}}\matu_\ell\rangle\\
            &=\langle\nabla^2f(\tensX)[\dot{\tensX}],\dot{\tensX}\rangle+\sum_{k=1}^{d}\langle\nabla f(\tensX),\sum_{j\neq k}\tensG\times_k(\matu_k^\perp\matR_{k,2})\times_j(\matu_j^\perp\matR_{j,2})\times_{\ell\notin\{j,k\}}\matu_\ell\rangle,
        \end{aligned}
    \end{equation*}
    where $\dot{\tensX}=\mathrm{D}\tilde{\varphi}(\tilde{x})[\dot{x}]$ and the last equality follows from the first-order optimality conditions in Proposition~\ref{prop: 1st-order optimality}. Since $\dot{\tensX}$ can also be expressed by $\dot{\tensX}=(\dot{\tensG}+\sum_{k=1}^d\tensG\times_k\matR_{k,1})\times_{k=1}^d\matu_k+\sum_{k=1}^d\tensG\times_k(\matu_k^\perp\matR_{k,2})\times_{j\neq k}\matu_j$, we obtain that 
    \[\langle\nabla^2\tilde{g}(\tilde{x})[\dot{x}],\dot{x}\rangle=\langle\Hess\!g(x)[\dot{y}],\dot{y}\rangle_x\geq 0,\]
    where $\dot{y}\in\tangent_x\!\tensM$ is represented by $((\dot{\tensG}+\sum_{k=1}^d\tensG\times_k\matR_{k,1}),\matu_1^\perp\matR_{1,2},\dots,\matu_d^\perp\matR_{d,2})$. Therefore, $\tilde{x}=(\tensG,\matu_1,\dots,\matu_d)$ is a second-order stationary point of~\eqref{eq: problem (Q-Tucker)}. Since $\varphi(x)=\tilde{\varphi}(\tilde{x})$, it holds that $\varphi(S_2)=\tilde{\varphi}(\tilde{S}_2)$. \hfill\squareforqed
\end{proof}

While different parametrizations are able to share the same stationary points, the numerical performance can be disparate in certain scenarios; see Fig~\ref{fig: 0}. Moreover, since $\tilde{\varphi}(\tilde{S}_2)=\varphi(S_2)$, the counterexample in subsection~\ref{subsec: counterexample} demonstrates that $\tilde{\varphi}(\tilde{x})$ may not be stationary for~\eqref{eq: problem (P)} even if $\tilde{x} \in \tensM^\mathrm{Tucker}$ is a second-order stationary point of~\eqref{eq: problem (Q-Tucker)}. In other words, optimization via these parametrizations does not necessarily lead to a stationary point of~\eqref{eq: problem (P)}, which is essentially different from the matrix case; see Table~\ref{tab: existing works} for an overview. We leave the effort of finding a stationary point of~\eqref{eq: problem (P)} for future research.

\section{Desingularization of tensor train varieties}\label{sec: TT desing}
Recently, the tensor train (TT) decomposition~\cite{oseledets2011tensor}---also known as the matrix product states with open boundary condition~\cite{haegeman2014geometry}---appears to be prosperous~\cite{holtz2012manifolds,steinlechner2016riemannian,kutschan2018tangent,vermeylen2023rank}. However, TT varieties still suffer from the singularity. It is appealing to know how to desingularize the TT varieties. In this section, we introduce a desingularization of tensor train varieties.

\subsection{Tensor train decomposition}
Given $\tensX\in\mathbb{R}^{n_1\times n_2\times\cdots\times n_d}$, the TT decomposition of $\tensX$ is defined by 
\[\tensX=\llbracket\tensU_1,\tensU_2,\dots,\tensU_d\rrbracket\quad\text{or}\quad\tensX(i_1,i_2,\dots,i_d)=\matu_1(i_1)\matu_2(i_2)\cdots\matu_d(i_d)\]
for $i_k\in[n_k]$ and $k\in[d]$, where $\tensU_k\in\mathbb{R}^{r_{k-1}\times n_k\times r_k}$ is a core tensor, $\matu_k(i_k)={\tensU_k(:,i_k,:)}$ and $r_0=r_d=1$. Figure~\ref{fig: TT} depicts the TT decomposition of a third-order tensor.
\begin{figure}[htbp]
	\centering\scriptsize
	\begin{tikzpicture}[scale = 0.9]
		\draw[-] (-1.4,1.6) -- (-0.6,2.2);
		\draw[-] (-0.6,2.2) -- (-2.1,2.2);
		\draw[-] (-2.1,2.2) -- (-2.9,1.6);
		\draw[-] (-2.9,1.6) rectangle (-1.4,0.1);
		\draw[-] (-1.4,0.1) -- (-0.6,0.7);
		\draw[-] (-0.6,0.7) -- (-0.6,2.2);
		\node (v) at (-2.15,0.85) {$\tensX$};

            \fill (-2.52,0.36+1) circle (1pt);
		\node (x123) at (-2.52,-0.32) {$\tensX(i_1,\cdots,i_d)$};
		\draw[->] (-2.52,0.36+1) -- (x123);
	
		\node (u) at (-0.3,1.3) {$=$};
		\node (u) at (-0.3,-0.3) {$=$};
	
		\draw[-] (1.5-0.9,1.5-0.3) -- (1.98-0.9,1.86-0.3);
		\draw[-] (1.98-0.9,1.86-0.3) -- (1.38-0.9,1.86-0.3);
		\draw[-] (0.9-0.9,1.5-0.3) -- (1.38-0.9,1.86-0.3);
		\draw[-] (0.9-0.9,1.5-0.3) -- (1.5-0.9,1.5-0.3);
		\node (G) at (1.2-0.9,2) {$\tensU_1$};
		\path[draw, pattern=dots] (0.2,1.2) -- (0.25,1.2) -- (0.73,1.56) -- (0.68,1.56) -- cycle;
		\node (G1) at (0.465,-0.3) {$\matU_1(i_1)$};
		\draw[->] (0.465,1.38) -- (G1);
		
		\draw[-] (1.5+0.6,1.5) -- (1.98+0.6,1.86);
		\draw[-] (1.98+0.6,1.86) -- (1.38+0.6,1.86);
		\draw[-] (0.9+0.6,1.5) -- (1.38+0.6,1.86);
		\draw[-] (1.98+0.6,1.86) -- (1.98+0.6,1.26);
		\draw[-] (1.5+0.6,0.9) -- (1.98+0.6,1.26);
		\draw[-] (0.9+0.6,1.5) rectangle (1.5+0.6,0.9);
		\node (G) at (1.2+0.6,2) {$\tensU_{2}$};
		\path[draw, pattern=dots] (0.2+1.5,0.9) -- (0.68+1.5,1.26) -- (0.68+1.5,1.86) -- (0.2+1.5,1.5) -- cycle;
		\node (G2) at (0.44+1.5,-0.3) {$\matU_{2}(i_2)$};
		\draw[->] (0.44+1.5,1.38) -- (G2);
		
		\node(cdots) at (3,1.38) {$\cdots$};

		\draw[-] (1.5+0.6+1.92,1.5) -- (1.98+0.6+1.92,1.86);
		\draw[-] (1.98+0.6+1.92,1.86) -- (1.38+0.6+1.92,1.86);
		\draw[-] (0.9+0.6+1.92,1.5) -- (1.38+0.6+1.92,1.86);
		\draw[-] (1.98+0.6+1.92,1.86) -- (1.98+0.6+1.92,1.26);
		\draw[-] (1.5+0.6+1.92,0.9) -- (1.98+0.6+1.92,1.26);
		\draw[-] (0.9+0.6+1.92,1.5) rectangle (1.5+0.6+1.92,0.9);
		\node (G) at (1.2+0.6+1.92,2) {$\tensU_{d-1}$};
		\path[draw, pattern=dots] (0.2+1.5+1.92,0.9) -- (0.68+1.5+1.92,1.26) -- (0.68+1.5+1.92,1.86) -- (0.2+1.5+1.92,1.5) -- cycle;
		\node (Gd1) at (0.44+1.5+1.92,-0.3) {$\matU_{d-1}(i_{d-1})$};
		\draw[->] (0.44+1.5+1.92,1.38) -- (Gd1);

		\draw[-] (0.9+0.6+1.92+1.5,1.5+0.18) rectangle (1.5+0.6+1.92+1.5,0.9+0.18);
		\node (G) at (1.2+0.6+1.92+1.5,2) {$\tensU_{d}$};
		\path[draw, pattern=dots] (0.2+1.5+1.92+1.5,0.9+0.18) -- (0.25+1.5+1.92+1.5,0.9+0.18) -- (0.25+1.5+1.92+1.5,1.5+0.18) -- (0.2+1.5+1.92+1.5,1.5+0.18) -- cycle;
		\node (Gd) at (0.225+1.5+1.92+1.5,-0.3) {$\matU_{d}(i_{d})$};
		\draw[->] (0.225+1.5+1.92+1.5,1.38) -- (Gd);
	\end{tikzpicture}
        \caption{Tensor train decomposition of a third-order tensor}
        \label{fig: TT}
    \end{figure}
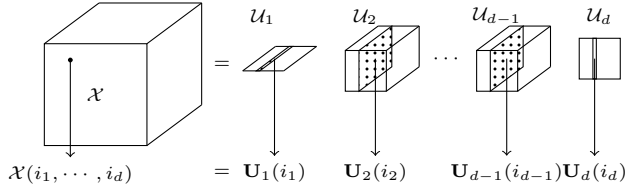

We introduce the notation for tensor train tensors as follows (see, e.g.,~\cite{oseledets2011tensor}). The \emph{$k$-th unfolding matrix} of $\tensX$ is defined by $\matx_{\langle k\rangle}\in\mathbb{R}^{(n_1n_2\cdots n_k)\times(n_{k+1}n_{k+2}\cdots n_d)}$ for $k\in[d-1]$ with 
\[\matx_{\langle k\rangle}(i_1+\sum_{j=2}^{k}(i_j-1)\prod_{\ell=1}^{j-1}n_\ell,\ i_{k+1}+\sum_{j=k+2}^{d}(i_j-1)\prod_{\ell=k+1}^{j-1}n_\ell)=\tensX(i_1,i_2,\dots,i_d).\]
The TT rank of $\tensX$ is $\ranktt(\tensX)=(\rank(\matx_{\langle 1\rangle}),\rank(\matx_{\langle 2\rangle}),\dots,\rank(\matx_{\langle d-1\rangle}))$. A core tensor $\tensU_k$ can be reshaped to the \emph{left} and \emph{right unfoldings} defined by $\leftunfolding(\tensU_k)\in\mathbb{R}^{(r_{k-1}n_k)\times r_k}$ and $\rightunfolding(\tensU_k)\in\mathbb{R}^{r_{k-1}\times (n_kr_k)}$. The \emph{interface} matrices $\matx_{\leq k}$ and $\matx_{\geq k+1}$ of $\tensX$ are defined by $\matx_{\leq k}(i_1+\sum_{j=2}^{k}(i_j-1)\prod_{\ell=1}^{j-1}n_\ell, :)=\matu_1(i_1)\cdots\matu_k(i_k)$ and $\matx_{\geq k+1}(i_{k+1}+\sum_{j=k+2}^{d}(i_j-1)\prod_{\ell=k+1}^{j-1}n_\ell,:)=(\matu_{k+1}(i_{k+1})\cdots\matu_d(i_d))^\top$ respectively. It holds that $\matx_{\langle k\rangle}=\matx_{\leq k}^{}\matx_{\geq k+1}^\top$ and the interface matrices can be constructed recursively by
\begin{equation}\label{eq: recursive interface}
    \matx_{\leq k}=(\matI_{n_k}\otimes\matx_{\leq k-1})\leftunfolding(\tensU_k)\quad\text{and}\quad\matx_{\geq k+1}=(\matx_{\geq k+2}\otimes\matI_{n_{k+1}})\rightunfolding(\tensU_{k+1})^\top.
\end{equation}
A tensor $\tensX=\llbracket\tensU_1,\tensU_2,\dots,\tensU_d\rrbracket$ is called $k$-orthogonal if $\leftunfolding(\tensU_j)\in\St(r_j,r_{j-1}n_j)$ for $j\in[k-1]$ and $\rightunfolding(\tensU_j)^\top\in\St(r_{j-1},n_jr_j)$ for $j=k+1,k+2,\dots,d$. The tensor is called left- or right-orthogonal if $k=d$ or $k=1$, respectively. It follows from~\cite[\S 3.1]{steinlechner2016riemannian} that any tensor $\tensX$ can be left- or right-orthogonalized via QR decomposition.

\subsection{Desingularization of tensor train varieties}
Given $\vecr^\mathrm{TT}=(r_1,r_2,\dots,r_{d-1})$, the set of bounded TT-rank tensors $\mathbb{R}^{n_1\times n_2\times\cdots\times n_d}_{\leq\vecr^\mathrm{TT}}$ is an algebraic variety~\cite{kutschan2018tangent}. We aim to desingularize $\mathbb{R}^{n_1\times n_2\times\cdots\times n_d}_{\leq\vecr^\mathrm{TT}}$ in a similar fashion as Tucker tensor varieties. However, the generalization is not straightforward due to the different formulations of varieties.

Since core tensors in the TT decomposition are interconnected through interface matrices, we introduce a modified set of slack variables that maintain the right-orthogonal structure while providing a smooth manifold representation:
\begin{equation}\label{eq: desingularization of TT}
    \tensM^\mathrm{TT}=\left\{\begin{array}{rl}
        (\tensX,\matP_1,\matP_2,\dots,\matP_{d-1}):&\tensX=\llbracket\tensU_1,\tensU_2,\dots,\tensU_d\rrbracket,\\
        &\matP_k=\matI_{n_{k+1}\dots n_d}-\matP_{\geq k+1},\ k\in[d-1],\vspace{1mm}\\
        &(\tensU_1,\tensU_2,\dots,\tensU_d)\in \tensS^\mathrm{TT}
    \end{array}\right\},
\end{equation}
where $\matP_{\geq k+1}=\matx_{\geq k+1}^{}(\matx_{\geq k+1}^\top\matx_{\geq k+1}^{})^{-1}\matx_{\geq k+1}^\top$ and $\tensS^\mathrm{TT}=\{(\tensU_1,\tensU_2,\dots,\tensU_d):\tensU_k\in\mathbb{R}^{r_{k-1}\times n_k\times r_{k}},\ k\in[d],\ \rank(\rightunfolding(\tensU_k))=r_{k-1},\ k=2,3,\dots,d\}$ is the parameter space of TT tensors with rank constraints. The TT variety $\mathbb{R}^{n_1\times n_2\times\cdots\times n_d}_{\leq\vecr^\mathrm{TT}}$ can be parametrized by the smooth mapping 
\[\varphi^\mathrm{TT}:\tensM^\mathrm{TT}\to\mathbb{R}^{n_1\times n_2\times\cdots\times n_d}_{\leq\vecr^\mathrm{TT}},\quad \varphi^\mathrm{TT}(\tensX,\matP_1,\matP_2,\dots,\matP_{d-1})=\tensX.\] 
Note that there is no rank constraint on $\tensU_1$ and $\varphi^\mathrm{TT}(\tensM^\mathrm{TT})=\mathbb{R}^{n_1\times n_2\times\cdots\times n_d}_{\leq\vecr^\mathrm{TT}}$. By using~\eqref{eq: recursive interface} and the rank constraints in $\tensS^\mathrm{TT}$, we observe that $\rank(\matx_{\geq k+1})=r_k$ for $k\in[d-1]$. Consequently, the parametrization~\eqref{eq: desingularization of TT} is well-defined.

\subsection{Manifold structure}
We consider the group action
\[\theta^\mathrm{TT}((\tensU_1,\dots,\tensU_d),\mata_1,\dots,\mata_{d-1})=(\tensU_1\times_3\mata_1^\top,\tensU_2\times_1\mata_1^{-1}\times_3\mata_2^\top,\dots,\tensU_d\times_1\mata_{d-1}^{-1})\]
for matrices $\mata_k\in\GL(r_k)$, which introduces an equivalent class $\sim$. We show that $\theta^\mathrm{TT}$ is free and proper. Given fixed $(\tensU_1,\dots,\tensU_d)\in\tensS^\mathrm{TT}$ and $(\tensV_1,\dots,\tensV_d)\in\tensS^\mathrm{TT}$ satisfying $(\tensU_1,\dots,\tensU_d)=\theta^\mathrm{TT}((\tensV_1,\dots,\tensV_d),\mata_1,\dots,\mata_{d-1})$ for some $\mata_k\in\GL(r_k)$, the matrices $\mata_k$ can be constructed recursively by 
\[\mata_{d-1}=\rightunfolding(\matv_d)\rightunfolding(\matu_d)^\dagger \quad\text{and}\quad \mata_{k-1}=\rightunfolding(\matv_k)(\mata_k\otimes\matI_{n_k})\rightunfolding(\matu_k)^\dagger\] 
for $k=2,3,\dots,d-1$, where $\rightunfolding(\matu_k)^\dagger=\rightunfolding(\matu_k)^\top(\rightunfolding(\matu_k)\rightunfolding(\matu_k)^\top)^{-1}$ is the Moore--Penrose pseudoinverse of $\rightunfolding(\matu_k)$. Therefore, the mapping $\theta^\mathrm{TT}$ is invertible and $(\tensU_1,\dots,\tensU_d)=(\tensV_1,\dots,\tensV_d)$ indicates $\mata_k=\matI_{r_k}$ for $k\in[d-1]$. Consequently, $\theta^\mathrm{TT}$ is free and proper, and the set $\tensS^\mathrm{TT}/\!\sim$ is a quotient manifold of~$\tensS^\mathrm{TT}$. 

In order to show that the desingularization $\tensM^\mathrm{TT}$ is an embedded submanifold of $\tensE^\mathrm{TT}=\mathbb{R}^{n_1\times n_2\times\cdots\times n_d}\times\Sym(n_2n_3\cdots n_d)\times\cdots\times\Sym(n_d)$, we first figure out the \emph{vertical} and \emph{horizontal spaces} of $\tensS^\mathrm{TT}$. Subsequently, we construct the mapping 
\begin{equation}
    \label{eq: mapping psi}
    \psi:\tensS^\mathrm{TT}/\!\sim\to\tensM^\mathrm{TT},\quad
    \psi([\tensU_1,\tensU_2,\dots,\tensU_d])=(\llbracket\tensU_1,\tensU_2,\dots,\tensU_d\rrbracket,\matP_1,\matP_2,\dots,\matP_{d-1})
\end{equation}
with $\matP_k=\matI_{n_{k+1}\dots n_d}-\matx_{\geq k+1}^{}(\matx_{\geq k+1}^\top\matx_{\geq k+1}^{})^{-1}\matx_{\geq k+1}^\top$ for $k\in[d-1]$. Note that $\psi$ is well-defined on the quotient manifold $\tensS^\mathrm{TT}/\!\sim$; see Proposition~\ref{prop: well defined psi}. Furthermore, we prove that $\psi$ is a homeomorphism and a smooth immersion. Finally, we obtain that $\tensM^\mathrm{TT}$ is a submanifold of $\tensE^\mathrm{TT}$; see Fig.~\ref{fig: communicative TT} for the diagram of desingularization. 

\begin{figure}[htbp]
    \centering
    \begin{tikzpicture}
        \node (M) at (0,-1.6) {$\tensM^\mathrm{TT}$};
        \node (Stt) at (-2.2,0) {$\tensS^\mathrm{TT}$};
        \node (Sttq) at (0,0) {$\tensS^\mathrm{TT}/\!\sim$};
        \node (E) at ($(-5,-0.025)+(M)$) {$\mathbb{R}^{n_1\times n_2\times\cdots\times n_d}\times\Sym(n_2n_3\cdots n_d)\times\cdots\times\Sym(n_d)$};
        \node (Ett) at ($(-7.95,0.35)+(M)$) {$\tensE^\mathrm{TT}=$};
        \node (R) at ($(M)+(2.5,-0.025)$) {$\mathbb{R}^{n_1\times n_2\times\cdots\times n_d}_{\leq\vecr^{\mathrm{TT}}}$};

        \draw[->,thick] (Stt) -- (Sttq);
        \draw[->,thick] (Sttq) -- (M);
        \draw[->,thick] (Stt) -- (M);
        \draw[left hook->,thick] ($(M)+(-0.5,-0.05)$) -- ($(E)+(3.5,-0.025)$);    
        \draw[->,thick] ($(M)+(0.5,-0.05)$) -- ($(R)+(-1.1,-0.025)$);    

        \node[above] at ($0.5*(Stt)+0.5*(Sttq)-(0.1,0)$) {$\pi$};
        \node[right] at ($0.5*(Sttq)+0.5*(M)$) {$\psi$};
        \node at ($0.5*(Stt)+0.5*(M)-(0.5,0)$) {$\psi\circ\pi$};
        \node[above] at ($0.5*(R)+0.5*(M)+(-0.3,-0.05)$) {$\varphi^\mathrm{TT}$};
    \end{tikzpicture}
    \caption{Diagram of desingularization of tensor train varieties. $\pi$ represents the natural projection}
    \label{fig: communicative TT}
\end{figure}

Specifically, the vertical spaces of $\tensS^\mathrm{TT}$ can be characterized as follows. 
\begin{lemma}
    The vertical space of $\tensS^\mathrm{TT}$ at $x=(\tensU_1,\tensU_2,\dots,\tensU_d)$ can be characterized by
        \begin{equation}
            \vertical_{x}\!\tensS^\mathrm{TT}=\left\{\begin{array}{r}
                (\tensU_1\times_3\matd_1^\top,-\tensU_2\times_1\matd_1+\tensU_2\times_3\matd_2^\top,\dots,-\tensU_d\times_1\matd_{d-1}^\top):~~~~~~\\
                \matd_k\in\mathbb{R}^{r_k\times r_k},\ k\in[d-1]
            \end{array}\right\}.\label{eq: TT vertical}
        \end{equation}
\end{lemma}
\begin{proof}
    Denote the right hand side of~\eqref{eq: TT vertical} by $V$. We aim to show that: 1) for all $\eta\in V$, it holds that $\eta\in\vertical_{x}\!\tensS^\mathrm{TT}$; 2) $\dim(V)=\sum_{k=1}^{d-1}r_k^2$. 

    For the first claim, we consider the smooth curve
    \[\gamma(t)=\theta^\mathrm{TT}(x,\gamma_1(t),\gamma_2(t),\dots,\gamma_{d-1}(t)),\]
    where $\gamma_k(t)=\matI_{r_k}+t\matd_k$ is a smooth curve on $\GL(r_k)$. Since $\gamma(0)=x$ and $\gamma^\prime(0)=\eta$, we have $\eta\in\vertical_{x}\!\tensS^\mathrm{TT}$ from~\cite[Definition 9.24]{boumal2023intromanifolds}. The second claim is straightforward by computing the degree of freedom. \hfill\squareforqed
\end{proof}

For the horizontal space of $\tensS^\mathrm{TT}$, we consider the set 
\begin{equation}
    \label{eq: TT horizontal}
    \horizontal_{x}\!\tensS^\mathrm{TT}=\left\{
    \begin{array}{rl}
        (\dot{\tensU}_1,\dots,\dot{\tensU}_d):&\dot{\tensU}_k\in\mathbb{R}^{r_{k-1}\times n_k\times r_k},\ k\in[d],\\ &\rightunfolding(\dot{\tensU}_k)(\matx_{\geq k+1}^\top\matx_{\geq k+1}^{}\otimes\matI_{n_k})\rightunfolding(\tensU_k)^\top=0,\ k=2,\dots,d
    \end{array}\right\}
\end{equation}
in the light of~\cite[(26)]{uschmajew2013geometry} and the forthcoming analysis, where we set $\matx_{\geq d+1}=1$. Since $\horizontal_{x}\!\tensS^\mathrm{TT}$ is a linear space, and $\dim(\horizontal_{x}\!\tensS^\mathrm{TT})=\sum_{k=1}^d r_{k-1}n_kr_k-\sum_{k=1}^{d-1}r_k^2=\dim(\tensS^\mathrm{TT})-\dim(\vertical_{x}\!\tensS^\mathrm{TT})$, $\horizontal_{x}\!\tensS^\mathrm{TT}$ is able to serve as a horizontal space.

Subsequently, we study the properties of $\psi$ in~\eqref{eq: mapping psi}. The following proposition illustrates that $\psi$ is well-defined on the quotient manifold $\tensS^\mathrm{TT}/\!\sim$.
\begin{proposition}\label{prop: well defined psi}
    The mapping $\psi$ is well-defined on the quotient manifold $\tensS^\mathrm{TT}/\!\sim$.
\end{proposition}
\begin{proof}
    Since $\psi$ is well-defined on the parameter space $\tensS^\mathrm{TT}$, it suffices to prove that $\psi(x)=\psi(\tilde{x})=(\tensX,\matP_1,\dots,\matP_{d-1})$ implies $x\sim\tilde{x}$ for $x=(\tensU_1,\dots,\tensU_d),\tilde{x}=(\tilde{\tensU}_1,\dots,\tilde{\tensU}_d)\in \tensS^\mathrm{TT}$. 
    
    Let $\matx_{\geq k+1}$ and $\tilde{\matx}_{\geq k+1}$ be the interface matrices constructed by parameters $\tensU_{k+1},\dots,\tensU_d$ and $\tilde{\tensU}_{k+1},\dots,\tilde{\tensU}_d$ via~\eqref{eq: recursive interface} respectively for $k\in[d-1]$. Since $\psi(x)=\psi(\tilde{x})$ and $\matx_{\geq k+1}$ and $\tilde{\matx}_{\geq k+1}$ are of full rank, there exist $\mata_k\in\GL(r_k)$, such that $\matx_{\geq k+1}=\tilde{\matx}_{\geq k+1}\mata_k$ for $k\in[d-1]$. We obtain recursively from~\eqref{eq: recursive interface} that $\tensU_d=\tilde{\tensU}_d\times_1\mata_{d-1}^{-1}$, $\tensU_k=\tilde{\tensU}_k\times_1\mata_{k-1}^{-1}\times_3\mata_k^\top$ for $k=2,3,\dots,d-1$. For $\tensU_1$ and $\tilde{\tensU}_1$, it holds that $\leftunfolding(\tensU_1)=\matx_{\langle 1\rangle}\matx_{\geq 2}=\matx_{\langle 1\rangle}\tilde{\matx}_{\geq 2}\mata_1=\leftunfolding(\tilde{\tensU}_1)\mata_1$. Consequently, we obtain that $x=\theta^\mathrm{TT}(\tilde{x},\mata_1,\dots,\mata_{d-1})$. \hfill\squareforqed
\end{proof}

\begin{lemma}
    The mapping $\psi:\tensS^\mathrm{TT}/\!\sim\to\tensM^\mathrm{TT}$ is a homeomorphism.
\end{lemma}
\begin{proof}
    Since $\psi(\tensS^\mathrm{TT}/\!\sim)=\psi(\tensS^\mathrm{TT})=\tensM^\mathrm{TT}$, $\psi$ is surjective and thus bijective from Proposition~\ref{prop: well defined psi}, i.e., $\psi^{-1}$ exists. 
    
    Next, we aim to show that $\psi^{-1}$ is continuous. Given $(\tensX,\matP_1,\matP_2,\dots,\matP_{d-1})\in\tensM^\mathrm{TT}$, since $\psi$ is bijective, we construct a representative $(\tensU_1,\tensU_2,\dots,\tensU_d)$ of the preimage $\psi^{-1}(\tensX,\matP_1,\matP_2,\dots,\matP_{d-1})$, which satisfies $\rightunfolding(\tensU_k)\rightunfolding(\tensU_k)^\top=\matI_{r_{k-1}}$ for $k=2,3,\dots,d$, i.e., the right orthogonality.

    We observe that there exist $\matv_k\in\St(r_k,n_{k+1}\dots n_d)$, such that $\matP_k=\matI_{n_{k+1}\dots n_d}-\matv_k^{}\matv_k^\top$ for $k\in[d-1]$. Subsequently, the representative $(\tensU_1,\tensU_2,\dots,\tensU_d)$ can be constructed recursively from~\eqref{eq: recursive interface} by
    \begin{equation}
        \label{eq: TT params}
        \begin{aligned}
            \rightunfolding(\tensU_d)&=\matv_{d-1}^\top,\
            \rightunfolding(\tensU_{k})=\matv_{k-1}^\top(\matx_{\geq k+1}\otimes\matI_{n_k}),\quad\text{and}\quad
            \leftunfolding(\tensU_1)=\matx_{\langle 1\rangle}\matx_{\geq 2}
        \end{aligned}
    \end{equation}
    for $k=2,3,\dots,d$. 
    Note that $\tensX$ is only applied to $\leftunfolding(\tensU_1)=\matx_{\langle 1\rangle}\matx_{\geq 2}$ in~\eqref{eq: TT params}. We construct the tensor $\tensU_{k}$ by using $\matv_{k-1}=(\matx_{\geq k+1}\otimes\matI_{n_k})\rightunfolding(\tensU_{k})^\top$ and the interface matrix $\matx_{\geq k+1}$ which has been constructed by~\eqref{eq: recursive interface} via $\tensU_{k+1},\tensU_{k+2},\dots,\tensU_d$. We observe from~\eqref{eq: TT params} that $(\tensU_1,\dots,\tensU_d)$ is continuous with respect to $(\tensX,\matv_1,\dots,\matv_{k-1})$. 

    Consider a sequence $\{(\tensX^{(t)},\matP_1^{(t)},\matP_2^{(t)},\dots,\matP_{d-1}^{(t)})\}_{t=0}^\infty\subseteq\tensM^{\mathrm{TT}}$ which converges to $(\tensX^*,\matP_1^*,\matP_2^*,\dots,\matP_{d-1}^*)$. Since $\matP_k^{(t)}$ converges to $\matP_k^*$, there exists $\matv_k^{(t)},\matv_k^*\in\St(r_k,n_{k+1}\dots n_d)$, such that $\matP_k^{(t)}=\matI_{n_{k+1}\dots n_d}-\matv_k^{(t)}(\matv_k^{(t)})^\top$, $\matP_k^*=\matI_{n_{k+1}\dots n_d}-\matv_k^*(\matv_k^*)^\top$ and $\matv_k^{(t)}$ converges to $\matv_k^*$ for $k\in[d-1]$. We consider the preimages $(\tensU_1^{(t)},\dots,\tensU_d^{(t)})$ and $(\tensU_1^*,\dots,\tensU_d^*)$ constructed by~\eqref{eq: TT params}. It follows from the continuity that $\tensU_k^{(t)}$ converges to $\tensU_k^*$. Consequently, $\psi$ is bijective with continuous inverse, i.e., homeomorphism. \hfill\squareforqed
\end{proof}

\begin{theorem}
    The parametrization $\tensM^\mathrm{TT}$ is an embedded submanifold of $\tensE^\mathrm{TT}$ with dimension $\dim(\tensM^\mathrm{TT})=\dim(\tensS^\mathrm{TT}/\!\sim)=\sum_{k=1}^{d} r_{k-1}n_kr_k-\sum_{k=1}^{d-1}r_k^2$.
\end{theorem}
\begin{proof}
    It suffices to prove that $\mathrm{D}\psi(\tensU_1,\tensU_2,\dots,\tensU_d)|_{\horizontal_{(\tensU_1,\tensU_2,\dots,\tensU_d)}\!S}$ is injective, i.e., $\mathrm{D}\psi(\tensU_1,\tensU_2,\dots,\tensU_d)[\dot{\tensU}_1,\dot{\tensU}_2,\dots,\dot{\tensU}_d]=0$ and $(\dot{\tensU}_1,\dot{\tensU}_2,\dots,\dot{\tensU}_d)\in\horizontal_{(\tensU_1,\tensU_2,\dots,\tensU_d)}\!\tensS^\mathrm{TT}$ imply that $(\dot{\tensU}_1,\dot{\tensU}_2,\dots,\dot{\tensU}_d)=0$.

    Denote $\mathrm{D}\psi(\tensU_1,\tensU_2,\dots,\tensU_d)[\dot{\tensU}_1,\dot{\tensU}_2,\dots,\dot{\tensU}_d]=(\tilde{\tensX},\tilde{\matP}_1,\tilde{\matP}_2,\dots,\tilde{\matP}_{d-1})$. We prove $\dot{\tensU}_k=0$ for $k=2,3,\dots,d$ recursively. Since $\tilde{\matP}_{d-1}=-2\sym(\rightunfolding(\dot{\tensU}_d)^\top\rightunfolding(\tensU_d))=0$ and $\rightunfolding(\dot{\tensU}_{d})\rightunfolding(\tensU_{d})^\top=0$ from~\eqref{eq: TT horizontal}, we have $\dot{\tensU}_{d}=0$. If $\dot{\tensU}_{k}=0$ holds for all $k=\ell+1,\dots,d$, we aim to show that $\dot{\tensU}_{k}=0$ holds for $k=\ell$ with $\ell\geq 2$. We observe that 
    \[\tilde{\matP}_{\ell-1}=-2\sym((\matx_{\geq \ell+1}\otimes\matI_{n_\ell})\rightunfolding(\dot{\tensU}_\ell)^\top\matx_{\geq \ell}^\top)=0,\]
    where we use the fact that $\dot{\tensU}_{k}=0$ for all $k=\ell+1,\dots,d$. Multiplying $\matx_{\geq \ell}^\top=\rightunfolding(\tensU_\ell)(\matx_{\geq \ell+1}\otimes\matI_{n_\ell})^\top$ on the left, we obtain $\matx_{\geq \ell}^\top\matx_{\geq \ell}^{}\rightunfolding(\dot{\tensU}_\ell)(\matx_{\geq \ell+1}\otimes\matI_{n_\ell})^\top=0$. Since $\matx_{\geq \ell}$ and $\matx_{\geq \ell+1}$ are of full rank, it holds that $\dot{\tensU}_\ell=0$.

    For $\dot{\tensU}_1$, we observe that $\tilde{\tensX}=\llbracket\dot{\tensU}_1,\tensU_2,\dots,\tensU_d\rrbracket=0$, or $\leftunfolding(\dot{\tensU}_1)\matx_{\geq 2}^\top=0$ equivalently. It follows from $\rank(\matx_{\geq 2})=r_1$ that $\dot{\tensU}_1=0$. Consequently, $\psi$ is an injective immersion and thus $\tensM^\mathrm{TT}$ is a Riemannian submanifold of $\tensE^\mathrm{TT}$ with $\dim(\tensM^\mathrm{TT})=\dim(\tensS^\mathrm{TT}/\!\sim)$. \hfill\squareforqed
\end{proof}

\begin{remark}
    The new parametrization~\eqref{eq: desingularization of TT} is constructed based on the rank constraints on the right unfolding matrices of core tensors. One may consider other types of rank constraints. Notably, the proposed parametrization~\eqref{eq: desingularization of TT} for TT varieties differs from the desingularization of matrix and Tucker tensor varieties. Furthermore, it preserves the structure of the TT format, avoiding the explicit formulation of large tensors. The Riemannian geometry of $\tensM^\mathrm{TT}$ can be developed in a similar fashion.
\end{remark}

\section{Numerical experiments}\label{sec: experiments}
In this section, we compare the proposed methods with other candidates on the tensor completion problem over the set of bounded-rank Tucker tensors. Specifically, given a partially observed tensor $\tensA\in\mathbb{R}^{n_1\times n_2\cdots\times n_d}$ on an index set $\Omega\subseteq[n_1]\times [n_2]\times\cdots\times[n_d]$. The goal of Tucker tensor completion is to recover the tensor $\tensA$ from its entries on $\Omega$ based on the low-rank Tucker decomposition. The optimization problem can be formulated on the Tucker tensor variety~$\mathbb{R}^{n_1\times\cdots\times n_d}_{\leq\vecr}$, i.e., 
\begin{equation}\label{eq: LRTC P}
    \begin{aligned}
        \min_\tensX\ \ & \frac12\|\proj_\Omega(\tensX)-\proj_\Omega(\tensA)\|_\frob^2\\
        \subjectto\ \ & \quad \quad \tensX\in\mathbb{R}^{n_1\times\cdots\times n_d}_{\leq\vecr},
    \end{aligned}
\end{equation}
where $\proj_\Omega$ is the projection operator onto $\Omega$, i.e, $\proj_\Omega(\tensX)(i_1,\dots,i_d)=\tensX(i_1,\dots,i_d)$ if~$(i_1,\dots,i_d)\in\Omega$, otherwise $\proj_\Omega(\tensX)(i_1,\dots,i_d)=0$ for $\tensX\in\mathbb{R}^{n_1\times n_2\times\cdots\times n_d}$. The \emph{sampling rate} is denoted by $p:=|\Omega|/(n_1n_2\cdots n_d)$. 

All experiments are performed on a workstation with two Intel(R) Xeon(R) Processors Gold 6330 (at 2.00GHz$\times$28, 42M Cache) and 512GB of RAM running Matlab R2019b under Ubuntu 22.04.3. The codes of the proposed methods are available at \href{https://github.com/JimmyPeng1998}{https://github.com/JimmyPeng1998.}

\subsection{Test on different parametrizations}\label{subsec: diff params}
We test the performance of Riemannian methods under different parametrizations on synthetic data. Specifically, by choosing a search space $\tensM$ and a smooth mapping $\varphi$ such that $\varphi(\tensM)=\mathbb{R}^{n_1\times n_2\times\cdots\times n_d}_{\leq\vecr}$, we reformulate the completion problem~\eqref{eq: LRTC P} to 
\begin{equation}
    \label{eq: LRTC Q}
    \min_{x\in \tensM} g(x)=\frac12\|\proj_\Omega(\varphi(x))-\proj_\Omega(\tensA)\|_\frob^2.
\end{equation}  
There are various choices of the search space $\tensM$ and mapping $\varphi$ in~\eqref{eq: LRTC Q}. For instance, a tensor $\tensX\in\mathbb{R}^{n_1\times\cdots\times n_d}_{\leq\vecr}$ can be represented by $x=(\tensG,\matU_1,\dots,\matu_d)\in\tensM$ via Tucker decomposition $\tensX=\varphi(x)=\tensG\times_{k=1}^d\matu_k$; see subsection~\ref{subsec: params via Tucker} for details. For the sake of numerical stability, we impose the orthogonality conditions and set the search space $\tensM$ by
\[\mathbb{R}^{r_1\times r_2\times\cdots\times r_d}\times\St(r_1,n_1)\times\St(r_2,n_2)\times\cdots\times\St(r_d,n_d).\]
We refer to this parametrization as ``Tucker''. We also consider the quotient manifold
\[\mathbb{R}^{r_1\times\cdots\times r_d}\times\St(r_1,n_1)\times\cdots\times\St(r_d,n_d)/(\mathcal{O}(r_1)\times\cdots\times\mathcal{O}(r_d))\]
in~\cite{kasai2016low}. The parametrization is denoted by ``quotient''. The proposed desingularization $\tensM(\mathbb{R}^{n_1\times\cdots\times n_d},\vecr)$ is denoted by ``desing''. To solve~\eqref{eq: LRTC Q}, we adopt the Riemannian gradient descent (RGD), Riemannian conjugate gradient (RCG) methods for all parametrizations.

Given $\vecr^*=(r_1^*,r_2^*,\dots,r_d^*)$, we consider a synthetic low-rank tensor $\tensA$ generated by 
\[\tensA=\tensG^*\times_{k=1}^d\matu_k^*,\]
where the entries of $\tensG^*\in\mathbb{R}^{r_1^*\times\cdots\times r_d^*}$ and $\matu_k^*\in\mathbb{R}^{n_k^{}\times r_k^*}$ are sampled from the standard normal distribution $\mathcal{N}(0,1)$. $\matu_k^*$ is orthogonalized by the QR decomposition. We set $d=3$, $(n_1,n_2,n_3)=(400,400,400)$, sampling rate $p=0.01$, the size of test set $|\Gamma|=pn_1n_2n_3$, and $\vecr^*=(6,6,6)$. The initial guess $x^{(0)}=(\tensG^{(0)},\matu_1^{0},\dots,\matu_d^{0})$ for ``Tucker'' and ``quotient'' parametrization is generated in a same fashion. The initial guess for desingularization is $x^{(0)}=(\tensG^{(0)}\times_{k=1}^d\matu_k^{(0)},\matI_{n_1}-\matu_1^{(0)}(\matu_1^{(0)})^\top,\matI_{n_2}-\matu_2^{(0)}(\matu_2^{(0)})^\top,\dots,\matI_{n_d}-\matu_d^{(0)}(\matu_d^{(0)})^\top)$. A method is terminated if the Riemannian gradient satisfies $\|\grad g(x^{(t)})\|\leq 10^{-13}$ or reaches the maximum iteration $500$. The performance of all methods is evaluated by the training and test errors
\[\varepsilon_{\Omega}^{\tensM}(x):=\frac{\|\proj_\Omega(\varphi(x))-\proj_\Omega(\tensA)\|_\frob}{\|\proj_\Omega(\tensA)\|_\frob}\quad\text{and}\quad\varepsilon_{\Gamma}^{\tensM}(x):=\frac{\|\proj_\Gamma(\varphi(x))-\proj_\Gamma(\tensA)\|_\frob}{\|\proj_\Gamma(\tensA)\|_\frob},\]
where $\Gamma$ is a test set different from the training set $\Omega$ with $|\Gamma|=|\Omega|$. Denote the singular values of the unfolding matrices $\mata_{(k)}$ and $(\varphi(x^{(t)}))_{(k)}$ by $\{\sigma_{i,k}^*\}$ and $\{\sigma_{i,k}^{(t)}\}$ respectively. We also evaluate the performance by the error of singular values
\[\varepsilon(x^{(t)})=\sum_{k=1}^d\sum_{i=1}^{r_k}|\sigma_{i,k}^{(t)}-\sigma_{i,k}^*|.\]
To ensure a fair comparison, Riemannian methods for~\eqref{eq: problem (Q-desing)} are implemented on large matrices and tensors $(\tensX, \matP_1, \dots, \matP_d)$ for the sake of compatibility with Manopt toolbox v7.1.0~\cite{boumal2014manopt}. As a result, we only report the numerical performance in terms of iterations.

\paragraph{Test with unbiased rank}
First, we examine the performance of Tucker-based methods with unbiased rank, i.e., $\vecr=\vecr^*=(6,6,6)$. Figure~\ref{fig: synthetic true rank} reports the training and test error of RGD and RCG methods under all parametrizations. We observe that RGD and RCG methods under all parametrizations successfully recover the underlying low-rank tensor $\tensA$, which is reflected on the error of singular values. Moreover, RGD-desing and RCG-desing methods are comparable to RGD-quotient and RCG-quotient, and perform better than RGD-Tucker and RCG-Tucker methods. 
 
\begin{figure}[htbp]
    \centering
    \includegraphics[width=\textwidth]{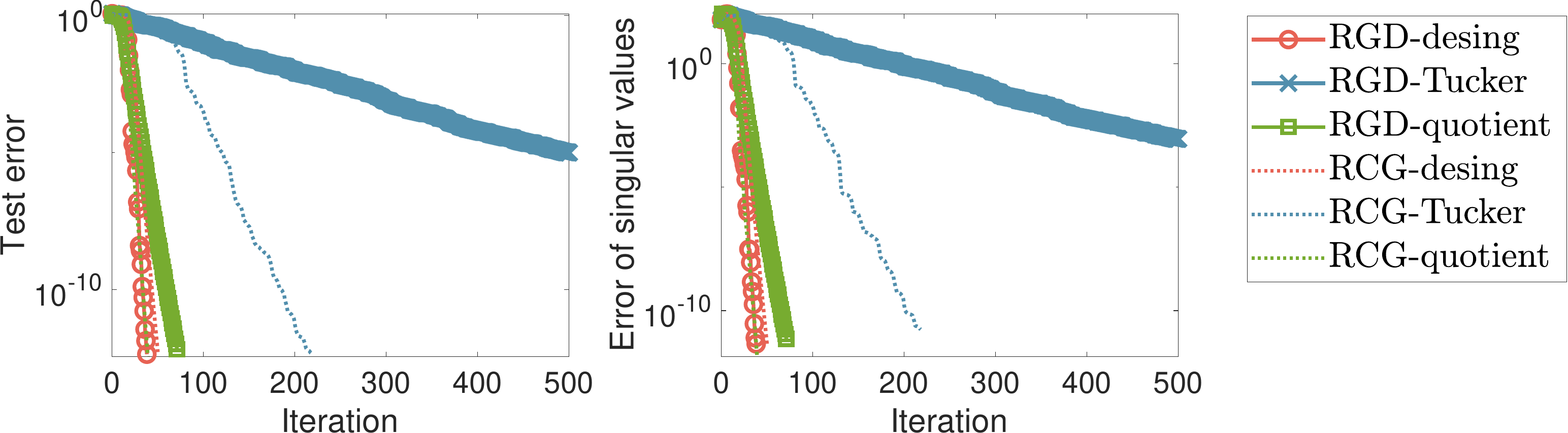}
    \caption{Errors of RGD and RCG under different parametrizations with unbiased rank parameter. Left: test error. Right: errors of singular values}
    \label{fig: synthetic true rank}
\end{figure}

\paragraph{Test with over-estimated rank}
We test the performance of RGD and RCG under over-estimated rank parameter $\vecr=(10,10,10)>\vecr^*$. The numerical results are reported in Fig.~\ref{fig: synthetic over rank}. We observe from Fig.~\ref{fig: synthetic over rank} (left) that RGD-desing and RCG-desing methods successfully recover the data tensor while RGD and RCG methods under other parametrizations fail due to the over-estimated rank parameter, which is reflected on the error of singular values of unfolding matrices in Fig.~\ref{fig: synthetic over rank} (right). Specifically, we observe that the singular values of unfolding matrices of $\varphi(x^{(t)})$ fail to converge to the truth if $\{x^{(t)}\}_{t\geq 0}$ is generated by methods under ``quotient'' and ``Tucker'' parametrizations. On the contrary, for the iterates generated by RGD-desing and RCG-desing, all the singular values of all unfolding matrices of $\varphi(x^{(t)})$ converge to the truth.

\begin{figure}[htbp]
    \centering
    \includegraphics[width=\textwidth]{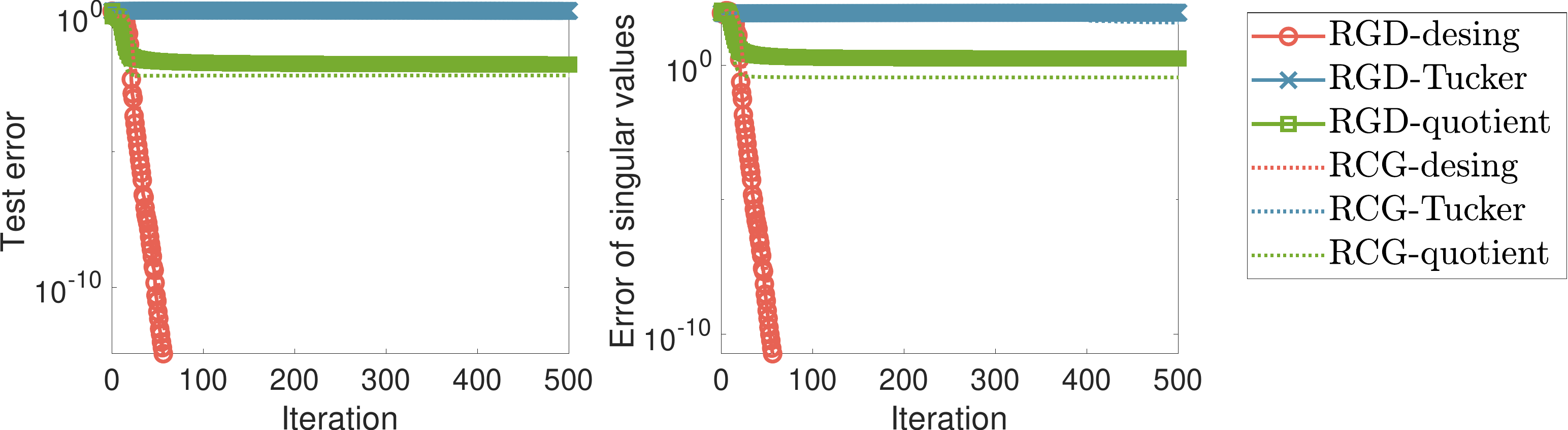}
    \caption{Errors of RGD and RCG under different parametrizations with over-estimated rank parameter. Left: test error. Right: errors of singular values}
    \label{fig: synthetic over rank}
\end{figure}

\subsection{Test on other methods}
Subsequently, we compare the proposed methods with other candidates. In preliminary numerical experiments, we observe that RTR-desing is less efficient than RGD-desing and RCG-desing methods due to the large computational cost in solving the subproblem~\eqref{eq: RTR subproblem}. Therefore, we only compare the RGD-desing and RCG-desing methods with others. Nevertheless, the code for RTR-desing is available for interested readers. We compare the proposed methods with: 1) a Riemannian conjugate gradient method\footnote{GeomCG toolbox: \url{https://www.epfl.ch/labs/anchp/index-html/software/geomcg/}.} (GeomCG)~\cite{kressner2014low}; 2) a Riemannian conjugate gradient method\footnote{Available at: \url{https://bamdevmishra.in/codes/tensorcompletion/}.} on quotient manifold under a preconditioned metric (RCG-quotient)~\cite{kasai2016low} for optimization on fixed-rank manifolds; 3) the gradient-related approximate projection method (GRAP) and 4) the Tucker rank-adaptive method\footnote{Available at: \url{https://github.com/JimmyPeng1998/TRAM}.} (TRAM) for optimization on Tucker tensor varieties~\cite{gao2023low}.

\paragraph{Default settings and implementation details}
Given a specific objective function $f$ and $g$ in problems~\eqref{eq: LRTC P} and~\eqref{eq: LRTC Q}, we provide default settings and more concrete implementation details for Algorithms~\ref{alg: RGD-desing}--\ref{alg: RTR-desing}.

First, We introduce additional implementation details for Algorithms~\ref{alg: RGD-desing} (RGD-desing) and~\ref{alg: RCG-desing} (RCG-desing). We adopt the modified Hestenes--Stiefel rule~\cite{hestenes1952methods} 
\[\beta^{(t)}=\max\left\{0,
\frac{
    \langle\grad f(x^{(t)})-\tensT_{t\gets t-1}\grad f(x^{(t-1)}),
    \grad f(x^{(t)})\rangle_{x^{(t)}}}
    {\langle\grad f(x^{(t)})-\tensT_{t\gets t-1}\grad f(x^{(t-1)}),
    \tensT_{t\gets t-1}\eta^{(t-1)}\rangle_{x^{(t)}}}\right\}\]
for the conjugate parameter $\beta^{(t)}$ in Algorithm~\ref{alg: RCG-desing}, and the projection operator $\proj_{\tangent_{x^{(t)}}\!\tensM}$ for vector transport $\tensT_{t\gets t-1}$. Note that the inner product can be efficiently computed by~\eqref{eq: inner product desing}. For the selection of stepsizes in Algorithms~\ref{alg: RGD-desing} and~\ref{alg: RCG-desing}, similar to the optimization on fixed-rank manifold of Tucker tensors~\cite{kressner2014low} and on Tucker tensor varieties~\cite{gao2023low}, we implement exact line search on tangent space by calculating the stepsize
\begin{equation*}
    \begin{aligned}
        s^{(t)}&=\argmin_{s\geq 0}\|\proj_\Omega(\varphi(x^{(t)}+s\eta^{(t)}))-\proj_\Omega\tensA\|_\frob^2\\
        &=\argmin_{s\geq 0}\|\proj_\Omega(\tensX^{(t)}+s\dot{\tensX}^{(t)})-\proj_\Omega\tensA\|_\frob^2\\
        &=\frac{\langle\proj_\Omega\dot{\tensX}^{(t)},\proj_\Omega(\tensA-\tensX^{(t)})\rangle}{\langle\proj_\Omega\dot{\tensX}^{(t)},\proj_\Omega\dot{\tensX}^{(t)}\rangle}.
    \end{aligned}
\end{equation*}
The computation of $\proj_\Omega\dot{\tensX}^{(t)}$ is implemented in a \texttt{MEX} function.

Second, we introduce the implementation details for RTR-desing in Algorithm~\ref{alg: RTR-desing}. For computing the parameter $\hat{\matu}_k$ in~\eqref{eq: Riemannian Hessian}, we observe that 
\[(\nabla f(\tensX))_{(k)}\!\dot{\matv}_k=\sum_{j\neq k}(\nabla f(\tensX)\times_j\dot{\matu}_j\times_{\ell\neq j,\ell\neq k}\matu_\ell)_{(k)},\]
where $\tensX=\varphi(x)$. Since $\nabla f(\tensX)=\proj_\Omega(\tensX)-\proj_\Omega(\tensA)$ is sparse, all terms $\nabla f(\tensX)\times_j\dot{\matu}_j\times_{\ell\neq j,\ell\neq k}\matu_\ell$ can be efficiently computed. We choose $\rho^\prime=0.1$, $\bar{\Delta}=\sqrt{\dim(\tensM)}$, $\delta^{(0)}=\bar{\Delta}/16$.

The parameters $\tensG^{(0)}$ and $\matu_1^{(0)},\matu_2^{(0)},\dots,\matu_d^{(0)}$ of initial guess of all methods can be generated in a same fashion as section~\ref{subsec: diff params}. For RGD-desing and RCG-desing, we set the initial guess $x^{(0)}=(\tensG^{(0)}\times_{k=1}^d\matu_k^{(0)},\matI_{n_1}-\matu_1^{(0)}(\matu_1^{(0)})^\top,\dots,\matI_{n_d}-\matu_d^{(0)}(\matu_d^{(0)})^\top)$. For other Tucker-based methods, we set the initial guess $\tensX^{(0)}=\tensG^{(0)}\times_{k=1}^d\matu_k^{(0)}$. Note that we never explicitly form $x^{(0)}$ and $\tensX^{(0)}$ with prohibitively large number of parameters.

The performance of all methods is evaluated by the training and test errors 
\[\varepsilon_{\Omega}(\tensX):=\frac{\|\proj_\Omega(\tensX)-\proj_\Omega(\tensA)\|_\frob}{\|\proj_\Omega(\tensA)\|_\frob}\quad\text{and}\quad\varepsilon_{\Gamma}(\tensX):=\frac{\|\proj_\Gamma(\tensX)-\proj_\Gamma(\tensA)\|_\frob}{\|\proj_\Gamma(\tensA)\|_\frob},\]
where $\Gamma$ is a test set different from the training set $\Omega$. For RGD-desing and RCG-desing, the training and test errors are evaluated by $\varepsilon_{\Omega}(\varphi(x))$ and $\varepsilon_{\Gamma}(\varphi(x))$ respectively. We terminate a method if: 1) the training error $\varepsilon_{\Omega}(\tensX^{(t)})<10^{-12}$; 2) the relative change of the training error $(\varepsilon_{\Omega}(\tensX^{(t)})-\varepsilon_{\Omega}(\tensX^{(t-1)}))/\varepsilon_{\Omega}(\tensX^{(t-1)})<10^{-12}$; 3) maximum iteration number is reached; 4) time budget is exceeded.

\paragraph{Test with unbiased rank parameter}
We examine the performance of Tucker-based methods with unbiased rank parameter, i.e., $\vecr=\vecr^*=(6,6,6)$. Figure~\ref{fig: unbiased rank} reports the test errors of Tucker-based methods with sampling rate $p=0.005,0.01,0.05$. We observe that proposed RGD-desing and RCG-desing are comparable to other candidates.

\begin{figure}[htbp]
    \centering
    \includegraphics[width=\textwidth]{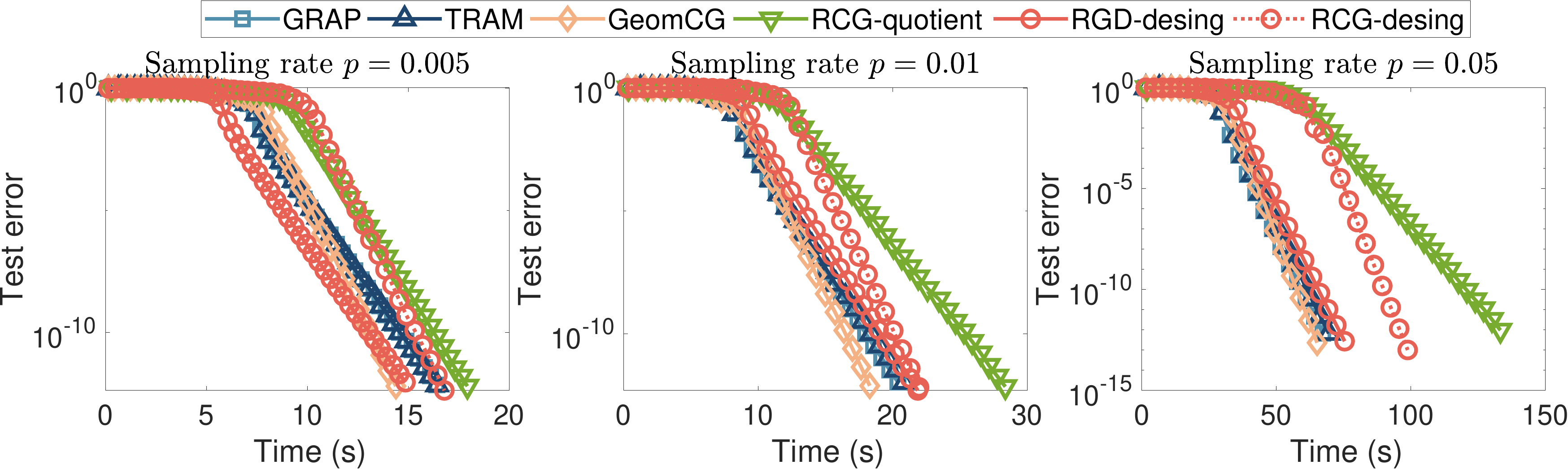}
    \caption{Test on different methods with unbiased rank parameter and sampling rate $p=0.005,0.01,0.05$}
    \label{fig: unbiased rank}
\end{figure}

\paragraph{Test with over-estimated rank parameters}
We examine the performance of all methods with over-estimated rank parameters $\vecr=(7,7,7),(8,8,8),\dots,(12,12,12)>(6,6,6)$. The test errors are reported in Fig.~\ref{fig: biased rank} (left) and average computation time per iteration for all methods is reported in Fig.~\ref{fig: biased rank} (right). We observe that: 1) due to the over-estimated rank parameters, only the proposed RGD-desing, RCG-desing methods and TRAM method successfully recover the low-rank tensor $\tensA$ under all selections of rank parameters; 2) the proposed RGD-desing method is comparable to TRAM, the rank-adaptive method; 3) computational cost per iteration of RGD-desing, RCG-desing methods are comparable to other candidates since we only store parameters.

\begin{figure}[htbp]
    \centering
    \includegraphics[width=\textwidth]{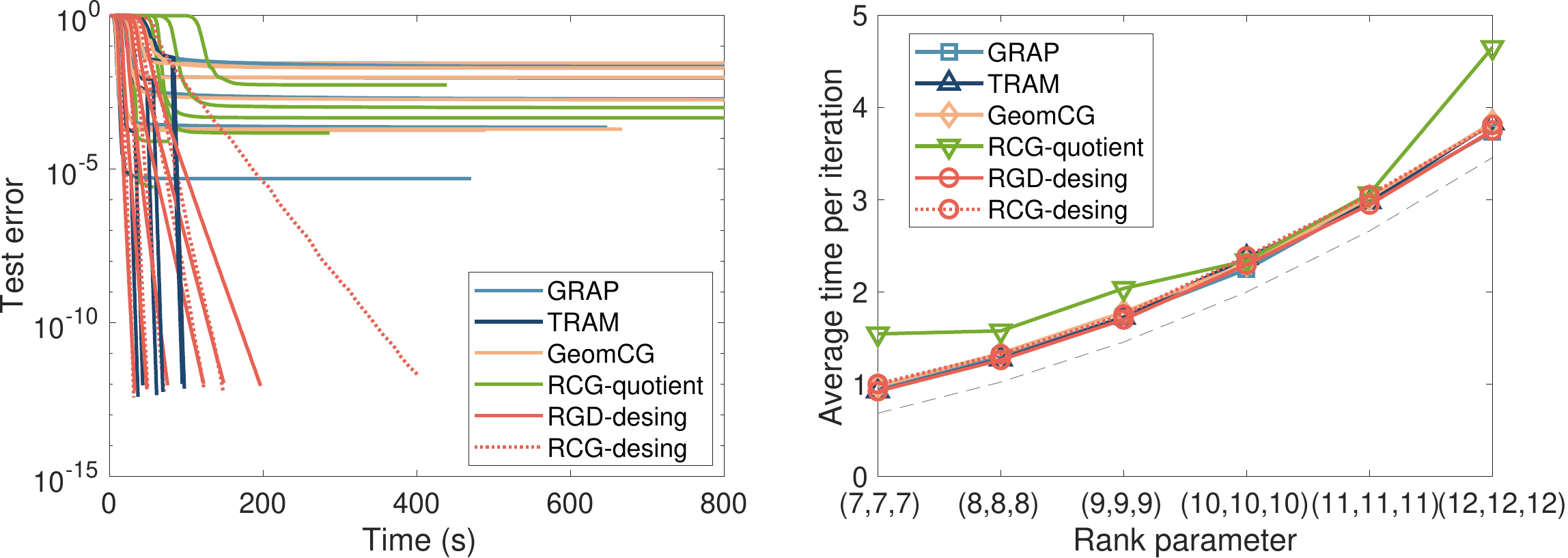}
    \caption{Numerical results on different methods with over-estimated rank parameters $\vecr=(7,7,7),(8,8,8),\dots,(12,12,12)$. Left: test errors. Right: average computation time per iteration. Gray dashed line: $\mathcal{O}(r^3)$}
    \label{fig: biased rank}
\end{figure}

\subsection{Test on movie ratings}
We consider tensor completion on a real-world dataset {``MovieLens 1M"\footnote{Available at \url{https://grouplens.org/datasets/movielens/1m/}.}}, which consists of $1000209$ movie ratings from 6040 users on 3952 movies from September 19th, 1997 to April 22nd, 1998. By choosing one week as a period, we formulate the movie ratings as a third-order tensor $\tensA\in\mathbb{R}^{6040\times 3952\times 150}$. We randomly select $8\times 10^5$ of the known ratings as a training set $\Omega$ and the rest ratings are test set $\Gamma$ to evaluate the performance of a method. The rank parameter is set to be $\vecr=(r,r,r)$ with $r=1,2,\dots,16$. Figure~\ref{fig: ML1M} (left) shows that the proposed RGD-desing method is comparable to GRAP method. We observe from Fig.~\ref{fig: ML1M} (right) that the test errors of TRAM, RGD-desing and GRAP methods are comparable.

\begin{figure}[htbp]
    \centering
    \includegraphics[width=\textwidth]{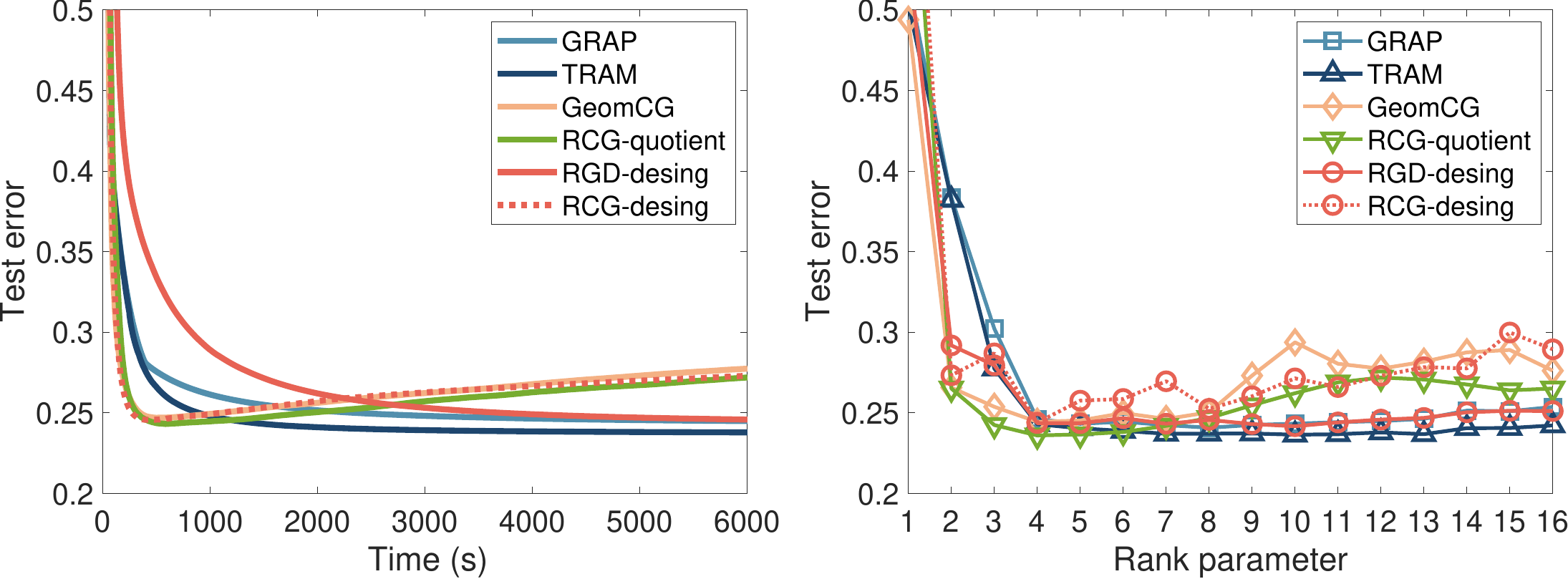}
    \caption{Numerical results on movie ratings. Left: test error with $\vecr=(12,12,12)$. Right: test error with rank parameters $(1,1,1),(2,2,2),\dots,(16,16,16)$}
    \label{fig: ML1M}
\end{figure}

\section{Conclusion and perspectives}\label{sec: conclusion}
In this paper, we have proposed a desingularization approach to minimize a smooth function on tensor varieties, the set of bounded-rank tensors. For Tucker decomposition, we have proposed a new parametrization of Tucker tensor varieties called desingularization by introducing slack variables. We observe that the geometry of desingularization of Tucker tensor varieties is closely connected to matrix and Tucker tensor varieties, but much more intricate than desingularization of the matrix varieties. For tensor train decomposition, since each core tensor in the TT decomposition is interconnected, we proposed a new parametrization by introducing a modified set of slack variables. The proposed parametrizations provide closed and smooth manifolds while preserving the structure of low-rank tensors, and extend the scope of geometries of low-rank tensor sets. Consequently, apart from optimization, these parametrizations may have independent interest in other applications involving low-rank tensors.

In contrast with the Tucker tensor varieties, where the projection onto the tangent cone does not have a closed-form expression~\cite{gao2023low}, the desingularization provides a smooth manifold structure, enabling explicit projection onto tangent spaces even at rank-deficient points. This structural advantage facilitates the application of Riemannian optimization techniques. 

Additionally, we have explored the relationship between the original optimization problem on the Tucker tensor varieties and its parametrized counterpart via desingularization. We have observed that even though a point $x\in\tensM$ is second-order stationary of $f\circ\varphi$, $\varphi(x)$ can be non-stationary of $f$ if $\varphi(x)$ is rank-deficient in all modes. The results extend the scope of bounded-rank matrices to bounded-rank tensors and give a negative answer that the desirable convergence properties in matrix case~\cite{rebjock2024optimization} no longer hold for tensors. A promising direction for future research is to develop methods that are capable of finding stationary points directly on the Tucker tensor varieties.

In practice, we have carefully considered the implementation details for Riemannian optimization via desingularization, avoiding formulation of large-scale matrices and tensors. Numerical experiments on tensor completion suggest that the proposed methods are favorably comparable to existing methods. Notably, the performance of proposed methods are not sensitive to the rank selections---an outcome which is previously achievable only through rank-adaptive strategies.

\section*{Declaration}
The authors declare that the data supporting the findings of this study are available within the paper. The authors have no competing interests to declare that are relevant to the content of this article.

\appendix
\section{Proof of Lemma~\ref{lem: stationary along group action}}\label{app: proof of lemma 4.2}
\begin{proof}
    Denote $\tilde{x}=(\tilde{\tensG},\tilde{\matu}_1,\tilde{\matu}_2,\dots,\tilde{\matu}_d)$, a vector $\dot{x}=(\dot{\tensG},\dot{\matu}_1,\dot{\matu}_2,\dots,\dot{\matu}_d)\in\tangent_{\tilde{x}}\!\tensM^\mathrm{Tucker}\simeq\tensM^\mathrm{Tucker}$, $\tensX=\tilde{\varphi}(\tilde{x})$ and 
    \[\dot{\tensX}=\mathrm{D}\tilde{\varphi}(\tilde{x})[\dot{x}]=\dot{\tensG}\times_{k=1}^d\tilde{\matu}_k+\sum_{k=1}^d\tilde{\tensG}\times_k\dot{\matu}_k\times_{j\neq k}\tilde{\matu}_j.\]
    We calculate the first- and second-order derivatives of $\tilde{g}$ as follows
        \begin{align}
            \partial_{\tilde{\tensG}}\tilde{g}(\tilde{x})&=\nabla f(\tensX)\times_{k=1}^d\tilde{\matu}_k^\top,\label{eq: partial G}\\
            \partial_{\tilde{\matu}_k}\tilde{g}(\tilde{x})&=(\nabla f(\tensX))_{(k)}\tilde{\matv}_k\tilde{\matG}_{(k)}^\top,\label{eq: partial uk}\\
            \mathrm{D}\partial_{\tilde{\matG}}\tilde{g}(\tilde{x})[\dot{x}]&=\nabla^2 f(\tensX)[\dot{\tensX}]\times_{k=1}^d\tilde{\matu}_k^\top+\sum_{k=1}^{d}\nabla f(\tensX)\times_k\dot{\matu}_k^\top\times_{j\neq k}\tilde{\matu}_j^\top,\nonumber\\
            \mathrm{D}\partial_{\tilde{\matu}_k}\tilde{g}(\tilde{x})[\dot{x}]&=(\nabla^2 f(\tensX)[\dot{\tensX}])_{(k)}\tilde{\matv}_k\tilde{\matG}_{(k)}^\top+(\nabla f(\tensX))_{(k)}(\dot{\matv}_k\tilde{\matG}_{(k)}^\top+\tilde{\matv}_k\dot{\matG}_{(k)}^\top)\nonumber\\
            &=(\nabla^2 f(\tensX)[\dot{\tensX}])_{(k)}\tilde{\matv}_k\tilde{\matG}_{(k)}^\top+(\nabla f(\tensX))_{(k)}\tilde{\matv}_k\dot{\matG}_{(k)}^\top\nonumber\\
            &~~~~~~~~~~~~~~~~~~~~~~~~~
            +\sum_{j\neq k}(\nabla f(\tensX)\times_j\dot{\matu}_j^\top\times_{\ell\notin\{j,k\}}\matu_\ell^\top)_{(k)}\tilde{\matG}_{(k)}^\top,\nonumber
        \end{align}
    where $\tilde{\matv}_k=(\tilde{\matu}_j)^{\otimes j\neq k}$.    Given the point $\tilde{x}^\matR=(\tilde{\tensG}\times_{k=1}^d\matR_k^{-1},\tilde{\matu}_1\matR_1,\dots,\tilde{\matu}_d\matR_d)$, we observe that 
    \begin{equation*}
        \begin{aligned}
            \dot{\tensX}&=\dot{\tensG}\times_{k=1}^d\tilde{\matu}_k+\sum_{k=1}^d\tilde{\tensG}\times_k\dot{\matu}_k\times_{j\neq k}\tilde{\matu}_j\\
            &=(\dot{\tensG}\times_{k=1}^d\matR_k^{-1})\times_{k=1}^d(\tilde{\matu}_k\matR_k)+\sum_{k=1}^d(\tilde{\tensG}\times_{k=1}^d\matR_k^{-1})\times_k(\dot{\matu}_k\matR_k)\times_{j\neq k}(\tilde{\matu}_j\matR_j)\\
            &=\mathrm{D}\tilde{\varphi}(\tilde{x}^\matR)[\dot{x}^\matR],
        \end{aligned}
    \end{equation*}
    where $\dot{x}^\matR=(\dot{\tensG}\times_{k=1}^d\matR_k^{-1},\dot{\matu}_1\matR_1,\dots,\dot{\matu}_d\matR_d)$. 
    Since $\tensX=(\tilde{\tensG}\times_{k=1}^d\matR_k^{-1})\times_{k=1}^d(\tilde{\matu}_k\matR_k)$, we can verify the following equalities 
        \begin{align}
            \partial_{\tilde{\tensG}}\tilde{g}(\tilde{x}^\matR)&=\partial_{\tilde{\tensG}}\tilde{g}(\tilde{x})\times_{k=1}^d\matR_k^\top,\label{eq: partial G under group}\\
            \partial_{\tilde{\matu}_k}\tilde{g}(\tilde{x}^\matR)&=\partial_{\tilde{\matu}_k}\tilde{g}(\tilde{x})\matR_k,\label{eq: partial uk under group}\\
            \mathrm{D}\partial_{\tilde{\tensG}}\tilde{g}(\tilde{x}^\matR)[\dot{x}^\matR]&=\mathrm{D}\partial_{\tilde{\tensG}}\tilde{g}(\tilde{x})[\dot{x}]\times_{k=1}^d\matR_k^{\top}.\nonumber
        \end{align}
    For $\mathrm{D}\partial_{\tilde{\matu}_k}\tilde{g}(\tilde{x}^\matR)[\dot{x}^\matR]$, we have 
    \begin{equation*}
        \begin{aligned}
            &~~~~\mathrm{D}\partial_{\tilde{\matu}_k}\tilde{g}(\tilde{x}^\matR)[\dot{x}^\matR]\\
            &=(\nabla^2 f(\tensX)[\dot{\tensX}])_{(k)}\tilde{\matv}_k\tilde{\matG}_{(k)}^\top\matR_k^{-\top}+(\nabla f(\tensX))_{(k)}\tilde{\matv}_k\dot{\matG}_{(k)}^\top\matR_k^{-\top}\\
            &
            ~~~~+\sum_{j\neq k}(\nabla f(\tensX)\times_j(\dot{\matu}_j\matR_j)^\top\times_{\ell\notin\{j,k\}}(\matu_\ell\matR_\ell)^\top)_{(k)}(\matR_j^{-1})^{\otimes j\neq k}\tilde{\matG}_{(k)}^\top\matR_k^{-\top}.\\
            &=\mathrm{D}\partial_{\tilde{\matu}_k}\tilde{g}(\tilde{x})[\dot{x}]\matR_k^{-\top}.
        \end{aligned}
    \end{equation*}
    Therefore, it holds that
        \begin{align}
            \langle\dot{x},\nabla^2 \tilde{g}(\tilde{x})[\dot{x}]\rangle&=\langle\dot{\tensG},\mathrm{D}\partial_{\tilde{\tensG}}\tilde{g}(\tilde{x})[\dot{x}]\rangle+\sum_{k=1}^d\langle\dot{\matu}_k,\mathrm{D}\partial_{\tilde{\matu}_k}\tilde{g}(\tilde{x})[\dot{x}]\rangle\nonumber\\
            &=\langle\dot{\tensG}\times_{k=1}^d\matR_k^{-1},\mathrm{D}\partial_{\tilde{\tensG}}\tilde{g}(\tilde{x}^\matR)[\dot{x}^\matR]\rangle+\sum_{k=1}^d\langle\dot{\matu}_k\matR_k,\mathrm{D}\partial_{\tilde{\matu}_k}\tilde{g}(\tilde{x}^\matR)[\dot{x}^\matR]\rangle\nonumber\\
            &=\langle\dot{x}^\matR,\nabla^2 \tilde{g}(\tilde{x}^\matR)[\dot{x}^\matR]\rangle.\label{eq: equal inner product}
        \end{align}
    Since $\matR_k$ is invertible, we have $\nabla \tilde{g}(\tilde{x}^\matR)=0$ if and only if $\nabla\tilde{g}(\tilde{x})=0$, and $\nabla^2 \tilde{g}(\tilde{x})$ is positive semi-definite if and only if $\nabla^2 \tilde{g}(\tilde{x}^\matR)$ is positive semi-definite. Hence, the claim is ready. \hfill\squareforqed
\end{proof}

\printbibliography



\end{document}